\documentclass[11pt,leqno,oneside,letterpaper]{amsart}
\usepackage[letterpaper,left=1.75truein, top=1truein]{geometry}
\usepackage{amsmath,amsfonts,amssymb,amscd,amsthm,amsbsy,epsf,graphics}
\usepackage{color}
\usepackage[colorlinks,linkcolor=blue,citecolor=blue]{hyperref}
\usepackage{epsfig}
\usepackage{graphicx}
\usepackage[utf8]{inputenc}
\usepackage{tikz}
\usepackage{amsthm,amssymb,amsfonts,amsmath,xfrac, xcolor}
    \usepackage{color,soul}
    \usepackage{amscd} 
    \usepackage{mathrsfs}
    \usepackage[numbers, sort&compress]{natbib}

\usepackage{mathtools}
\usepackage{todonotes}
\newtheorem{theorem}{Theorem}[section]
\newtheorem{lemma}[theorem]{Lemma}
\newtheorem{proposition}[theorem]{Proposition}
\newtheorem{definition}[theorem]{Definition}

\newtheorem{corollary}[theorem]{Corollary}
\theoremstyle{definition} 

\newtheorem{remark}[theorem]{Remark}
\newtheorem{example}[theorem]{Example} 

\newtheorem{question}[theorem]{Question} 

\newtheorem{claim}[theorem]{Claim}

\newtheoremstyle{cases}
  {12pt plus 6 pt}
  {2pt}
  {\bfseries}   
  {}
  {\bfseries}
  {.}
  {.5em}
  {}
\theoremstyle{cases}

\numberwithin{subcase}{case} \numberwithin{subsubcase}{subcase}
\numberwithin{equation}{subsection}
\usepackage{tikz}
\usetikzlibrary{matrix,arrows}
\usepackage[shortlabels]{enumitem}
\title{Generalized torsion in amalgams}

\begin{document}
\author[T.~Cai]{Tommy Wuxing Cai}

\address{Department of Mathematics, University of
	Manitoba, Winnipeg, MB, R3T 2N2}
\email{cait@myumanitoba.ca}

\author[A.~Clay]{Adam Clay}

\address{Department of Mathematics, University of
	Manitoba, Winnipeg, MB, R3T 2N2}
\email{Adam.Clay@umanitoba.ca}

\thanks{Adam Clay was partially supported by NSERC grant RGPIN-05343-2020.}

\begin{abstract} 
We give a condition  sufficient to ensure that an amalgam of groups is generalized torsion-free. As applications, we construct a closed 3-manifold whose fundamental group is generalized torsion-free and non bi-orderable; a one-relator group which is generalized torsion-free and non bi-orderable; and a group which is generalized torsion-free and non left-orderable.
\end{abstract}
\maketitle

\tableofcontents

\section{Introduction}

A \textbf{left-ordering} of a group $G$ is a strict total ordering $<$ of the elements of $G$ which is invariant under left multiplication, meaning $a<b$ implies $ca<cb$ for all $a,b,c\in G$. If a left-ordering is also invariant under right multiplication, we call it a \textbf{bi-ordering}. A group is said to be \textbf{left-orderable} (resp.  \textbf{bi-orderable}) if it admits a left-ordering (resp. bi-ordering).  An element $g \in G$ is a \textbf{generalized torsion element} if $g\neq1$ and there exist $h_1, \ldots, h_n \in G$, $n \geq 1$, such that $g^{h_1}g^{h_2}\dotsm g^{h_n}=1$, where we write $g^h$ for the conjugate $h^{-1}gh$. A group containing no generalized torsion elements is said to be \textbf{generalized torsion-free}, which we will often write as GTF for short.

Left-orderable groups are torsion-free, and similarly, bi-orderable groups are GTF.  For each of these implications, the converse does not hold: There are plenty of non-left-orderable torsion-free groups, for instance the fundamental groups of many $3$-manifolds (see \cite{DMPT05, CW13, BGW13} for examples); and there exist GTF groups which are not bi-orderable (see \cite[Section 4.3]{MR77} and \cite{Bludov72}, for example).  However, it has long been an open question whether or not being GTF is sufficient to guarantee left-orderability, or whether GTF is sufficient to imply bi-orderability in certain special classes of groups.

In this manuscript, we provide a sufficient condition for a free product with amalgamation to be GTF; see Theorem \ref{T:GTFAmalgamIffFamilies}.
For certain kinds of amalgams, our condition is also necessary; see Corollary \ref{C:BOGTFOfExtendableAmalgam}.
Our proofs use a combinatorial argument to provide a lower bound on the length of products of conjugates to rule out the existence of generalized torsion elements that are conjugate into a factor of an amalgam.  To show that no other elements can be generalized torsion, we combine techniques from \cite{CH21} and \cite{IMT19}, which give an upper bound and a lower bound on stable commutator length, respectively.  As applications of Theorem \ref{T:GTFAmalgamIffFamilies}, we can address three open problems concerning the relationship between generalized torsion-freeness
and orderability.

First, there has been a considerable effort in recent years to investigate the generalized torsion elements of fundamental groups of $3$-manifolds \cite{MT17, IMT19, ITM21}.  This investigation is largely motivated by the rich theory surrounding the orderability properties of such groups, which has developed as a result of the L-space conjecture \cite{BGW13}.  Having checked a compelling number of cases, the authors conjecture in \cite{MT17} that the fundamental group of a $3$-manifold is bi-orderable if and only if it is GTF \cite[Conjecture 1.1]{MT17}.  In subsequent work, they resolve \cite[Problem 16.49]{kourovka} by showing that a generalized torsion element in a free product of torsion-free groups must be conjugate into one of the factors.  Since the fundamental group of a $3$-manifold can be expressed as a free product if and only if the manifold is not prime \cite{Stallings71}, 
\cite[Conjecture 1.1]{MT17} holds if and only if it holds for all prime $3$-manifolds. There have since been a number of manuscripts focused on exhibiting generalized torsion elements in fundamental groups of prime $3$-manifolds, whenever their groups are non-bi-orderable \cite{KMT23, Keisuke23, Sekino24a, Sekino24b}.  We provide a counterexample to \cite[Conjecture 1.1]{MT17}.

\begin{theorem}
\label{thm:intro3mfldthm}
There exists a closed, connected $3$-manifold whose fundamental group is GTF and not bi-orderable.
\end{theorem}

The $3$-manifold we construct is a union of two figure eight knot complements glued together along their boundary tori.  It is straightforward to show that the fundamental group resulting from our choice of gluing is not bi-orderable, whereas showing that the fundamental group is GTF relies upon recent orderability results related to an investigation of the L-space conjecture \cite{BGH21}, together with an application of Theorem \ref{T:GTFAmalgamIffFamilies}.  See Section \ref{S:3mflds}.

Second, one-relator groups are known to exhibit exceptional behavior with respect to orderability.  In \cite{Brodskii84, Howie82}, it is shown that every torsion-free one relator group is locally indicable, and so left-orderable (in fact, Conradian left-orderable, which is a strengthening of left-orderability that is weaker than bi-orderability \cite{Conrad59}).  In \cite{CGW15}, the authors investigate bi-orderability of one-relator groups, with the single relator being subject to certain combinatorial conditions, and ask if every one-relator GTF group is bi-orderable \cite[Question 3]{CGW15}. We answer this question in the negative.

\begin{theorem}
\label{thm:onerelatorGTFnonBO}
The following group is GTF and non-bi-orderable:$$G=\langle a,b,d\mid a^{b^2}(a^{b})^{-1}a=(a^{d^2})^{-1}a^{d}a^{-1}\rangle.$$
\end{theorem}

Our construction uses a combination of Theorem \ref{T:GTFAmalgamIffFamilies} to guarantee that the group is GTF, and a technique of \cite{Ak23} to obstruct bi-orderability.  See Section \ref{Section:One-RelatorNon-BO-GTF-Group}.

Third, we are also able to address the question of whether a GTF group must be left-orderable,
 which appears as \cite[Problem 16.48]{kourovka} and \cite[Question 2.1]{BGKM09}. 

\begin{theorem}
\label{thm:introGTFnonLO}
There exists a group which is GTF and not left-orderable.
\end{theorem}

The group we construct in this case is an amalgam of two copies of the free group on two generators, amalgamated along certain finitely generated subgroups. 
Our choices in constructing the amalgam are determined by combinatorial considerations related to the generalized torsion-freeness, specifically, our choices guarantee that the generators of the amalgamating subgroup satisfy very strong ``small cancellation conditions'' that aid us in our analysis of products of conjugates (see Subsection \ref{Subsection:ConstructionOfTheAmalgamAndNonLO}).

\begin{question} Does \cite[Conjecture 1.1]{MT17} hold for all $3$-manifolds whose fundamental group cannot be expressed as an amalgam along a $\mathbb{Z} \oplus \mathbb{Z}$ subgroup?
\end{question}

\begin{question}
     Is there a one-relator group which is GTF, non-biorderable and has only two generators?
\end{question}

\subsection{Organization of the manuscript}
In Section \ref{S:notation}, we introduce our notation for amalgams, elements of amalgams, the notion of ``end-preserving'' products and results involving this notion. In Section \ref{S:TamingProductOfConjugates}, we introduce ``tamed products'' and prove an estimate for their length.  In Section \ref{S:GeneralizedTorsionInAmalgams}, we provide a sufficient condition that an amalgam be GTF, and we use this result to generate new examples of GTF groups. In Section \ref{S:3mflds}, we construct a $3$-manifold group that is GTF but not bi-orderable. In Section \ref{Section:One-RelatorNon-BO-GTF-Group}, we construct a one-relator group which is GTF and not bi-orderable,  and we construct a group which is GTF and not left-orderable in Section \ref{Section:GTFAndNonLOGroups}.

Note that the last three sections are independent of one another, as they are applications of the results in Section \ref{S:GeneralizedTorsionInAmalgams}. 

\section{Notation and preparatory results}

\label{S:notation}
    We refer to Serre's book \cite{Serre1980} for definitions and notations concerning amalgams of groups. Let $G_i$ ($i\in I$) be a collection of groups and $C$ be a common subgroup of all these $G_i$. 
       Then every element $g$ in the amalgam $G=\ast_CG_i$ can be written as $g=cg_1g_2\dotsm g_n$ with $n\geq0$, $c\in C$ and $g_j\in G_{i_j}\backslash C$  satisfying $i_j\neq i_{j+1}$ for $j=1,2,\dotsc,n-1$, and the \textbf{index vector} $(i_1,\dotsc, i_n)$ is uniquely determined by $g$. (The index vector is empty $(\empty)$ when $g$ is in $C$.) We say $g=cg_1g_2\dotsm g_n$ is an \textbf{alternating product}, each $g_i$ a \textbf{component} of $g$ and that the $j$th component of $g$ is from $G_{i_j}$, or that the $j$th component \textbf{has index $i_j$}.
       We define the \textbf{length} of $g$ to be $l(g)=n$. Note $l(g)=0$ (resp. $l(g)=1$) if and only if $g\in C$ (resp. $g\in \cup_{i\in I} G_i\backslash C$). Moreover, when $l(g)\geq1$, we can require the ``$c$" in the alternating product $g=cg_1g_2\dotsm g_n$ to be the unit element $1 \in G$, and thus write $g$ in the form $g=g_1g_2\dotsm g_n$.   
    
    Given a tuple $(g_1,\dotsc, g_n)$ with $g_i \in G$, we say that $(g_1,\dotsc, g_n)$ is \textbf{reduced} if $l(g_1\dotsm g_n)=l(g_1)+\dotsm+l(g_n)$.  When no confusion will arise from doing so, we will simply say that $g_1 \cdots g_n$ is a reduced product (of $n$ terms) and make no reference to the underlying tuple. If $g_1g_2$ is a reduced product, we say that $g_1$ (resp. $g_2$) is a \textbf{left factor} (resp. \textbf{right factor}) of $g_1g_2$. Note that in the case that $l(g_i)\geq1$ for all $i =2, \ldots , n-1$, the product $g_1\dotsm g_n$ is reduced if and only if $l(g_ig_{i+1})=l(g_i)+l(g_{i+1})$ for all $i$ with $1\leq i\leq n-1$.

For $g,h\in G$, write $g$ as an alternating product $g=cx_m\dotsm x_1$ with $c\in C$ and write $h=y_1\dotsm y_nd$, where $y_1\dotsm y_n$ is an alternating product (with $l(y_i)=1$) and $d\in C$. Denote by $K(g,h)$ the maximum $k\geq0$ such that $x_k\dotsm x_1y_1\dotsm y_k\in C$; we call $K(g,h)$ the \textbf{cancellation number} of the product $gh$. Using the uniqueness of index vectors, it can be shown that the cancellation number $K(g,h)$ depends only on $g$ and $h$, and not how the alternating products are written.

\subsection{Some results about products in amalgams}
In this subsection, we prepare basic results concerning cancellation for later use. Let $G=\ast_CG_i$ ($i\in I$) be an amalgam of groups. If $g \in G$ has index vector $(i_1,\ldots, i_n)$ and $n\geq1$, then we call $i_1$ (resp. $i_n$) the left-end (resp. right-end) index of $g$, denote it as $LEI(g)$ (resp. $REI(g)$). Note that $LEI(g)$ and $REI(g)$ are undefined if $g \in C$.

Given a tuple $(g_1, \ldots, g_n)$ where $g_k \in G$ for $k =1, \ldots, n$ with at least one $l(g_k)\geq1$, let $i$ (resp. $j$) be the minimal (resp. maximal) index such that $l(g_i)\geq1$ (resp. $l(g_j)\geq1$). If $LEI(g_1\dotsm g_n)=LEI(g_i)$ (resp. $REI(g_1\dotsm g_n)=REI(g_j)$), we say the tuple $(g_1, \ldots, g_n)$ is \textbf{left end-preserving} (resp. \textbf{right end-preserving}). We say the tuple is \textbf{end-preserving} if it's both left end-preserving and right end-preserving. As in the case of reduced products, we will sometimes make no reference to the underlying tuple, and instead write that $g_1\dotsm g_n$ is left or right end-preserving, when no confusion arises from doing so.  Note that by definition, $l(g_1\dotsm g_n)\geq1$ if the tuple $(g_1, \ldots, g_n)$ is left end-preserving or right end-preserving.
 
For the definitions of reduced products and end-preserving products, we will often bracket a product to indicate the underlying tuple.  As an example, if we say the product $g_1\dotsm g_n$ is reduced, this means that the tuple $(g_1, \ldots, g_n)$ is reduced; whereas if we say that the product $(g_1g_2)g_3\dotsm g_n$ is reduced, we mean that the tuple $(g_1g_2, g_3, \ldots, g_n)$ is reduced.  It is necessary to take such care, as the following example shows: Considering the amalgam $\langle a\mid\rangle\ast\langle b\mid\rangle$, which is the free group of rank two generated by $\{a,b\}$, let $g_1=a,g_2=a^{-1}b,g_3=b^{-1}a$. Then $g_1g_2g_3$ and $g_1(g_2g_3)$ are end-preserving, but $(g_1g_2)g_3$ is not left end-preserving.

\begin{lemma}\label{L:End-preservingAndCancellationNumber}
    Let $\alpha,\beta\in G$ and $k=K(\alpha,\beta)$. 
    \begin{enumerate}[(1)]
        \item If $l(\alpha)>0$, then the product $\alpha\beta$ is left end-preserving if and only if $k<l(\alpha)$;
        \item Assume that $\alpha\beta$ is left end-preserving and $\beta_1$ is a left factor of $\beta$. Also assume that $l(\alpha)>0$ or $l(\beta_1)>0$. Then $\alpha\beta_1$ is left end-preserving. 
    \end{enumerate}
\end{lemma}
\begin{proof}
First, we prove (1). As it is clearly true if $l(\beta)=0$, we assume that $l(\beta)>0$. Write $\alpha$ and $\beta$ as alternating products $\alpha=x_m\dotsm x_1$ and $\beta=y_1\dotsm y_n$; observe that $k\leq m$ and $k\leq n$. Consider two cases.
    \begin{enumerate}[(i)]
\item $k \geq l(\alpha)$. We then have $k=m \leq n$. Now $c_i:=x_i\dotsm x_1y_1\dotsm y_i \in C$ for all $i$ with $1\leq i\leq m$.  If $k=m=n$, then $\alpha\beta=c_m\in C$ and thus $\alpha\beta$ is not left end-preserving.  On the other hand, if $k=m<n$, then $\alpha\beta=c_my_{m+1}\dotsm y_n$.
    Because $y_{m}$ and $y_{m+1}$ have different indices and $x_m$ and $y_m$ have the same index (as $c_{m-1}, c_m=x_mc_{m-1}y_m\in C$), we have $LEI(\alpha\beta)\neq LEI(\alpha)$, and thus the product $\alpha\beta$ is not left end-preserving.

\item $k<l(\alpha)$. Then $\alpha\beta$ is equal to one of the following alternating products
    \begin{align*}
        x_m\dotsm x_{k+2}(x_{k+1}c_{k}y_{k+1})y_{k+2}\dotsm y_n \mbox{ or } x_m\dotsm x_{k+2}(x_{k+1}c_{k})y_{k+1}\dotsm y_{n},
    \end{align*} 
    where the brackets are used to indicate a single component of the alternating product.  This shows that the product $\alpha\beta$ is left end-preserving. 
\end{enumerate}
Now we prove (2). If $l(\alpha)=0$, then $l(\beta_1)>0$ and $\alpha\beta_1$ is (left) end-preserving by definition. Otherwise, it follows from (1) as we have $K(\alpha,\beta_1)\leq K(\alpha,\beta)$.
\end{proof}

We remark there is an analog to Lemma \ref{L:End-preservingAndCancellationNumber} concerning right end-preserving, cancellation numbers and right factors. We only state one of the two versions of this result, as the other version can be deduced easily from the one provided and/or proved similarly. Note that this situation happens several times later, and this occurs quite frequently in Section \ref{Section:GTFAndNonLOGroups}. 
 
The next result, which follows from Lemma \ref{L:End-preservingAndCancellationNumber}, will only be used in the proof of Theorem \ref{T:CIsRTFInA} in Section \ref{Section:GTFAndNonLOGroups}.
\begin{corollary}\label{C:TwosidedCancellationAndSumOfCancellationLengths}
 Let $\alpha_i,g_i\in G$, $n\geq1$. Consider the product $T=g_0\alpha_1g_1\alpha_2g_2\dotsm \alpha_ng_n.$
Assume that for each $i\in\{1,2,\dotsc,n-1\}$, we have 
 \begin{align}\label{I:RTFLeqs+2}
K(g_{i-1,2}'\alpha_i, g_i)+K(g_i, \alpha_{i+1}g_{i+1,1}')\leq l(g_i)-2
\end{align} for every right factor $g_{i-1,2}'$ of $g_{i-1}$ and every left factor $g_{i+1,1}'$ of $g_{i+1}$. Then $T\neq1$.
\end{corollary}

\begin{proof} For each $i\in\{1,2,\dotsc,n-1\}$, let $l_i$ (resp. $r_i$) be the maximum of $K(g_{i-1,2}'\alpha_i, g_i)$ (resp. $K(g_i, \alpha_{i+1}g_{i+1,1}')$) for all right factors $g_{i-1,2}'$ of $g_{i-1}$ (resp. left factors $g_{i+1,1}'$ of $g_{i+1}$), then we have $l_i+r_i\leq l(g_i)-2$ by \eqref{I:RTFLeqs+2}. Now, for each $i\in\{1,2,\dotsc,n-1\}$, we write $g_i=g_{i,1}g_{i,2}$ as a reduced product such that $l(g_{i,1})=l_i+1$, then $l(g_{i,2})\geq r_i+1$. Also, let $g_{0,2}=g_0$ and $g_{n,1}=g_n$. Now by Lemma \ref{L:End-preservingAndCancellationNumber}, $g_{i-1,2}\alpha_ig_{i,1}$ (resp. $g_{i,2}\alpha_{i+1}g_{i+1,1}$) is right (resp. left) end-preserving for all $i\in\{1,2,\dotsc,n-1\}$.  Then the following product is reduced and thus nontrivial
 $$T:=\left(g_0\alpha_1g_{1,1}\right)\left(g_{1,2}\alpha_2g_{2,1}\right)\dotsm \left(g_{n-2,2}\alpha_{n-1}g_{n-1,1}\right)\left(g_{n-1,2}\alpha_{n}g_n\right).$$
\end{proof}

The next result is about the relation between cancellation numbers and right factors.  
\begin{lemma}\label{L:CancellationNumberAndRightFactor}
Assume that $K(\alpha,\beta)=l(\beta)$ and $\gamma$ is a right factor of $\alpha\beta$.  Then there is a right factor $\alpha_2$ of $\alpha$ such that $\gamma=\alpha_2\beta$. 
\end{lemma}
\begin{proof}
    First, there is a reduced product $\alpha=\alpha_1'\alpha_2'$ such that $\alpha_2'\beta=1$. Hence $\alpha\beta=\alpha_1'$. Now $\gamma$ is a right factor of $\alpha_1'$, so we can write $\alpha_1'$ as a reduced product $\gamma_1\gamma$. Letting $\alpha_2=\gamma\alpha_2'$, we have $\alpha=\alpha_1'\alpha_2'=\gamma_1\gamma\alpha_2'=\gamma_1 \alpha_2$, all being reduced products. Thus, $\alpha_2$ is a right factor of $\alpha$ and we have 
    $\gamma=\gamma 1=\gamma \alpha_2'\beta=\alpha_2\beta$, as desired. 
\end{proof}

\section{Taming products of conjugates}\label{S:TamingProductOfConjugates}
To study generalized torsion, we need to investigate products of conjugates, wherein there may be a considerable amount of cancellation. In this section, we introduce an essential technique for controlling this cancellation. 

The setup of this section is as follows.  We let $G=\ast_CG_i$ be an amalgam of groups and  consider the tuple $v=((t_1,g_1),\dots,(t_n,g_n))$, where $n\geq1$, $t_i,g_i\in G$, and write the associated product of conjugates as $T=C_1C_2\dotsm C_n$, where $C_i=t_i^{g_i}$.  For convenience of notation, we set $t_0=t_{n+1}=g_0=g_{n+1}=1$. We make the following definitions to aid in our analysis of products of conjugates.
\begin{definition}\label{D:CancellableFromOneSideOrTwosidedly}
 Given $i$ with $1\leq i\leq n$, we say that $t_i$ is \textbf{cancellable} if at least one of the following is true:
\begin{enumerate}
    \item $t_i$ is cancellable from LHS, meaning that 
    $R_{i-1}g_i^{-1}t_i\in C$ for a right factor $R_{i-1}$ of $g_{i-1}$;
    \item $t_i$ is cancellable from RHS, meaning that 
     $t_ig_iL_{i+1}\in C$ for a left factor $L_{i+1}$ of $g_{i+1}^{-1}$;
    \item $t_i$ is cancellable two-sidedly, meaning that 
    $R_{i-1}t_i^{g_i}L_{i+1}\in C$ for a a right factor $R_{i-1}$ of $g_{i-1}$ and a left factor $L_{i+1}$ of $g_{i+1}^{-1}$.
\end{enumerate} 
\end{definition}

\begin{lemma}\label{L:CancellableFromOneSide}
Fix an $i$ with $1\leq i\leq n$. 
    Assume that $l(t_i)>0$. If $t_i$ is cancellable from LHS (resp. RHS), then 
    $l(g_{i-1}')<l(g_{i-1})$ (resp. $l(g_{i+1}')<l(g_{i+1})$), where $g_{i-1}':=g_{i-1}C_i$ and $g_{i+1}':=g_{i+1}C_i^{-1}$.
\end{lemma}
\begin{proof}
    Assume that $t_i$ is cancellable from LHS. Then we can write $g_{i-1}=L_{i-1}R_{i-1}$ as a reduced product, such that $c_i:=R_{i-1}g_i^{-1}t_i\in C$. 
    Thus, $g_i=t_ic_i^{-1}R_{i-1}$, and then we find $l(g_i)\leq l(t_i)+0+l(R_{i-1})<l(R_{i-1})$, as $l(t_i)>0$. Now we have, as desired,
    \begin{align*}
        l(g_{i-1}')=l(L_{i-1}R_{i-1}g_{i}^{-1}t_ig_i)=l(L_{i-1}c_ig_i)\leq l(L_{i-1})+0+l(g_i)<l(L_{i-1})+l(R_{i-1})=l(g_{i-1}).
    \end{align*}
 
    Similarly, if $t_i$ is cancellable from RHS, then we can write $g_{i+1}^{-1}=L_{i+1}R_{i+1}$ as a reduced product such that $d_i:=t_ig_iL_{i+1}\in C$, which implies that  $l(g_i)<l(L_{i+1})$ and 
    \begin{align*}
        l(g_{i+1}'^{-1})=l(C_ig_{i+1}^{-1})=l(g_i^{-1}t_ig_iL_{i+1}R_{i+1})=l(g_i^{-1}d_iR_{i+1})<l(L_{i+1})+l(R_{i+1})=l(g_{i+1}). 
    \end{align*}
\end{proof}

A key to tame the products of conjugates is to notice that $C_{i-1}C_{i}=C_{i}C_{i-1}'$, where $C_{i-1}'=C_{i-1}^{C_{i}}=t_{i-1}^{g_{i-1}C_{i}}$. This gives a glimpse into the role of the next result, which is to assist in the proof of Theorem \ref{T:TauInSAIfInC}, a key result of Section \ref{S:GeneralizedTorsionInAmalgams}. 
\begin{proposition} \label{CancellabilityMeansShorterConjugateLength}
Fix an $i$ with $1\leq i\leq n$. Let $g_{i-1}'=g_{i-1}C_i$ and $g_{i+1}'=g_{i+1}C_i^{-1}$.
Assume that  $l(t_i)\geq1$ and the product $g_i^{-1}t_ig_i$ is reduced. Also assume that $t_i$ is cancellable,  $l(g_{i-1}')\geq l(g_{i-1})$ and $l(g_{i+1}')\geq l(g_{i+1})$. Then we have 
$l(g_{i-1}')=l(g_{i-1})>l(g_i)$ and
$l(g_{i+1}')=l(g_{i+1})>l(g_i).$
\end{proposition}
\begin{proof}

By Lemma \ref{L:CancellableFromOneSide}, we know that $t_i$ is not cancellable from LHS or RHS,  hence it is cancellable two-sidedly; i.e., we can write $g_{i-1}=L_{i-1}R_{i-1}$ and $g_{i+1}^{-1}=L_{i+1}R_{i+1}$ as reduced products, 
    such that $c_i:=R_{i-1}C_iL_{i+1}\in C$, from which we have 
        \begin{align}\label{E:EqualitiesFromTwosidedCancellation}
        &l(R_{i-1})=l(C_iL_{i+1}), \qquad l(L_{i+1})=l(R_{i-1}C_i),\\
\label{E:InequalitiesFromTwosidedCancellation}
        &l(C_i)\leq l(R_{i-1})+l(L_{i+1}).
    \end{align}
    Here, the inequality follows by considering the length of $C_i=R_{i-1}^{-1}c_iL_{i+1}^{-1}$.
    
    Next, we proceed as follows:
    \begin{enumerate}[(1)]
       \item \label{Item 3:CancellableTwosidedly}We prove that $l(C_iL_{i+1})=l(L_{i+1})$ and $l(R_{i-1}C_i)=l(R_{i-1})$. We only prove the first equality, as the second follows from the first and \eqref{E:EqualitiesFromTwosidedCancellation}. First, suppose that $l(C_iL_{i+1})<l(L_{i+1})$, then we have
    $$l(g_{i+1}'^{-1})=l(C_iL_{i+1}R_{i+1})\leq l(C_iL_{i+1})+l(R_{i+1})<l(L_{i+1})+l(R_{i+1})=l(g_{i+1}),$$ contradicting the assumptions of the lemma. Second, if $l(C_iL_{i+1})>l(L_{i+1})$, then we deduce $l(R_{i-1})>l(R_{i-1}C_i)$ from \eqref{E:EqualitiesFromTwosidedCancellation}. Adding $l(L_{i-1})$ to both sides, we find 
    \begin{align*}
        l(g_{i-1})=l(L_{i-1})+l(R_{i-1})>l(L_{i-1})+l(R_{i-1}C_i)\geq l(L_{i-1}R_{i-1}C_i)=l(g_{i-1}'),
    \end{align*} again a contradiction. This finishes the proof that $l(C_iL_{i+1})=l(L_{i+1})$. 

         \item  We have, using \ref{Item 3:CancellableTwosidedly}, 
         $$l(g_{i-1})=l(L_{i-1})+l(R_{i-1})=l(L_{i-1})+l(R_{i-1}C_i)\geq l(L_{i-1}R_{i-1}C_i)=l(g_{i-1}').$$
         From this and our assumptions, we obtain $l(g_{i-1}')=l(g_{i-1})$. Similarly, $l(g_{i+1}')=l(g_{i+1})$.

        \item \label{Item 4:CancellableTwosidedly}We prove that $l(R_{i-1})=l(L_{i+1})\geq\frac12l(C_i)$. The equality follows from \ref{Item 3:CancellableTwosidedly} and \eqref{E:EqualitiesFromTwosidedCancellation}. This equality and \eqref{E:InequalitiesFromTwosidedCancellation} together imply the inequality. 

        \item Note that $l(C_i)=2l(g_i)+l(t_i)$, applying \ref{Item 4:CancellableTwosidedly} we have $$l(g_{i-1})\geq l(R_{i-1})\geq \frac12l(C_i)= l(g_i)+\frac12l(t_i)>l(g_i),$$ as $l(t_i)\geq1$. Similarly, $l(g_{i+1})>l(g_i)$.      
    \end{enumerate}
\end{proof}

\sloppy Our key result is that the length of the product $t_1^{g_1}\dotsm t_n^{g_n}$ is bounded below by $n$ when $((t_1,g_1), \dots,(t_n,g_n))$
 is tamed (see Definition \ref{D:Tamedness} and Proposition \ref{P:l(T_i)IsAtLeasti}). We first provide two lemmas which build towards this result.

 For each $i$ with $1\leq i\leq n$, we set $r_i=l(g_{i-1}g_{i}^{-1})$, and when $r_i \geq 1$ we write $g_{i-1}g_{i}^{-1}$ as an alternating product $$g_{i-1}g_{i}^{-1}=e_{i,1}\dotsm e_{i,r_i}.$$
\begin{lemma}\label{L:Whentigigi+1-1IsNotReduced}
 Fix an $i$ with $1\leq i\leq n$. Assume that $g_i^{-1}t_ig_i$ is reduced, $l(t_i)=1$ and $t_i$ is not cancellable. Let $L_i$ be a left factor of $g_ig_{i+1}^{-1}$.
   \begin{enumerate}[(1)]
       \item If $l(L_i)\geq1$ and $l(t_ie_{i+1,1})=1$, then the product $t_iL_i$ is not reduced.
       \item Assume that $t_iL_i$ is not reduced. Then $l(g_ig_{i+1}^{-1})\geq1$, $l(t_ie_{i+1,1})=1$, and there is a left factor $L_{i+1}$ of $g_{i+1}^{-1}$ such that $e_{i+1,1}=g_iL_{i+1}$. 
       \item The product $t_iL_i$ is end-preserving.
   \end{enumerate}
\end{lemma}
\begin{proof} \begin{enumerate}[(1)]
   \item This is clear, as $e_{i+1,1}$ is a left factor of $L_i$.
   \item First, we know that $l(g_ig_{i+1}^{-1})\geq1$, as otherwise, $L_i\in C$ thus the product $t_iL_i$ is reduced.  
    
    Next, since the product $t_ig_i$ is reduced, we know that either $g_i\in C$ or the product $g_ig_{i+1}^{-1}$ is not left end-preserving and thus $K(g_i,g_{i+1}^{-1})=l(g_i)$ by Lemma \ref{L:End-preservingAndCancellationNumber}. By Lemma \ref{L:CancellationNumberAndRightFactor}, there is a left factor $L_{i+1}$ of $g_{i+1}^{-1}$ such that $e_{i+1,1}=g_iL_{i+1}$. 
    
    Last, let's show that  $l(t_ie_{i+1,1})=1$. (i) If $l(t_ie_{i+1,1})=0$, then $t_ig_iL_{i+1}=t_ie_{i+1,1}\in C$, and thus $t_i$ is cancellable (from RHS), contradicting the assumption. (ii) If $l(t_ie_{i+1,1})=2$, then $t_i(g_ig_{i+1}^{-1})$ is reduced and thus so is $t_iL_i$, contradicting the assumption. Therefore $l(t_ie_{i+1,1})=1$, as desired.     

     \item If the product $t_iL_i$ is reduced, it is end-preserving. So we assume that $t_iL_i$ is not reduced. Then $l(t_ie_{i+1,1})=1$ by (2) and $l(L_i)\geq1$. Since $L_i$ is a left factor of $g_ig_{i+1}^{-1}=e_{i+1,1}\dotsm e_{i+1,r_{i+1}}$, the right-hand side of which is an alternating product, we know that $L_i=e_{i+1,1}\dotsm e_{i+1,j}c$ for some $j$ with $1\leq j\leq r_{i+1}$ and $c\in C$. Therefore, 
    $$t_iL_i=(t_ie_{i+1,1})e_{i+1,2}\dotsm e_{i+1,j}c,$$
    and thus $t_iL_i$ is end-preserving, as $(t_ie_{i+1,1})e_{i+1,2}\dotsm e_{i+1,j}$ is an alternating product. 
    \end{enumerate}
\end{proof}
Note that the parallel version of Lemma \ref{L:Whentigigi+1-1IsNotReduced} considers a right factor $R_{i-1}$ of $g_{i-1}g_i^{-1}$, with the other conditions unchanged, and the conclusions are: (1) If $l(R_{i-1})\geq1$ and $l(e_{i,r_i}t_i)=1$, then the product $R_{i-1}t_i$ is not reduced; (2) if $R_{i-1}t_i$ is not reduced, then $l(g_{i-1}g_i^{-1})\geq1$, $l(e_{i,r_i}t_i)=1$, and there is a right factor $R_{i-1}'$ of $g_{i-1}$ such that $e_{i,r_i}=R_{i-1}'g_{i}^{-1}$; (3) the product $R_{i-1}t_i$ is end-preserving.

\begin{lemma}\label{L:TamedProductEndPreserving}
Fix an $i$ with $1\leq i\leq n$. Suppose $g_i^{-1}t_ig_i$ is a reduced product, $l(t_i)=1$ and $t_i$ is not cancellable. Let $L_i$ be a left factor of $g_ig_{i+1}^{-1}$ and $R_{i-1}$ a right factor of $g_{i-1}g_i^{-1}$.  Then we have: \looseness=-1
 \begin{enumerate}[(1)]
  \item $t_iL_i, R_{i-1}t_i, R_{i-1}t_iL_i, (R_{i-1}t_i)L_i, R_{i-1}(t_iL_i)$ are end-preserving;
  \item If $l(g_{i}g_{i+1}^{-1})\geq 1$ and $l(t_ie_{i+1,1})=1$, then $l(R_{i-1}t_ie_{i+1,1})=l(R_{i-1}t_i)$.
  \end{enumerate}
\end{lemma}
\begin{proof} 
\begin{enumerate}[(1)]
\item We use Lemma \ref{L:Whentigigi+1-1IsNotReduced} (and its ``right version") and proceed in 3 steps.
\begin{enumerate}
    \item First, $t_iL_i$ and $R_{i-1}t_i$ are end-preserving by Lemma \ref{L:Whentigigi+1-1IsNotReduced}(3).  
    \item We prove that $R_{i-1}(t_iL_i)$ is right end-preserving. 
    By Lemma \ref{L:End-preservingAndCancellationNumber}, it suffices to prove that $K(g_{i-1}g_i^{-1},t_iL_i)=0$. Assuming otherwise, then depending on whether the product $t_iL_i$ is reduced or not, 
\begin{align}\label{E:Ri-1LiInC}
    e_{i,r_i} t_i\in C \text{ or } e_{i,r_i}t_ie_{i+1,1}\in C,
\end{align} respectively,  where $e_{i+1,1}=g_iL_{i+1}$ for some left factor $L_{i+1}$ of $g_{i+1}^{-1}$ by Lemma \ref{L:Whentigigi+1-1IsNotReduced}. Note that $t_iL_i$ is (left) end-preserving and $(g_{i-1}g_i^{-1})(t_iL_i)$ is not reduced, thus the product $(g_{i-1}g_i^{-1})t_i$ is not reduced either. Hence, by Lemma \ref{L:Whentigigi+1-1IsNotReduced}, there is a right factor $R'_{i-1}$ of $g_{i-1}$ such that $e_{i,r_i}=R'_{i-1}g_{i}^{-1}$. Plugging the expressions for $e_{i+1,1}$ and $e_{i,r_i}$ into \eqref{E:Ri-1LiInC}, we find
$R'_{i-1}g_i^{-1}t_i\in C$ or $R'_{i-1}g_i^{-1}t_ig_iL_{i+1}\in C$,
implying that $t_i$ is cancellable, contradicting the assumption. This finishes proving that $R_{i-1}(t_iL_i)$ is right end-preserving.
\item We now wrap up the proof of (1). It follows similarly to (b) that the product $(R_{i-1}t_i)L_i$ is left end-preserving. We have $R_{i-1}t_iL_i$ is right end-preserving because both $R_{i-1}(t_iL_i)$ and $t_iL_i$ are so. Similarly, $R_{i-1}t_iL_i$ is left end-preserving. So $R_{i-1}t_iL_i$ is end-preserving, from which it follows that $(R_{i-1}t_i)L_i$ is  right end-preserving and $R_{i-1}(t_iL_i)$ is left end-preserving.  
\end{enumerate}

\item If $l(R_{i-1})=0$, the equality is clear. Otherwise,  since $e_{i,r_{i}}t_ie_{i+1,1}, t_ie_{i+1,1}\not\in C$ by (1),  $R_{i-1}t_ie_{i+1,1}$ can be written as one of the following alternating products
   \begin{align*}
   &ce_{i,j}\dotsm e_{i,r_{i}-1}(e_{i,r_{i}}t_ie_{i+1,1}),\quad ce_{i,j}\dotsm e_{i,r_{i}-1}e_{i,r_{i}}(t_ie_{i+1,1}),
    \end{align*}where $c\in C$ and $1\leq j\leq r_i$. 
    We can write $R_{i-1}t_i$ similarly, with the only difference being the omission of $e_{i+1,1}$. We see that if $l(e_{i,r_{i}}t_i)=1$ (resp. $l(e_{i,r_{i}}t_i)=2$), then both $l(R_{i-1}t_ie_{i,1})$ and $l(R_{i-1}t_i)$ are equal to $r_{i}-j+1$ (resp. $r_{i}-j+2$).
\end{enumerate}
\end{proof}

The following definition will allow us to control products of conjugates in an amalgam.

\begin{definition}\label{D:Tamedness}
We say the tuple $v=((t_1,g_1),\dots,(t_n,g_n))$ is \textbf{tamed} if the following are satisfied:
\begin{enumerate}
\item For $i=1, \dots, n$, $l(t_i)=1$ and $t_i^{g_i}=g_i^{-1}t_ig_i$ is a reduced product;
\item For $i=1, \dots, n-1$, if $l(g_ig_{i+1}^{-1})=0$, then $l(t_it_{i+1})=2$;
\item No $t_i$ is cancellable.
\end{enumerate}
\end{definition}

\begin{lemma}\label{L:IterativeFormulaForTi}
Assume that $((t_1,g_1), \dots,(t_n,g_n))$ is tamed. Let $T_i=t_1^{g_1}\dotsm t_i^{g_i}$ for $1\leq i\leq n$ and $T_0=1$.  (Also recall that $t_0=g_0=1$.) Then for each $i$ with $1\leq i\leq n$,  
\begin{align}\label{E:T_iReducedProdIteratively}
T_i=(T_{i-1}g_{i-1}^{-1}\delta_i)(\delta_i^{-1}g_{i-1}g_i^{-1}t_i)g_i
\end{align}
is a reduced product (of three terms), where
$$\delta_i:=\begin{cases} e_{i,1}&\quad \text{ if } l(g_{i-1}g_i^{-1})\geq1 \text{ and } l(t_{i-1}e_{i,1})=1,\\
                         1      &\quad \text{ otherwise.}
           \end{cases}$$

\end{lemma}
\begin{proof} We use Lemma \ref{L:TamedProductEndPreserving}
 and make induction on $i$. When $i=1$, $T_1=\delta_1(\delta_1^{-1}g_1^{-1}t_1) g_1$ is clearly a reduced product. Now assume that for $i$ with $1\leq i<n$, \eqref{E:T_iReducedProdIteratively} is a reduced product; we will show that $T_{i+1}$ is also a reduced product of three terms as in \eqref{E:T_iReducedProdIteratively}. Consider
\begin{align*}
T_{i+1}&=T_ig_{i+1}^{-1}t_{i+1}g_{i+1}\\
       &=(T_{i-1}g_{i-1}^{-1}\delta_i)(\delta_i^{-1}g_{i-1}g_i^{-1}t_i)g_ig_{i+1}^{-1}t_{i+1}g_{i+1}\\
       &=(T_{i-1}g_{i-1}^{-1}\delta_i)(\delta_i^{-1}g_{i-1}g_i^{-1}t_i\delta_{i+1})(\delta_{i+1}^{-1}g_ig_{i+1}^{-1}t_{i+1})g_{i+1}.
\end{align*} We will prove that the last expression is a reduced product, then by combining the first two terms, 
$T_{i+1}=(T_ig_i^{-1}\delta_{i+1})(\delta_{i+1}^{-1}g_ig_{i+1}^{-1}t_{i+1})g_{i+1}$ is a reduced product, as desired.

For ease of exposition, write $T_{i+1}=S_1S_2S_3S_4$ and $T_i=S_1S_2'g_i$ with
\[ S_1 = T_{i-1}g_{i-1}^{-1}\delta_i, \, S_2 = \delta_i^{-1}g_{i-1}g_i^{-1}t_i\delta_{i+1}, \, S_3= \delta_{i+1}^{-1}g_ig_{i+1}^{-1}t_{i+1}, \, S_4 = g_{i+1}, \,S_2'=\delta_i^{-1}g_{i-1}g_i^{-1}t_i.
\]
 Since $(\delta_i^{-1}g_{i-1}g_i^{-1})t_i\delta_{i+1}$ and $(\delta_{i+1}^{-1}g_ig_{i+1}^{-1})t_{i+1}$ are end-preserving by \ref{L:TamedProductEndPreserving}(1), we have $l(S_2)\geq1$ and $l(S_3)\geq1$, and thus we only need to show that $S_kS_{k+1}$ is reduced for $k=1,2,3$. Since $S_1S_2'$ is reduced, $ S_2=S_2'\delta_{i+1}$ is (left) end-preserving and $l(S_2')\geq1$ by Lemma \ref{L:TamedProductEndPreserving}(1), $S_1S_2$ is reduced. Since $t_{i+1}g_{i+1}$ is reduced and $(\delta_{i+1}^{-1}g_ig_{i+1}^{-1})t_{i+1}$ is (right) end-preserving by Lemma \ref{L:TamedProductEndPreserving}(1), $S_3S_4$ is reduced.
 
 Now we only need to show that $S_2S_3$ is reduced. Since the product $S_2=(\delta_i^{-1}g_{i-1}g_i^{-1})t_i\delta_{i+1}$ is (right) end-preserving and $S_3=(\delta_{i+1}^{-1}g_ig_{i+1}^{-1})t_{i+1}$ is (left) end-preserving by Lemma \ref{L:TamedProductEndPreserving}(1), we only need to show that
 the product of the rightmost term of $S_2$ which is not in $C$, and the leftmost term of $S_3$ which is not in $C$, is reduced. This is given by analyzing the following possible cases.
\begin{enumerate}
\item $l(g_ig_{i+1}^{-1})=0$. In this case, $\delta_{i+1}=1$, $g_ig_{i+1}^{-1}\in C$ and $l(t_it_{i+1})=2$ by tameness. So $t_it_{i+1}$ is reduced, as desired. 
\item $l(g_ig_{i+1}^{-1})\geq1$ and $l(t_{i}e_{i+1,1})=2$. Then $\delta_{i+1}=1$ and $t_i(g_ig_{i+1}^{-1})$ is reduced, as desired.
\item $l(g_ig_{i+1}^{-1})\geq1$ and $l(t_{i}e_{i+1,1})=1$. In this case, $\delta_{i+1}=e_{i+1,1}$. 
     (i) If $l(g_ig_{i+1}^{-1})\geq2$, then $\delta_{i+1}(e_{i+1}^{-1}g_ig_{i+1}^{-1})$ is reduced, as desired.   (ii) If $l(g_ig_{i+1}^{-1})=1$, then $e_{i+1,1}^{-1}g_ig_{i+1}^{-1}=1$ and we must show that $e_{i+1,1}t_{i+1}$ is reduced.  Since $g_ig_{i+1}^{-1} =e_{i+1,1}$, $e_{i+1,1}$ is a right factor of $g_{i+1}^{-1}$ or a left factor of $g_i$. But $e_{i+1,1}$ can't be a left factor of $g_i$, since $t_ig_i$ is reduced and $t_ie_{i+1,1}$ is not. Now since since $g_{i+1}^{-1}t_{i+1}$ is reduced, so is $e_{i+1,1}t_{i+1}$, as desired.
\end{enumerate}
\end{proof}
The next result plays an indispensable role in the proof of Theorem \ref{T:TauInSAIfInC} of Section \ref{S:GeneralizedTorsionInAmalgams}.
\begin{proposition}\label{P:l(T_i)IsAtLeasti}
If $((t_1,g_1),\dots,(t_n,g_n))$ is tamed, then
$l(t_1^{g_1}\dotsm t_n^{g_n})\geq l(g_1)+n+l(g_n).$
\end{proposition}
\begin{proof} Let $T_i=t_1^{g_1}\dotsm t_i^{g_i}$ for $1\leq i\leq n$. We will show by induction that
\begin{align}\label{E:l(T_i)IsAtLeasti}
l(T_i)\geq l(g_1)+i+l(g_i) \text{ for all }i\in\{1,2,\dotsc,n\}.
\end{align} We proceed with two steps, using the definition of $\delta_i$ from the statement of Lemma \ref{L:IterativeFormulaForTi}.
\begin{enumerate}[(1)]
\item \label{Item:ProveOfTheLengthFormulaForTamedProductStep1} We prove that $l(T_{i}g_{i}^{-1}\delta_{i+1})=l(T_ig_i^{-1})$ for $1\leq i\leq n-1$.

This is true if $\delta_{i+1}=1$. Thus, we consider when $\delta_{i+1}=e_{i+1,1}$ meaning that $l(t_ie_{i+1,1})=1$ and $l(g_ig_{i+1}^{-1})\geq1$. Consider the products
\begin{align*}
T_ig_{i}^{-1}=(T_{i-1}g_{i-1}^{-1}\delta_i)(\delta_i^{-1}g_{i-1}g_i^{-1}t_i), \quad T_ig_{i}^{-1}\delta_{i+1}=(T_{i-1}g_{i-1}^{-1}\delta_i)(\delta_i^{-1}g_{i-1}g_i^{-1}t_ie_{i+1,1}).
\end{align*} The first product is reduced by Lemma \ref{L:IterativeFormulaForTi}, and so is the second product because $(\delta_i^{-1}g_{i-1}g_i^{-1}t_i)e_{i+1,1}$ is (left) end-preserving and $l(\delta_i^{-1}g_{i-1}g_i^{-1}t_i)\geq1$ by Lemma \ref{L:TamedProductEndPreserving}(1). 
Thus it suffices to prove that $l(\delta_i^{-1}g_{i-1}g_i^{-1}t_i) = l(\delta_i^{-1}g_{i-1}g_i^{-1}t_ie_{i+1,1})$, 
which does hold by Lemma \ref{L:TamedProductEndPreserving}(2).

 \item Now we prove \eqref{E:l(T_i)IsAtLeasti} by induction. It is true when $i=1$ since $t_1^{g_1}$ is reduced. Assume that it is true for $T_{i-1}$ and consider $T_i$. We have:
\begin{align*}
l(T_i)&=l(T_{i-1}g_{i-1}^{-1}\delta_i)+l(\delta_i^{-1}g_{i-1}g_i^{-1}t_i)+l(g_i)\quad  \text{(by Lemma \ref{L:IterativeFormulaForTi})}\\
      &\geq l(T_{i-1}g_{i-1}^{-1})+1+l(g_i)\quad\text{ (by \ref{Item:ProveOfTheLengthFormulaForTamedProductStep1} and Lemma \ref{L:TamedProductEndPreserving}(1))}\\
      &\geq l(T_{i-1})-l(g_{i-1})+1+l(g_i)\\
      &\geq (l(g_1)+(i-1)+l(g_{i-1}))-l(g_{i-1})+1+l(g_i)\quad \text{(by induction assumption})\\
      &=l(g_1)+i+l(g_i).
\end{align*}
\end{enumerate}
\end{proof}

\section{Generalized torsion in amalgams}\label{S:GeneralizedTorsionInAmalgams}
The main goal of this section is to derive Theorem \ref{T:GTFAmalgamIffFamilies}, which gives a sufficient condition ensuring that an amalgam of groups is GTF. In Subsection \ref{Subsection:BergmanAnalog}, we derive an analog of a theorem of George Bergman \cite[Theorems 21 and 28]{Bergman90} from Theorem \ref{T:GTFAmalgamIffFamilies}. In Subsection \ref{Subsec:Multi-malnormalityAndTwoCorollaries}, we derive from Theorem \ref{T:GTFAmalgamIffFamilies} and its proof two corollaries, which will be useful in Section \ref{Section:One-RelatorNon-BO-GTF-Group} and Section \ref{Section:GTFAndNonLOGroups}.

For a subset $S$ of a group $G$, we denote by $\mathrm{NSS}_G(S)$ the normal subsemigroup of $G$ generated by $S$; i.e., $\mathrm{NSS}_G(S):=\{s_1^{g_1}\dotsm s_n^{g_n}\mid s_i\in S, g_i\in G,n\geq1\}$. Note that $a\in G\backslash\{1\}$ is generalized torsion if and only if $1\in \mathrm{NSS}_G(\{a\}$). Also note that we allow a (normal) subsemigroup to be empty.

\begin{theorem}\label{T:TauInSAIfInC}
Consider the amalgam $G=\ast_CG_i$ $(i\in I)$.  Let $P_i$ be a normal subsemigroup of $G_i$ for each $i\in I$, such that $P_i\cap C=P_j\cap C$ for all $i,j\in I$. Then $G_i\cap \mathrm{NSS}_G(\cup_{i\in I}P_i)=P_i$ for all $i\in I$. \looseness=-1
\end{theorem}
\begin{proof} Fix an $i_0\in I$ and let $S=\cup_{i\in I} P_i$. We want to prove that $G_{i_0}\cap \mathrm{NSS}_G(S) = P_{i_0}$. Since the inclusion $\supseteq$ is clear, we only need to prove the other inclusion.

Let $\mathcal S$  be the set of tuples $((s_1,h_1),\dotsc,(s_k,h_k))$ such that $k\geq1$, $s_i\in S$, $h_i\in G$ and $s_1^{h_1}\dotsm s_k^{h_k}\in G_{i_0} \backslash P_{i_0}$. Suppose $\mathcal S$ is non-empty; we will find a contradiction. This will complete the proof. 

For $w=((s_1,h_1),\dotsc,(s_k,h_k))$ in $\mathcal S$, we associate three nonnegative integers to $w$: $NC(w)=k$, $L(w)=\sum_{i=1}^kl(h_i)$, and $N(w)$, which is the number of pairs $(i,j)$ with $i,j \in \{1, \ldots, k\}$ such that $i<j$ and $l(h_i)<l(h_j)$. 

Choose $v_0=((t_1,g_1),\dotsc,(t_n,g_n))\in \mathcal S$ such that $(NC(v),L(v),N(v))$ is minimal, where the set $\mathbb N^3$ is given the lexicographic order: $(a_1,a_2,a_3)<(b_1,b_2,b_3)$ if $a_1<b_1$ or $a_1=b_1$ and $a_2<b_2$ or $a_1=b_1,a_2=b_2$ and $a_3<b_3$. We will finish the proof by showing that $v_0$ must be tamed. Then by Proposition \ref{P:l(T_i)IsAtLeasti}, $l(t_1^{g_1} \cdots t_n^{g_n})\geq n$ and thus $n=1$ as $t_1^{g_1} \cdots t_n^{g_n}\in G_{i_0}$.  But then the product $t_1^{g_1} \cdots t_n^{g_n}$ reduces to $t_1^{g_1}$ with $l(t_1^{g_1})=1$. As $(t_1,g_1)$ is tamed, the product $t_1^{g_1}$ is reduced and $l(t_1)=1$, we have $t_1 \in G_{i_0}\backslash C$ and $g_1 \in C$, recalling that $t_1^{g_1} \in G_{i_0}$.  But this leads to $t_1 \in (G_{i_0}\backslash C)\cap S=P_{i_0}\backslash C$ and then $t_1^{g_1} \in P_{i_0}$, a contradiction as desired.

We use the minimality of $v_0$ and $S\cap C=P_i\cap C$ ($i\in I$) to check that $v_0$ is tamed as follows.
 \begin{enumerate}[(1)]

\item \label{Item:ProofOfTheorem{T:TauInSAIfInC}}We first check that $l(t_i) = 1$ for all $i=1, \ldots, n$.  We only need to show that $t_i \notin C$, since $t_i \in S$ already implies $l(t_i) \leq 1$.  Assume that $t_i$ is in $C$, then it is in $S\cap C=P_j\cap C$ for all $j\in I$.  Consider two possible cases.
\begin{enumerate}
    \item \label{Item:lgigeq1TauInSAIfInC} $l(g_i) \geq 1$. In this case, we can write $g_i = c_i x_1 \cdots x_m$ as an alternating product with $m \geq 1$.  Set $t_i' = t_i^{c_ix_1}$ and $g_i' = x_2 \cdots x_m$.  Then observe that $t_i' \in S$, $g_i' \in G$ and $$G_{i_0}\backslash P_{i_0}\ni t_1^{g_1} \cdots t_n^{g_n}= t_1^{g_1} \cdots t_{i-1}^{g_{i-1}} (t_i')^{g_i'}t_{i+1}^{g_{i+1}} \cdots t_n^{g_n}.$$
 So $w:=((t_1,g_1),\dotsc,(t_{i-1},g_{i-1}), (t_i',g_i'),(t_{i+1},g_{i+1}),\dotsc,(t_n,g_n))$ is in $\mathcal{S}$. But $w$ satisfies $NC(w)=NC(v_0)$ and $L(w)<L(v_0)$, contradicting the minimality of $v_0$.
\item \label{Item:lgieq1TauInSAIfInC}$l(g_i) =0$. This means $g_i \in C$ and then we have $c_i:=t_i^{g_i} \in P_{i_0}\cap C$, as $t_i\in P_{i_0}\cap C$. So we have $n\geq2$. Letting $g_j' = g_jc_i^{-1}$ for $j=i+1, \ldots, n$, we have
\[ T:=t_1^{g_1} \cdots t_n^{g_n} =t_1^{g_1} \cdots t_{i-1}^{g_{i-1}} c_it_{i+1}^{g_{i+1}} \cdots t_n^{g_n}c_i^{-1} c_i = t_1^{g_1} \cdots t_{i-1}^{g_{i-1}} t_{i+1}^{g_{i+1}'} \cdots t_n^{g_n'}c_i.
\]
 Now, $c_i \in C \subseteq G_{i_0}$ and by assumption $T \in G_{i_0}$, so $T':=t_1^{g_1} \cdots t_{i-1}^{g_{i-1}} t_{i+1}^{g_{i+1}'} \cdots t_n^{g_n'} \in G_{i_0}$.  Moreover, $T' \not\in P_{i_0}$, because otherwise, since $c_i\in P_{i_0}$, we have $T=T'c_i\in P_{i_0}$, a contradiction.  Letting 
$w=((t_1,g_1),\dotsc,(t_{i-1},g_{i-1}),(t_{i+1},g_{i+1}'),\dotsc,(t_n,g_n'))$, we can now conclude that $w$ is in $\mathcal{S}$, 
a contradiction since $NC(w) < NC(v_0)$.
\end{enumerate}
 \item Then we show that each $t_i^{g_i}$ is a reduced product.  Fix $i$ and suppose $t_i^{g_i}$ is not reduced.  Then, writing $g_i = c_i x_1 \cdots x_m$ as an alternating product, there is a $j\in I$ such that $t_i, x_1 \in G_j\backslash C$. We set $t_i' = t_i^{c_ix_1}$ and $g_i' = x_2 \cdots x_m$ and argue as in the previous case \ref{Item:lgigeq1TauInSAIfInC} that 
 $$w:=((t_1,g_1),\dotsc,(t_{i-1},g_{i-1}), (t_i',g_i'),(t_{i+1},g_{i+1}),\dotsc,(t_n,g_n)) \in \mathcal{S},$$
which contradicts the minimality of $v_0$.

\item  \label{Item:ltiti+1T:TauInSAIfInC}Next, we assume that $c_i:=g_ig_{i+1}^{-1}\in C$, and we will show $l(t_it_{i+1})=2$.
Note that $t_i^{g_i}t_{i+1}^{g_{i+1}}=(t_it_{i+1}^{c_i^{-1}})^{g_i}={t_i'}^{g_i}$, where $t_i':=t_it_{i+1}^{c_i^{-1}}$.  If $l(t_it_{i+1})\leq 1$, then $t_i$ and $t_{i+1}$ are in the same set $P_j \setminus C$ for a $j\in I$, implying $t_i'\in S$, which leads to the existence of a tuple $w \in \mathcal S$ with $NC(w)<NC(v_0)$, a contradiction as desired.

\item \label{Item:Non-cancellableProofOfT:TauInSAIfInC}Last, we show that no $t_i$ is cancellable. Assume that $t_i$ is cancellable for a fixed $i$ with $1\leq i\leq n$. We will find a contradiction using Proposition \ref{CancellabilityMeansShorterConjugateLength} and the notations there. Letting $C_j'=t_j^{g_j'}$ for $j=i-1,i+1$, we have $C_{i-1}C_i=C_iC_{i-1}'$ and $C_iC_{i+1}=C_{i+1}'C_i$. 
Using these identities, we rewrite the product $C_1\dotsm C_n$ and consider the corresponding tuples
\begin{align*}
    u_i&:=((t_1,g_1),\ldots,(t_i,g_i),(t_{i-1},g_{i-1}'),(t_{i+1},g_{i+1}),\ldots,(t_n,g_n)),\\
    w_i&:=((t_1,g_1),\ldots,(t_{i-1},g_{i-1}),(t_{i+1},g_{i+1}'),(t_i,g_i),\ldots,(t_n,g_n)).
 \end{align*} 
If $l(g_{i-1}')<l(g_{i-1})$ (resp. $l(g_{i+1}')<l(g_{i+1})$), then $i>1$, $u_i\in \mathcal S$ and $L(u_i)<L(v_0)$ (resp. $i<n$, $w_i\in\mathcal S$ and $L(w_i)<L(v_0)$), contradicting the minimality of $v_0$. 
Therefore $l(g_{i-1}')\geq l(g_{i-1})$ and $l(g_{i+1}')\geq l(g_{i+1})$. Then $i<n$, since otherwise $t_n$ is cancelable from LHS (as $g_{n+1}=1)$, implying that $l(g_{i-1}')< l(g_{i-1})$ by Lemma \ref{L:CancellableFromOneSide}, a contradiction. Now $w_i\in\mathcal S$, and $l(g_{i+1}')=l(g_{i+1})>l(g_i)$ by Proposition \ref{CancellabilityMeansShorterConjugateLength}, implying that $L(w_i)=L(v_0)$ while $N(w_i)<N(v_0)$, a contradiction.
 \end{enumerate}
\end{proof}

\begin{proposition}\label{P:SASBN}
Consider the amalgam $G= \ast_CG_i$ ($i\in I)$. Fix an $i_0\in I$ and an $R\subseteq G_{i_0}$.  Then there exist a normal subsemigroup $P_i$ of $G_i$ for each $i\in I$, such that
\begin{align*}
    (1)~ R \subseteq P_{i_0}, \quad (2)~ P_{j} \cap C = P_{j'} \cap C, \forall j,j'\in I ,\quad (3) ~P_{i_0} \subseteq \mathrm{NSS}_G(R).
\end{align*}
\end{proposition}

\begin{proof}
 We first define sets $S_j^{(i)}$ for all $i\in\mathbb N$ and $j\in I$. Let $S_{i_0}^{(0)}=\mathrm{NSS}_{G_{i_0}}(R)$ and then for $i=0,1,2,\dotsc$, we iteratively define 
$$S_j^{(i)}=\mathrm{NSS}_{G_j}(S_{i_0}^{(i)}\cap C) \text{ for each } j\in I\backslash\{i_0\}\text{ and } S_{i_0}^{(i+1)}=\mathrm{NSS}_{G_{i_0}}\left(R\cup\big(\cup_{j\neq i_0}(S_j^{(i)}\cap C)\big)\right).$$
 Now take  $P_j=\cup_{i\geq0}S_j^{(i)}$ for each $j\in I$.  We check that these sets have the required properties.

 First, for every $j\in I\backslash\{i_0\}$ and $i\geq0$, we have $S_{i_0}^{(i)}\cap C\subseteq S_{j}^{(i)}\cap C\subseteq S_{i_0}^{i+1}\cap C\subseteq S_{j}^{i+1}\cap C$. 
 Therefore, for all $j\in I$, we have $S_j^{(0)}\cap C\subseteq S_j^{(1)}\cap C\subseteq S_j^{(2)}\cap C\subseteq\dotsm$, which leads to $S_j^{(0)}\subseteq S_j^{(1)}\subseteq S_j^{(2)}\subseteq\dotsm$, implying that  $P_j$ is a normal subsemigroup of $G_j$. 

 The inclusion (1) is true by definition. To prove (2), we only need to show $P_{i_0}\cap C=P_j\cap C$ for all $j\in I$, which holds because for $j\in I\backslash\{i_0\}$ we have
 \begin{align*}
     \bigcup_{i\geq0}(S_{i_0}^{(i)}\cap C)\subseteq \bigcup_{i\geq0}(S_{j}^{(i)}\cap C)\subseteq \bigcup_{i\geq0}(S_{i_0}^{(i+1)}\cap C)=\bigcup_{i\geq1}(S_{i_0}^{(i)}\cap C)\subseteq \bigcup_{i\geq0}(S_{i_0}^{(i)}\cap C),
 \end{align*} where the first and the last sets are $P_{i_0}\cap C$ and the second is $P_j\cap C$.

 To prove (3) it suffices to prove that $S_j^{(i)}\subseteq N:=\mathrm{NSS}_G(R)$ for all $i\geq0$ and $j\in I$, which can be shown by induction on $i$. First, $S_{i_0}^{(0)}\subseteq N$ and then $S_{j}^{(0)}\subseteq N$ for $j\neq i_0$. Second, fix an $i$ and assume that $S_j^{(i)}\subseteq N$ for all $j\in I$, then we have, using also $R\subseteq N$,
 \begin{align*}
     &S_{i_0}^{(i+1)}=\mathrm{NSS}_{G_{i_0}}\left(R\cup\left(\cup_{j\neq i_0}(S_j^{(i)}\cap C)\right)\right)\subseteq \mathrm{NSS}_{G_{i_0}}\left(N\cup\big(\cup_{j\neq i_0}(N\cap C)\big)\right)
     \subseteq \mathrm{NSS}_{G}(N)=N,
 \end{align*} and thus for $j\neq i_0$, we have $S_j^{(i+1)}=\mathrm{NSS}_{G_j}(S_{i_0}^{(i+1)}\cap C)\subseteq\mathrm{NSS}_{G}(N\cap C)\subseteq N$. This finishes the proof of (3).
\end{proof}

The next result follows immediately by combining Theorem \ref{T:TauInSAIfInC} and Proposition \ref{P:SASBN}. It gives a necessary and sufficient condition which guarantees that an element of  $G_{i_0} \subseteq \ast_CG_i$ is a not a generalized torsion element of $G=\ast_CG_i$.

\begin{theorem}\label{T:NecessaryAndSufficient4SubsetGTF}
Consider the amalgam $G= \ast_iG_i$ ($i\in I)$. Fix an $i_0\in I$ and let $R\subseteq G_{i_0}\backslash\{1\}$. Then $1\not\in \mathrm{NSS}_G(R)$ if and only if there exist a normal subsemigroup $P_j$ of $G_j$ for each $j\in I$ such that $R \subseteq P_{i_0}$, $1 \notin P_{i_0}$, and $P_{j}\cap C=P_{j'}\cap C$ for all $j,j'\in I$. 
\end{theorem}
\begin{proof}
Suppose there are such $P_j$ ($j\in I$), then by Theorem \ref{T:TauInSAIfInC}, we have $G_{i_0}\cap \mathrm{NSS}_G(\cup_{j\in I}P_j)=P_{i_0}.$ Since $1\in G_{i_0}$ and $1\not\in P_{i_0}$, we have $1\not\in \mathrm{NSS}_G(\cup_{j\in I}P_j)$. Therefore $1\not\in \mathrm{NSS}_G(R)$, as $R\subseteq P_{i_0}$.

On the other hand, assume that $1\not\in \mathrm{NSS}_G(R)$. Then apply Proposition \ref{P:SASBN}, and note that $P_{i_0} \subseteq \mathrm{NSS}_G(R)$ guarantees that $1 \notin P_{i_0}$.
\end{proof}

While the previous theorem allows us to rule out generalized torsion in $G$ arising from elements of $\cup_{j\in I} G_j$, we need the following result to deal with other elements in the amalgam.

\begin{definition}(\cite{CH21})\label{D:RTF}
Let $H$ be a subgroup of a group $G$.  The subgroup $H$ is called \textbf{relatively torsion-free}, or \textbf{RTF} for short, if for all $g \in G \setminus H$ and all $h_1, \ldots, h_k \in H$ ($k \geq 1$), we have
\[ gh_1gh_2 \cdots gh_k \neq 1.
\]
\end{definition}

In the proof below, we use $\mathrm{scl}(g)$ to denote the stable commutator length of $g$. (See \cite{Ca2009} for background on stable commutator length.)

\begin{proposition}
\label{prop:hypelements}
Consider the amalgam $G= \ast_iG_i$ ($i\in I)$. Assume that $g\in G$ is not conjugate into $\cup_{i\in I}G_i$.  If $C$ is RTF in $G_i$ for all $i\in I$, then $g$ is not generalized torsion.    
\end{proposition}
\begin{proof}
The usual construction of the Bass-Serre tree associated with the amalgam $\ast_C G_i$ provides a tree $\Gamma$ such that the vertex stabilizers of the $G$-action on $\Gamma$ are precisely the conjugate subgroups $h^{-1}G_ih$, $h \in G, i\in I$.  Thus, if $g\in G$ is not conjugate into $\cup_{i\in I}G_i$, then $g$ stabilizes no vertices of $\Gamma$.  Consequently, if $C$ is RTF in all $G_i$ ($i\in I$) then $\mathrm{scl}(g) \geq \frac{1}{2}$ by \cite[Theorem A]{CH21}.  On the other hand, generalized torsion elements of $G$ have $\mathrm{scl}$ strictly less than $\frac{1}{2}$ by \cite[Theorem 2.4]{IMT19}, so $g$ cannot be generalized torsion.
\end{proof}

Combining Theorem \ref{T:NecessaryAndSufficient4SubsetGTF} and Proposition \ref{prop:hypelements}, we can give a sufficient condition that the amalgam
contains no generalized torsion. It is worth noting the similarity between the conditions in the theorem below and those in \cite[Theorem A]{BG09}.

\begin{theorem}\label{T:GTFAmalgamIffFamilies}
Let $G=\ast_C G_i$ be the free product of groups $G_i$ ($i\in I$), amalgamated by a common subgroup $C$. Then no element of $\cup_{i\in I}G_i$ is generalized torsion if and only if there is a family $\mathcal F_i$ of normal subsemigroups of $G_i$ for each $i\in I$, satisfying:
 \begin{enumerate}
 \item $G_i\backslash\{1\}=\bigcup_{P\in\mathcal F_i}P$ for every $i\in I$;
 \item For all $i, j\in I$ and each $P\in\mathcal F_i$, there is a $Q\in\mathcal F_j$ such that $P\cap C=Q\cap C$.
 \end{enumerate} 
If moreover $C$ is RTF in $G_i$ for all $i\in I$, then $G$ is GTF. 
\end{theorem}
\begin{proof}
First, assume the existence of the families of normal subsemigroups. Fix an $i_0\in I$ and an $a\in G_{i_0} \backslash\{1\}$; we only need to prove that $1\not\in \mathrm{NSS}_G(\{a\})$. There exists a $P_{j_0}\in \mathcal F_{i_0}$ such that $a\in P_{j_0}$, and then for every $j\in I\backslash\{i_0\}$, there is a $P_{j}\in \mathcal F_j$ such that $P_{i_0}\cap C=P_{j}\cap C$. Thus $P_j\cap C=P_{j'}\cap C$ for all $j,j'\in I$. Applying Theorem \ref{T:NecessaryAndSufficient4SubsetGTF} with $R = \{ a \}$, we get that $1\not\in \mathrm{NSS}_G(\{a\})$, as desired.

For the converse, we assume that $\cup_{i\in I}G_i$ contains no generalized torsion elements of $G$. Then for every $i\in I$ and $a\in G_i\backslash\{1\}$, we have $1\not\in\mathrm{NSS}_G(\{a\})$, and thus by Theorem \ref{T:NecessaryAndSufficient4SubsetGTF}, there is a normal subsemigroup $P_{i,a,j}$ of $G_j$ for each $j\in I$,  such that $a\in P_{i,a,i}$, $1\not\in P_{i,a,i}$ and $P_{i,a,j}\cap C=P_{i,a,j'}\cap C$ for all $j,j'\in I$. Now let
$\mathcal F_j=\{ P_{i,a,j}\mid i\in I, a \in G_i \setminus \{1\} \}$ for each $j\in I$. Then the families $\mathcal F_j$ ($j\in I$) satisfy the required properties.

The last claim of the theorem now follows from Proposition \ref{prop:hypelements}.
\end{proof}

\subsection{Examples, and an analog of a theorem of Bergman} \label{Subsection:BergmanAnalog}
In this subsection, we use Theorem \ref{T:GTFAmalgamIffFamilies} to derive two corollaries, which give us many examples of generalized torsion-free amalgams. We also provide two comments about the RTF condition in Theorem \ref{T:GTFAmalgamIffFamilies}. The content of this subsection will not be used in subsequent sections.

Recall that the \textbf{positive cone} of a left-ordering $<$ of a group $G$ is the semigroup $P = \{g \in G \mid g>1\}$, and if $<$ is a bi-ordering, then this semigroup is also normal.

\begin{corollary}\label{C:GTFAmalgamWhenPhiPreservesOrders}
Let  $G=\ast_C G_i$ ($i\in I$). Assume that $C$ is RTF in $G_i$ for all $i\in I$. 
If there is a positive cone $P_i$ of $G_i$ for each $i\in I$ such that $P_i\cap C=P_j\cap C$ for all $i,j\in I$, then $G$ is GTF.
\end{corollary}

\begin{proof}
The families $\mathcal F_i:=\{P_i,P_i^{-1}\}$ ($i\in I$) satisfy the conditions in Theorem \ref{T:GTFAmalgamIffFamilies}.
\end{proof}

Consider the following sample application of Corollary \ref{C:GTFAmalgamWhenPhiPreservesOrders}. Let $A$ (resp. $B$) be a free group generated by a non-empty set $S$ (resp. $T$) with $S\cap T=\emptyset$, and suppose $\alpha\in A$ (resp. $\beta\in B$) is not a proper power. Then the one-relator group $\langle S\cup T\mid \alpha=\beta\rangle$ is GTF, because the cyclic subgroup of $A$ (resp. $B$) generated by $\alpha$ (resp. $\beta$) is left-relatively convex by \cite[Corollary 3.6]{LRR09}, thus also RTF in $A$ (resp. $B$) by \cite[Lemma 3.15]{CH21}. We don't know if $\langle S\cup T\mid \alpha=\beta\rangle$ is necessarily bi-orderable.\looseness=-1

Theorem \ref{T:GTFAmalgamIffFamilies} also yields a GTF-analog of an orderability theorem of George Bergman \cite[Theorems 21 and 28]{Bergman90}. In particular, this analog shows that despite the previous example, there are circumstances where the RTF condition is necessary and sufficient for an amalgam to be GTF.  

\begin{corollary}\label{C:BOGTFOfExtendableAmalgam}
Let $A$ be a group and $C \subseteq A$ a subgroup. Consider the amalgam $G=\ast_CG_i$ ($i\in I$), where $I$ is an index set of cardinality at least two and each $G_i=A$. 
Then $G$ is GTF if and only if $A$ is GTF and $C$ is RTF in $A$.
\end{corollary}
\begin{proof}
Suppose $A$ is GTF and $C$ is RTF in $A$. For $a\in A$, let $N_a=\mathrm{NSS}_A(\{a\})$. Then taking one copy of the family $\{N_a\mid a\in A\backslash\{1\}\}$ for each factor yields the families of normal subsemigroups that satisfy the conditions in Theorem \ref{T:GTFAmalgamIffFamilies}.

Now assume that $G$ is GTF, then so is its subgroup $G_0:=A\ast_CA$. Suppose $C$ is not RTF in $A$ and we will find a contradiction. We  first prepare some notation.  The group $G_0$ is an amalgam of two copies of $A$; call them $A_1$ and $A_2$.  For each $a\in A_1$, we write $a' \in A_2$ to denote the image of $a$ under the identity map $A_1 \rightarrow A_2$. Note that for every $g\in A$, $g=g'$ in $G_0$ if and only if $g\in C$.

As $C$ is not RTF in $A$, there exists $a\in A\backslash C$ and $c_1,\dotsc,c_n\in C$ such that $T:=ac_1\dotsm ac_{n-1} ac_n\in C$. 
Hence $n\geq2$, $a'a^{-1}\neq1$ and $T=T'$ in $G_0$. Note that $T'=a'c_1\dotsm a'c_{n-1} a'c_n$, we have 
$$P_n:=c_n^{-1}a^{-1}c_{n-1}^{-1}a^{-1}\dotsm c_1^{-1}a^{-1}\cdot a'c_1\dotsm a'c_{n-1} a'c_n=T^{-1}T'=1.$$
We claim that $P_n$ is a product of $n$ conjugates of $a'a^{-1}$. Note that $P_1 = c_1^{-1}a^{-1}a'c_1 = (a'a^{-1})^{ac_1}$ so the base case for an induction holds, and for $n>1$, we can write
$P_n = (P_{n-1}a'a^{-1})^{ac_n}$, proving the claim.
As $P_n =1$, the group $G_0$ contains generalized torsion, a contradiction as desired.
\end{proof}

\begin{remark}
Note that $G=A*_CB$ being GTF does not imply that $C$ is RTF in either $A$ or $B$.
For example, consider $G=\langle a,b,c,d\mid a^2=c,b=d^2\rangle$, which is isomorphic to the free group $\langle a, d\mid\rangle$ because the relations $a^2=c$ and $d^2=b$ can be used to eliminate the generators $c$ and $b$.  We also have $G=A*_\phi B$, where $C$ (resp. $D$) is the subgroup of $A:=\langle a,b\mid\rangle$ (resp. $B:=\langle c,d\mid\rangle$) generated by $\{a^2,b\}$ (resp. $\{c,d^2\}$) and $\phi:C\rightarrow D$ is given by $\phi(a^2)= c, \phi(b)=d^2$. Clearly, $C$ (resp. $D$) is not RTF in $A$ (resp. $B$).
\end{remark}

\begin{remark}
    By Corollary \ref{C:BOGTFOfExtendableAmalgam}, the RTF condition in Theorem \ref{T:GTFAmalgamIffFamilies} is indispensable for the conclusion.
    In fact, the conclusion of Theorem \ref{T:GTFAmalgamIffFamilies} does not hold in general, even upon replacing the RTF condition with the weaker condition that $C$ is RTF in all but  except one $G_i$. To see this, consider the amalgam $G=A\ast_\phi B$, where $A=\langle a\mid\rangle$, $B=\langle b,c\mid bc=cb\rangle$, $C=\{a^{mk}\mid k\in\mathbb Z\}$ with $m\geq2$ fixed and $D=\{c^k\mid k\in\mathbb Z\}$, and $\phi:C\rightarrow D$ is given by $\phi(a^m)=c$. Clearly, $C$ is not RTF in $A$, and $D$ is RTF in $B$. Consider the element $\alpha=[a,b]$ of $G$, where we write $[x,y]$ for the commutator $xyx^{-1}y^{-1}$. Since $[a^m,b]=[c,b]=1$, iteratively applying the well-known commutator formula $[xy,z]=[y,z]^{x^{-1}}[x,z]$ to $[a^m,b]$, we find a product of $m$ conjugates of $\alpha=[a,b]$ that is trivial. Hence $\alpha$ is a generalized torsion element (as $m\geq2$). 
    \footnote{Note that this group is in fact the Baumslag-Solitar group  $BS(m,m)=\langle a,b\mid ba^mb^{-1}=a^m\rangle$.}
\end{remark}

\subsection{Multi-malnormality and two corollaries}\label{Subsec:Multi-malnormalityAndTwoCorollaries}

In this subsection, we define multi-malnormal subgroups. Multi-malnormality will be useful for constructing the families of normal subsemigroups as in Theorem \ref{T:GTFAmalgamIffFamilies}. To this end, we deduce Corollary \ref{C:RTF+multi-malnormalityImpliesGTF}, which underpins our construction of GTF and non-left-orderable groups in Section \ref{Section:GTFAndNonLOGroups}, and Corollary \ref{C:FactorSubgroupsAreMulti-malnormal}, which is needed in Section \ref{Section:One-RelatorNon-BO-GTF-Group}. 
This subsection's content is not used in Section \ref{S:3mflds}.\looseness=-1
    \begin{definition}\label{D:Malnormalities}
    Let $C$ be a subgroup of a group $A$. We say that $C$ is \textbf{multi-malnormal} in $A$ if for every normal subsemigroup $C'$ of $C$ with $1\not\in C'$, we have
    $$c_1^{a_1}\dotsm c_n^{a_n}\in A\backslash C, \text{ for all } c_i\in C' \text{ and } a_i\in A\backslash C \text{ where }n\geq1.$$
    \end{definition}
 
    \begin{lemma}\label{L:multi-malnormalStrongMalnormalMalnormal} Let $C$ be a multi-malnormal subgroup of group $A$. Let $C'$ be a normal subsemigroup of $C$ with $1\not\in C'$. Then we have the following.
    \begin{enumerate}[(1)]
    \item\label{Item:multi-malnormalEquivalentDefinition}
     For all $ c_i\in C'$ and $a_i\in A$, if there is an  $a_j\in A\backslash C$, then $c_1^{a_1}\dotsm c_n^{a_n}\in A\backslash C$.
\item \label{Item:multi-malnormalToWeaklyMultimalnormal} $\mathrm{NSS}_A(C')\cap C=C'$.
    \end{enumerate}
        \end{lemma}
        
    \begin{proof} 
   To prove (1), we assume that $c_1, \ldots, c_n\in C'$ and $a_1, \ldots, a_n \in A$ and that there is an $a_j\in A\backslash C$, and yet $T:=c_1^{a_1}\dotsm c_n^{a_n}\in C$, where we choose $n$ to be the least possible (note necessarily $n \geq 2$).  By multi-malnormality, there exists $a_i \in C$. Note that we can rewrite:
   \[ T= c_i^{a_i}(c_1^{a_1}\ldots c_{i-1}^{a_{i-1}})^{c_i^{a_i}}c_{i+1}^{a_{i+1}} \ldots c_n^{a_n} = c_i^{a_i} c_1^{a_1'} \ldots c_{i-1}^{a_{i-1}'} c_{i+1}^{a_{i+1}} \ldots c_n^{a_n}=c_i^{a_i}T',
   \] where $a_k'=a_kc_i^{a_i}$ and $T'=c_1^{a_1'} \ldots c_{i-1}^{a_{i-1}'} c_{i+1}^{a_{i+1}} \ldots c_n^{a_n}$.
   But now $ T'\notin C$ by minimality of $n$, a contradiction since $T'=(c_i^{-1})^{a_i} T \in C$.

     Now we show (2), the RHS is clearly contained in the LHS. To prove the other direction, assume that $T:=c_1^{a_1}\dotsm c_n^{a_n}$ with $a_i\in A$ and $c_i\in C'$ is contained in the LHS. Then all $a_i\in C$, as otherwise $T$ is not in $C$ by \ref{Item:multi-malnormalEquivalentDefinition} . Now $T$ is in $C'$ by the normality of $C'$ in $C$, as desired.
    \end{proof}   

\begin{corollary}\label{C:RTF+multi-malnormalityImpliesGTF}
Consider the amalgam $G=\ast_C G_i$ ($i\in I$), where for each $i\in I$, $G_i$ is GTF and $C$ is RTF and multi-malnormal in $G_i$. Then $G$ is GTF. 
\end{corollary}

\begin{proof} We only need to find families of normal subsemigroups satisfying the conditions in Theorem \ref{T:GTFAmalgamIffFamilies}. For every $i\in I$ and $a\in G_i\backslash\{1\}$, let $N_{i,a}=\mathrm{NSS}_{G_i}(\{a\})$ and $M_{i,a,j}=\mathrm{NSS}_{G_j}(N_{i,a}\cap C)$  for each $j\in I$.
Note that $1\not\in N_{i,a}$ as $G_i$ is GTF, thus $1$ is not in $C_{i,a}':=N_{i,a}\cap C$, which is a normal subsemigroup of $C$.  Hence we have, by Lemma \ref{L:multi-malnormalStrongMalnormalMalnormal}\ref{Item:multi-malnormalToWeaklyMultimalnormal}, for all $i\in I$, $a\in G_{i}\backslash\{1\}$ and $j\in I$, 
\begin{align}\label{E:KeyIDInRTF+multi-malnormalityImpliesGTF}
    M_{i,a,j}\cap C=\mathrm{NSS}_{G_j}(C_{i,a}')\cap C=C_{i,a}'=N_{i,a}\cap C
\end{align}
and thus $1\not\in M_{i,a,j}$ as $1\not\in N_{i,a}\cap C$ and $1\in C$. Moreover, for all $i\in I$ and $a\in G_i\backslash\{1\}$, we have 
$M_{i,a,j_1}\cap C=M_{i,a,j_2}\cap C$ for all $j_1,j_2\in I$, since $M_{i,a,j_t}\cap C=N_{i,a}\cap C$ for $t=1,2$ by \eqref{E:KeyIDInRTF+multi-malnormalityImpliesGTF}. 

Now we find families $\mathcal F_j$ ($j\in I$) satisfing the two conditions in Theorem \ref{T:GTFAmalgamIffFamilies}, where for each $j\in I$, $$\mathcal F_j:=\big\{N_{j,a}\mid j\in I,a\in G_j\backslash\{1\}\big\}\cup\big\{M_{i,a,j}\mid i\in I, a\in G_i\backslash\{1\}\big\}.$$
\end{proof}

We now turn to the theorem below, whose proof is similar to Theorem \ref{T:TauInSAIfInC}, so we do not present all the details here. Using this theorem, we then show that factor groups in free products are multi-malnormal. This serves as an example of multi-malnormality and as an essential building block for our construction in Section \ref{Section:One-RelatorNon-BO-GTF-Group}. Below, instead of considering $G=\ast_IG_i$ ($i\in I$), we focus on the case $|I| = 2$ since the general case follows easily from this special case.

\begin{theorem}\label{T:Multi-malnormalFactorGroup-Alike-Result}
    Consider the amalgam $G=A\ast_CB$. Assume that $P$ is a normal subsemigroup of $A$ such that $P\cap C=\emptyset$. Let $T=d_1^{f_1}\dotsm d_n^{f_n}$ with $n\geq1$, all $d_i\in P$ and all $f_i\in G\backslash A$. Then $T\not\in A$.
\end{theorem}
\begin{proof} Let $\mathcal S$ be the set of tuples   
$((s_1,h_1),\dotsc,(s_m,h_k))$ such that $k\geq1$, $s_i\in P$, $h_i\in G\backslash A$ and $s_1^{h_1}\dotsm s_m^{h_k}\in A$. 
We want to prove that $\mathcal S$ is empty. Suppose not, and we will find a contradiction. We define the maps $NC,L,N$ as in the proof of Theorem \ref{T:TauInSAIfInC}, and let 
$v_0=((t_1,g_1),\dotsc,(t_n,g_n))$ 
be a minimal element in $\mathcal S$. We will prove that $v_0$ is tamed, then similarly we have $n=1$ and thus $t_1^{g_1}\in A$, which is impossible because  $t_1^{g_1}$ is reduced, $t_1\in P\backslash C$ and $g_1\not\in A$.

Now we verify that $v_0$ is tamed. First, $l(t_i)=1$ is already true as $t_i\in P\backslash C$. Second, each $t_i^{g_i}$ is reduced by the minimality of $v_0$ and the normality of $P$. Third, if $c_i:=g_ig_{i+1}^{-1}\in C$, then $t_i^{g_i}t_{i+1}^{g_{i+1}}=t_i'^{g_i}$, where $t_i'=t_it_{i+1}^{c_i^{-1}}$ is in $P$, which contradicts the minimality as in \ref{Item:ltiti+1T:TauInSAIfInC} in the proof of Theorem \ref{T:TauInSAIfInC}. Fourth, the proof that no $t_i$ is cancellable remains the same as \ref{Item:Non-cancellableProofOfT:TauInSAIfInC} in the proof of Theorem \ref{T:TauInSAIfInC}, except that right after the fifth sentence there, we insert this sentence: We can assume that both $g_{i-1}'$ and $g_{i+1}'$ are in $G\backslash A$, because if $g_{i-1}'
$ or $g_{i+1}'$ is in $A$, 
then $i>1$ or $i<n$ respectively, and we can ``move" the corresponding conjugation ($t_{i-1}^{g_{i-1}'}$ or $t_{i+1}^{g_{i+1}'}$ respectively) to one end using the trick in \ref{Item:lgieq1TauInSAIfInC} and similarly find a $w\in\mathcal S$ with $NC(w)<NC(v_0)$. 
\end{proof}

\begin{corollary}\label{C:FactorSubgroupsAreMulti-malnormal}
    Let $G=A\ast B$ be a free product of groups $A$ and $B$. Then $A$ is multi-malnormal in $G$. 
\end{corollary}
\begin{proof}
    Let $P$ be a normal subsemigroup of $A$ with $1\not\in P$, we want to show that $T:=d_1^{f_1}\dotsm d_n^{f_n}\not\in A$ for all $d_i\in P$, $f_i\in G\backslash A$ and $n\geq1$. This follows from Theorem \ref{T:Multi-malnormalFactorGroup-Alike-Result}, where $C$ is now $\{1\}$.\looseness=-1
\end{proof}

\section{A 3-manifold group which is GTF and not bi-orderable}\label{Section:3-minifold}
\label{S:3mflds}

 In this section, we prove Theorem \ref{thm:intro3mfldthm}.
Consider the figure eight knot $K \subseteq S^3$, depicted in Figure \ref{fig:fig8}.  Let $\nu(K)$ denote an open tubular neighbourhood of $K$, and set $M = S^3 \setminus \nu(K)$, the exterior of $K$.  Then $M$ is a compact, connected, orientable $3$-manifold whose fundamental group is 
\[ \pi_1(M) = \langle x, y \mid wx=yw \rangle
\]
where $w = xy^{-1}x^{-1}y$. We fix a choice of peripheral subgroup $\pi_1(\partial M) \cong \mathbb{Z} \oplus \mathbb{Z}$ by fixing generators $\mu = x$ and $\lambda= yx^{-1}y^{-1}x^2y^{-1}x^{-1}y$, called the meridian and longitude of $K$.

\begin{figure}[h]
 \includegraphics[scale=0.2]{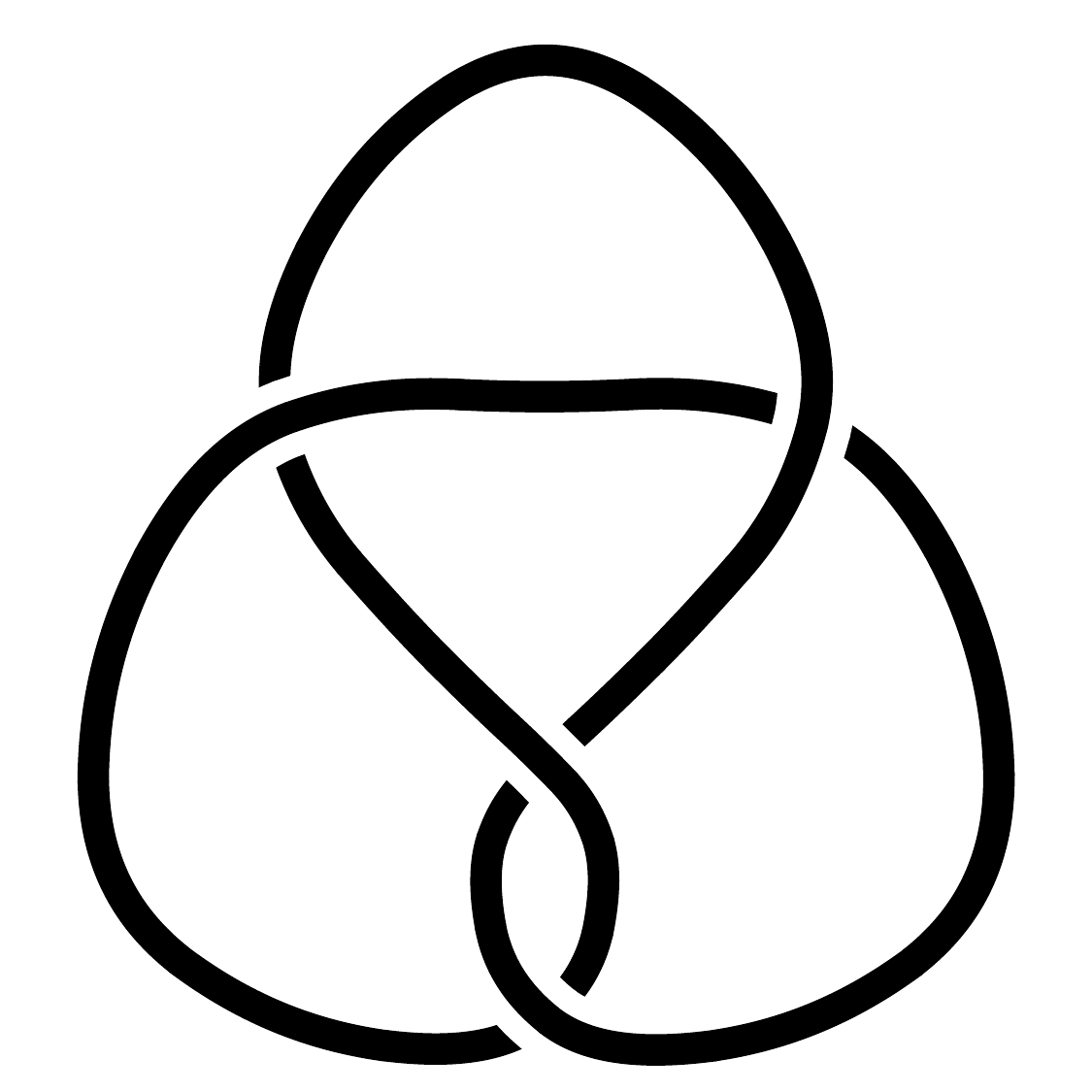}
 \caption{The figure eight knot.}
 \label{fig:fig8}
\end{figure}

Take two copies of $M$, call them $M_1$ and $M_2$, we denote the generators of their respective fundamental groups by $x_1, y_1$ and $x_2, y_2$, and the generators of their respective peripheral subgroups by $\mu_1, \lambda_1$ and $\mu_2, \lambda_2$.

Choose a homeomorphism $\phi : \partial M_1 \rightarrow \partial M_2$ such that the induced map $\phi_* : \pi_1(\partial M_1) \rightarrow \pi_1(\partial M_2)$ is given by $\phi_*(\mu_1) = \mu_2$ and $\phi_*(\lambda_1) = \mu_2\lambda_2$.  Let $W = M_1 \cup_{\phi} M_2$ be the manifold obtained by gluing $M_1$ and $M_2$ by $\phi$; by the Seifert-Van Kampen theorem, $ \pi_1(W) = \pi_1(M_1) \ast_{\phi_*} \pi_1(M_2).$

\begin{proposition}
    The group $\pi_1(W)$ is not bi-orderable.
\end{proposition}
\begin{proof}
Every finitely generated bi-orderable group admits a surjection onto $\mathbb{Z}$ \cite[Theorem 2.9]{CR16}.  However, one can verify directly from a presentation of $\pi_1(W)$ that $\pi_1(W)$ is perfect.
\end{proof}

The rest of this section is dedicated to proving that $\pi_1(W)$ is GTF.  We begin by constructing a handful of useful normal subsemigroups.

\begin{lemma}\label{L:NSSQandR}
\label{lem:RQ}
Let $\pi_1(M)$ denote the fundamental group of the figure eight knot complement, and $\pi_1(\partial M)$ the peripheral subgroup with generators $\mu, \lambda$ as above. Then there exist normal subsemigroups $Q$ and $R$ of $\pi_1(M)$ satisfying $$Q \cap \pi_1(\partial M) = \{ \lambda^k \mid k >0 \} \quad\text{  and  }\quad R \cap \pi_1(\partial M) = \{  \mu^n  \lambda^m \mid n>0 \mbox{ and } m \in \mathbb{Z} \}$$
such that $P = Q \cup R$ is the positive cone of a bi-ordering of $\pi_1(M)$.
\end{lemma}
\begin{proof}
There is a short exact sequence 
  \[ 1 \longrightarrow [\pi_1(M), \pi_1(M)] \longrightarrow \pi_1(M) \stackrel{h}{\longrightarrow} \mathbb{Z} \longrightarrow 1,
  \]
where $h$ is the abelianization map and $h(x) =1$, and $\lambda \in [\pi_1(M), \pi_1(M)]$ as $\lambda$ is a product of commutators.  In \cite{PR03}, they show there is a bi-ordering of $ [\pi_1(M), \pi_1(M)]$ with positive cone $Q$ such that $f(Q) = Q$, where $f: [\pi_1(M), \pi_1(M)] \rightarrow [\pi_1(M), \pi_1(M)]$ is the automorphism given by conjugation by $x$.  For our purposes, we choose such a $Q$ satisfying $\lambda \in Q$ and set $R = \{ g \in \pi_1(M) \mid h(g) >0\}$.  Note that $Q$ and $R$ are normal subsemigroups of $\pi_1(M)$, and that setting $P = Q \cup R$ yields the positive cone of a bi-ordering of $\pi_1(M)$  \cite[Corollary 3.4]{PR03}.

Since $Q$ is contained in the commutator subgroup of $\pi_1(M)$, we have $Q \cap \pi_1(\partial M) \subseteq \langle \lambda \rangle,$ and since we have chosen $Q$ to contain $\lambda$, we have $Q \cap \pi_1(\partial M) = \{ \lambda^{k} \mid k >0 \}$.  That $R \cap \pi_1(\partial M) = \{  \mu^n\lambda^m \mid n>0 \mbox{ and } m \in \mathbb{Z} \}$ follows from the fact that $h(\mu) = 1$ and $h(\lambda) =0$. 
\end{proof}

Next, recall that the universal covering group of $\mathrm{Homeo}_+(S^1)$ is 
$$\mathrm{H\widetilde{ome}o}_+(S^1) = \{ f \in \mathrm{Homeo}_+(\mathbb{R}) \mid f(x+1) = f(x) +1 \},$$ 
and if $f \in \mathrm{H\widetilde{ome}o}_+(S^1)$ then we may define the \textbf{translation number} of $f$ to be 
\[ \tau(f) = \lim_{n \to \infty} \frac{f^n(x_0) -x_0}{n}.
\]

It is known (see e.g. \cite[Chapter 11]{KH95}) that the limit exists and it is independent of the choice of $x_0$, and that $\tau(f) =0$ if and only if $f$ has a fixed point.  It follows that $\tau(f)>0$ (resp. $\tau(f)<0$) if and only if $f(x) > x$ (resp. $f(x) < x$).  Moreover, the restriction of $\tau$ to any abelian subgroup is a homomorphism. 

\begin{lemma}\label{lem:X}
 Let $\pi_1(M)$ denote the fundamental group of the figure eight knot complement, and $\pi_1(\partial M)$ the peripheral subgroup with generators $\mu, \lambda$ as above. Then there exist normal subsemigroups $X_{(1,1)}$ and $X_{(-1,1)}$ of $\pi_1(M)$ satisfying 
    \begin{align}\label{E:Xpm11}
        &X_{(1,1)} \cap \pi_1(\partial M) = \{ \mu^n \lambda^{n+m} \mid  n \in \mathbb{Z} \mbox{ and } m >0 \},\\ \nonumber 
        &X_{(-1,1)} \cap \pi_1(\partial M) = \{ \mu^{-n}\lambda^{n+m} \mid  n \in \mathbb{Z} \mbox{ and } m >0 \}.
    \end{align}
\end{lemma}
\begin{proof}
In \cite{BGH21}, the authors construct a representation $\rho : \pi_1(M) \rightarrow \mathrm{H\widetilde{ome}o}_+(S^1)$ such that $\tau(\rho(\mu))$ and $\tau(\rho(\lambda))=d>0$ are integers  \cite[Theorem 1.5(1) and proof of Theorem 1.3(2)]{BGH21}.  From this, the authors show that if $p$ divides $d$ and if $p,q$ are relatively prime integers, one can then construct a new representation $\rho_{(p,q)}: \pi_1(M) \rightarrow \mathrm{H\widetilde{ome}o}_+(S^1)$ such that $\tau(\rho_{(p,q)}(\mu^p \lambda^q)) =0$ and $\tau(\rho_{(p,q)}(\lambda))  >0$ \cite[Proof of Proposition 6.10]{BGH21}.  We review the construction of $\rho_{(p,q)}$ here,  
as we require such representations for $(p,q) = (1,1)$ and $(p,q) = (-1,1)$ in order to construct the normal subsemigroups $X_{(1,1)}$ and $X_{(-1,1)}$.

Beginning with $\rho : \pi_1(M) \rightarrow \mathrm{H\widetilde{ome}o}_+(S^1)$ such that $\tau(\rho(\mu))$ and $\tau(\rho(\lambda))=d>0$ are integers, and suppose $p$ divides $d$ and that $p, q$ are relatively prime.  Choose a homomorphism $\beta: \pi_1(M) \rightarrow \mathbb{Z}$ satisfying $\beta(\mu) = -\tau(\rho(\mu)) - q(d/p)$, then let the map $sh_k : \mathbb{R} \rightarrow \mathbb{R}$ be given by $sh_k(x) = x+k$, and one sets
\[\rho_{(p,q)}(g) = \rho(g) \circ sh_{\beta(g)}.
\]
Now we set $ X_{(p,q)} = \{ g \in \pi_1(M) \mid \tau(\rho_{(p,q)}(g))>0 \}$. Then $X_{(p,q)}$ is a normal subsemigroup of $\pi_1(M)$, because (1) $X_{(p,q)} = \{ g \in \pi_1(M) \mid \rho_{(p,q)}(g)(x)>x \mbox{ for all } x \in \mathbb{R} \}$ and defining property of this set is clearly preserved by taking products and (2) translation number is invariant under conjugation.

Note that $\tau(\rho_{(p,q)}(\lambda)) =d$ and $\tau(\rho_{(p,q)}(\mu)) = \tau(\rho(\mu)\circ sh_{-\tau(\rho(\mu)) - q(d/p)}) = -q(d/p)$, thus we have 
\begin{align*}
    X_{(p,q)}\cap \pi_1(\partial M)=\{\mu^n\lambda^k\mid n(-q(d/p))+kd>0\},
\end{align*} from which the equalities \eqref{E:Xpm11} follow.
    \end{proof}

\begin{lemma}
\label{lem:Y}
    Let $\pi_1(M)$ denote the fundamental group of the figure eight knot complement, and $\pi_1(\partial M)$ the peripheral subgroup with generators $\mu, \lambda$ as above. Then there exist normal subsemigroups $Y_1, Y_2$ of $\pi_1(M)$ with $Y_2 \cap \pi_1(\partial M) = \{ (\mu \lambda)^k \mid k >0 \}$ and  $Y_1 \cap \pi_1(\partial M) = \{ (\mu^{-1} \lambda)^k \mid k >0 \}$.\looseness=-1
\end{lemma}
\begin{proof}
Let $P$ be the positive cone of a bi-ordering of $\pi_1(M)$ chosen so that $\mu \lambda \in P$, and define $Y_2 = P \cap \langle \langle \mu \lambda \rangle \rangle$, where $\langle \langle \mu \lambda \rangle \rangle$ is the normal closure of $\mu \lambda$ in $\pi_1(M)$.  Then $Y_2$ is a normal subsemigroup of $\pi_1(M)$.  Similarly, let $P'$ be the positive cone of a bi-ordering of $\pi_1(M)$ chosen so that $\mu^{-1} \lambda \in P'$ and set $Y_1 = P' \cap \langle \langle \mu^{-1} \lambda \rangle \rangle$.

 Since $+1$ Dehn filling on the figure eight knot yields a manifold with torsion-free, nontrivial fundamental group, we know that $\langle \langle \mu \lambda \rangle \rangle \cap \pi_1( \partial M)=\{(\mu\lambda)^k\mid k\in\mathbb Z\}$, from which the first equality follows. The second equality follows from a similar argument. 
\end{proof}

\begin{theorem}
    The group $\pi_1(W)$ is GTF.
\end{theorem}
\begin{proof}
We use the notation established at the beginning of this section and the normal subsemigroups established in the preceding lemmas.  Recall that $\pi_1(W) = \pi_1(M_1) \ast_{\phi_*} \pi_1(M_2),$ where $\phi_* : \pi_1(\partial M_1) \rightarrow \pi_1(\partial M_2) $ is given by $\phi_*(\mu_1) = \mu_2$ and $\phi_*(\lambda_1) = \mu_2\lambda_2$.

First, note that each $\pi_1(M_i)$ is GTF because it is bi-orderable, by Lemma \ref{L:NSSQandR}.   
Now, consider the families of normal subsemigroups:
\[ \mathcal F_1 = \{R_1^{\pm 1}, Q_1^{\pm 1}, X_{(-1,1)}^{\pm 1}, Y_1^{\pm 1} \}
\text{ and }  \mathcal F_2 = \{R_2^{\pm 1}, Q_2^{\pm 1}, X_{(1,1)}^{\pm 1}, Y_2^{\pm 1} \},
\]
where we interpret $R_i^{\pm 1}$, $Q_i^{\pm 1}$ and $Y_i^{\pm 1}$ as copies of $R^{\pm 1}$, $Q^{\pm 1}$ and $Y^{\pm 1}$ lying inside $\pi_1(M_i)$ for $i=1,2$.  Moreover, using that $\pi_1(M_i)$ is a copy of $\pi_1(M)$ for $i=1, 2$, we take $X_{(-1,1)}^{\pm 1} \subseteq \pi_1(M_1)$ and $ X_{(1,1)}^{\pm 1} \subseteq \pi_1(M_2)$. 
We will show that these families satisfy the conditions of Theorem \ref{T:GTFAmalgamIffFamilies}, and that $\pi_1(\partial M_i)$ is RTF in $\pi_1(M_i)$.

First, that $\pi_1(\partial M_i)$ is RTF in $\pi_1(M_i)$ follows from an application of the fact that every isolated abelian subgroup of a bi-orderable group is left-relatively convex (See, e.g. \cite[Corollary 5.4(2)]{BC24}).  In our case, since the peripheral subgroup $\pi_1(\partial M)$ of $\pi_1(M)$ is isolated by \cite{Simon76}, it is relatively convex, and by \cite[Lemma 3.15]{CH21}, relatively convex subgroups are RTF.

Next, observe that no element of $\mathcal F_i$ contains the identity, because of equalities \eqref{E:Xpm11} and that every element of $\mathcal F_i$ other than $X_{\pm1,1}$ is a subset of a positive cone in $\pi_1(M_i)$.  Moreover, since $Q_i \cup R_i$ is the positive cone of a bi-ordering of $\pi_1(M_i)$, we know that the union of all elements in $\mathcal F_i$ is equal to $\pi_1(M_i) \setminus \{ 1\}$.

It remains to check condition (2) of Theorem \ref{T:GTFAmalgamIffFamilies} for the given families. Employing Lemmas \ref{lem:RQ}, \ref{lem:X}, \ref{lem:Y}, and recalling that $\phi_*(\mu_1)=\mu_2$ and $\phi_*(\lambda_1)=\mu_2\lambda_2$,  
we compute: 
\[ \phi_*(R_1^{\pm 1} \cap \pi_1(\partial M_1)) = X_{(1,1)}^{\mp 1} \cap \pi_1(\partial M_2),
\]
\[ \phi_*(Q_1^{\pm 1} \cap \pi_1(\partial M_1)) = Y_2^{\pm1 } \cap \pi_1(\partial M_2),
\]
\[ \phi_*(X_{(-1,1)}^{\pm 1} \cap \pi_1(\partial M_1)) = R_2^{\pm1 } \cap \pi_1(\partial M_2),
\]
\[ \phi_*(Y_1^{\pm 1} \cap \pi_1(\partial M_1)) = Q_2^{\pm1 } \cap \pi_1(\partial M_2).
\]
That $\pi_1(W)$ is GTF now follows from Theorem \ref{T:GTFAmalgamIffFamilies}.
\end{proof}

\section{A one-relator group which is GTF and not bi-orderable}\label{Section:One-RelatorNon-BO-GTF-Group}

The goal of this section is to answer \cite[Question 3]{CGW15} in the negative by proving Theorem \ref{thm:onerelatorGTFnonBO}. As mentioned in the introduction, we take inspiration for the non-bi-orderability of our example from \cite{Ak23}.  Specifically,  \cite[Theorem 1.1]{Ak23} says that that the group $\Gamma:=\langle a,t\mid a_1(a_0a_2)a_1^{-1}=(a_0a_2)^2\rangle$, where we write $a_i$ for $a^{t^{-i}}$, is a one-relator, non-bi-orderable and GTF group. 

Unfortunately, this group turns out to contain generalized torsion.  To see this, note that $\alpha:=a_0a_2a_1(a_0a_2)^{-1}a_1^{-1}a_0a_2$ is trivial, and hence, so is $\beta:=\alpha^{t^{-1}}\alpha^{a_2^{-1}a_1a_3}$.  Then we compute 
\begin{align*}
    \beta&=a_1a_3a_2(a_1a_3)^{-1}\cdot a_0a_2a_1(a_0a_2)^{-1}a_1^{-1}a_0a_2\cdot a_2^{-1}a_1a_3\\
         &=a_1a_3a_2(a_1a_3)^{-1}\cdot a_0a_2a_1(a_0a_2)^{-1}\cdot a_1^{-1}a_0a_1\cdot a_3\\
         &=a_2^{(a_1a_3)^{-1}}a_1^{(a_0a_2)^{-1}}a_0^{a_1} a_3,
\end{align*} and since the $a_i$'s are conjugates of $a$, this shows $a$ is a generalized torsion element. However, the group $\Gamma$ is non-bi-orderable, via a proof similar to that of Lemma \ref{lem:nonbolemma} below. 

\begin{lemma}
\label{lem:nonbolemma}
    Let $G$ be a group and $a, b \in G$.  Assume that $<$ is a bi-ordering of $G$ and $a>1$. Then $a^{b^2}(a^b)^{-1}a >1$ and $(a^{b^2})^{-1}a^ba^{-1} <1$.
\end{lemma}

\begin{proof}
    If $a=a^b$, the conclusion is obvious.  Hence we assume $a\neq a^b$, and then we have either $a^{b^2}>a^b>a$ or $a>a^b>a^{b^2}$. We then have $a^{b^2}(a^b)^{-1}>1$ or $(a^b)^{-1}a>1$ correspondingly, implying that $a^{b^2}(a^b)^{-1}a>1$ as $a>1$ and $a^{b^2}>1$. Thus we have $a^{b^2}(a^b)^{-1}a>1$ for all $b\in G$, hence upon replacing $b$ with $b^{-1}$, we have $a^{b^{-2}}(a^{b^{-1}})^{-1}a>1$, and then $aa^ba^{b^2}>1$ by conjugating $b^2$, implying $(a^{b^2})^{-1}a^ba^{-1} <1$ by taking inverse. 
    \end{proof}

  Note that non-biorderability of the group $G=\langle a,b,d\mid a^{b^2}(a^{b})^{-1}a=(a^{d^2})^{-1}a^{d}a^{-1}\rangle$ in Theorem \ref{thm:onerelatorGTFnonBO} is a direct consequence of the previous lemma. Assuming that $<$ is a bi-ordering of $G$, we may also assume that $1<a$; then by Lemma \ref{lem:nonbolemma} the left-hand side of the relator is positive, while the right-hand side is negative, a contradiction.

We now turn to proving that $G$ is GTF. To do so, we will use the fact that 
$$G=\langle a,b,c,d\mid a=c^{-1}, a^{b^2}(a^{b})^{-1}a=c^{d^2}(c^{d})^{-1}c\rangle,$$ 
where we recover the one-relator presentation of Theorem \ref{thm:onerelatorGTFnonBO} by replacing $c$ with $a^{-1}$.  

From this presentation, we see that $G$ is the amalgam $A\ast_{\phi}B$, where $A=\langle a,b\mid\rangle$, $B=\langle c,d\mid\rangle$ are free groups of rank two, $C$ and $D$ are the subgroups of $A$ and $B$ freely generated by $\{a,a^{b^2}(a^b)^{-1}a\}$ and $\{c^{-1}, c^{d^2}(c^d)^{-1}c\}$, respectively, with the isomorphism
$\phi:C\rightarrow D$ defined by $$\phi(a)=c^{-1},\phi(a^{b^2}(a^{b})^{-1}a)= c^{d^2}(c^{d})^{-1}c.$$

We introduce some notation for ease of discussion.  Let $F$ be a free group generated by a nonempty set $S$ and $t\in S$. We define a homomorphism $W_{S,t}: F \rightarrow \mathbb Z$ by $W_{S,t}(u)=\delta_{tu}$ ($u\in S$), where $\delta_{tu}=1$ if $t=u$ and $\delta_{tu}=0$ if $t\neq u$. We call $W_{S,t}(g)$ the \textbf{weight} of $t$ in $g$ (with respect to the generating set $S$); it is the sum of the exponents of $t$ when $g$ is written as a (reduced) word in $S$. 

Set $N = \ker(W_{\{a,b\},b})$, then $N$ is freely generated by $S=\{a_{i}\mid i\in\mathbb Z\}$, where $a_i=a^{b^i}$. Another free generating set of $N$ is $S'=\{v_i\mid i\in\mathbb Z\}$, where $v_1:=a_2a_1^{-1}$ and $v_i:=a_i$ for all $i\neq1$. Note that $C$ is the free subgroup of $N$ generated by $\{v_0, v_1\}$. Letting $N_1$ be the subgroup of $N$ generated by $S'\setminus \{v_0, v_1\}$, observe that  $N=C\ast N_1$, 
which is central to the following proofs.

We also note that the isomorphism $\Phi:A\rightarrow B$ defined by $\Phi(a)=c$ and $\Phi(b)=d$ satisfies
$\Phi(C)=D$.

\begin{lemma}\label{L:One-Relator-RTF}
    The subgroup $C$ (resp. $D$) is RTF in $A$ (resp. $B$).
\end{lemma}
\begin{proof}
     Because of the isomorphism $\Phi$, we only need to prove that $C$ is RTF in $A$. By \cite[Lemma 3.15]{CH21}, it suffices to check that $C$ is relatively convex in $A$.  Since $A/N \cong \mathbb{Z}$, $N$ is relatively convex in $A$.  Since $N=C\ast N_1$, $C$ is relatively convex in $N$ by \cite[Corollary 1.19]{ADW18}.  It follows that $C$ is relatively convex in $A$, since relative convexity is transitive. 
\end{proof}

   We now want to find two families of normal subsemigroups of $A$ and $B$ respectively, satisfying the conditions (1) and (2) in Theorem \ref{T:GTFAmalgamIffFamilies}. The main tools we will use in the proof are the Magnus representation of free groups (see Proposition 10.1 on page 68 of \cite{LS77-Reprint}), the Freiheitssatz (see Proposition 5.1 on Page 104 of \cite{LS77-Reprint}) and Corollary \ref{C:FactorSubgroupsAreMulti-malnormal}.  

   Let $F$ be a free group generated by $S$, $t\in S$ and $g\in F$.  Write $g=x_1\dotsm x_n$ (with $x_i\in S\cup S^{-1}$ and each $x_ix_{i+1}\neq1$) as a reduced word in $S$. If $\{t,t^{-1}\}\cap\{x_1,\dotsc,x_n\}\neq\emptyset$, we say that $t$ \textbf{appear}s in $g$ (or $g$ contains $t$). Recall that Freiheitssatz says that if $t$ appears in $g$ and $g$ is cyclically reduced (with respect to $S$), then $t$ appears in each nontrivial element of the normal subgroup of $F$ generated by $g$. 

We begin with a preparatory lemma.
\begin{lemma}\label{L:StrongMalnormalityLikeByFreiheitssatz} Let $\alpha\in A\backslash\{1\}$. Consider the product $\tau=\alpha^{g_1}\dotsm\alpha^{g_r}$, with $g_1,\dotsc,g_r\in N$ and $r\geq1$. 
    Then $\tau\not\in C$ if one of the following two conditions holds:
    \begin{enumerate}
    \item $\alpha$ is not conjugate into $C$, i.e., $\alpha^g \in A\backslash C$ for all $g \in N$; 
        \item $\alpha\in C$ and there exists $i \in \{1, \ldots, r\}$ such that $g_i\in N\backslash C$.
    \end{enumerate}
    \end{lemma}
    \begin{proof} 
If $\alpha\not\in N$, then $\tau\not\in N$ (and thus $\tau\not\in C$) by considering the weight function $W_{\{a,b\},b}$. Hence, we assume that $\alpha \in N$.

Assume that condition (1) holds. We write $\alpha = \alpha'^h$, where $h \in N$ and $\alpha' \notin C$ is cyclically reduced with respect to the generating set $S'$ of $N$.  Since $\alpha' \notin C$ there exists $v_i \in S'\setminus\{v_0, v_1\}$ such that $v_i$ appears in $\alpha'$.  Rewriting $\tau$ as $\alpha'^{g_1'}\dotsm \alpha'^{g_r'}$, with $g_i'\in N$, by Freiheitssatz $v_i$ must appear in $\tau$, in particular, $\tau \notin C$.

 Now, assume that condition (2) holds. Note that $C$ is multi-malnormal in $N$ by Corollary \ref{C:FactorSubgroupsAreMulti-malnormal} (recalling that $N=C\ast N_1$), so we conclude that $\tau\not\in C$ by Lemma \ref{L:multi-malnormalStrongMalnormalMalnormal}\ref{Item:multi-malnormalEquivalentDefinition}.
    \end{proof}
   
We recall the Magnus embedding for our next proofs.  Let $R=\mathbb Z[[X_i\mid i\in\mathbb Z]]$ be the ring of formal power series in the noncommuting variables $X_i$. The Magnus embedding $\mu : N \rightarrow R$ is defined by $\mu(a_i)=1+X_i$, 
whose image lies in the group of units of $R$. 

For every $\alpha\in N$, we can write $\mu(\alpha)=1+\sum_{i=1}^{\infty} P_i$, where $P_i$ is a homogeneous polynomial of degree $i$. When $\alpha\neq1$, there is a minimal $i\geq1$ such that $P_i\neq0$; we call this $P_i$ the \textbf{leading term} of $\mu(\alpha)$, and denote it by $L(\alpha)$. The following properties can be confirmed by straightforward computation and are often used without explicit mention in the proofs of this section. 
\begin{lemma}\label{L:SimplePropertyOfLeadingTerm}
\begin{enumerate}
    \item $L(\alpha^g)=L(\alpha)$ for all $\alpha,g\in N$.

    \item Let $\alpha\in N$, $\mu(\alpha)=1+\sum_{i=1}^{\infty} P_i$, where $P_i$ is a homogeneous polynomial of degree $i$. Then $P_1=\sum_{i\in\mathbb Z}W_{S,a_i}(\alpha)X_i$. 
    \item \label{Item:Degree1LeadingTermInC}Let $\alpha\in C$. Then $L(\alpha)$ is of degree one if and only if $(W_{S',v_0}(\alpha), W_{S',v_1}(\alpha))\neq(0,0)$, and in this case $L(\alpha)=W_{S',v_0}(\alpha)X_0+W_{S',v_1}(\alpha)(X_2-X_1)$ and in particular, the sum of the coefficients of $X_1$ and $X_2$ is zero.
\end{enumerate}
\end{lemma}
 Let $P_k\in R$ be a homogeneous polynomial of degree $k$ (with $k\geq1$). Then $$P_k=\sum_{(i_1,\dotsc,i_k)}c_{i_1,\dotsc,i_k} X_{i_1}\dotsm X_{i_k},$$ where $(i_1,\dotsc,i_k)$ runs through all $k$-tuples in $\mathbb Z^k$, $c_{i_1,\dotsc,i_k}\in \mathbb Z$ and all but finitely many of the $c_{i_1,\dotsc,i_k}$'s are zero. For $j\in\mathbb Z$, we say that the variable $X_j$ \textbf{appears} in $P_k$ if there is a $k$-tuple $(i_1,\dotsc,i_k)$, such that $c_{i_1,\dotsc,i_k}\neq0$ and $j\in\{i_1,\dotsc,i_k\}$.
Let $I$ be the ideal of $R$ generated by a subset $R_0$ of $R$. There is a canonical homomorphism $\pi$ from $R$ to the quotient $R/I$. 
For a homogeneous polynomial $P\in R$, we say that the relations $f=0$ ($f\in R_0$) \textbf{annihilate} $P$ if $\pi(P)=0$. For $\alpha\in N$, we say that the relations $f=0$ ($f\in R_0$) annihilate $\alpha$ if $\pi(\mu(\alpha))=\pi(1)$. 
 
\begin{lemma}\label{L:IdealGeneratedByHomogeneousPolynomials}
Assume that all polynomials in $R_0$ are homogeneous and let $\alpha\in N$ be nontrivial. Assume that the relations $f=0$ ($f\in R_0$) annihilate $\alpha$, then they annihilate $L(\alpha)$.
\end{lemma}
\begin{proof}To see this, we suppose $\mu(\alpha)=1+\sum_{i=1}^\infty P_i$ is in the ideal $I$ generated by $R_0$, where each $P_i$ is homogeneous of degree $i$. We will show that each $P_i$ is in $I$, and this finishes the proof.

As $\mu(\alpha)\in I$, it can be written as a finite sum $\sum_{j=1}^m Q_jr_jT_j$, where $Q_j,T_j$ are homogeneous polynomials in $R$ and $r_j\in R_0$. Now $P_i$ is the sum of those $Q_jr_jT_j$ whose degree is $i$, thus $P_i$ is in $I$, as desired. 
\end{proof}
\begin{example} \label{E:SomeExamplesAboutAnnihilation}
Let $i,j\in\mathbb Z$. Let $\alpha\in N$ and $P\in R$ be a nonzero homogeneous polynomial. 
\begin{enumerate}[(1)]
   \item The relation $X_i=0$ annihilates $a_i$; the relation $X_i-X_j=0$ annihilates $a_ia_j^{-1}$.

   \item\label{Item:X0AnnihilationX1-X2} Let $\alpha=v_0^{k_1}v_1^{l_1}v_0^{k_2}v_1^{l_2}\dotsm v_0^{k_r}v_1^{l_r}$, with $k_i,l_i\in\mathbb Z$, and recall that $v_0=a_0$ and $v_1=a_2a_1^{-1}$. If $\sum_{i=1}^rl_i=0$ then $X_0=0$ annihilates $\alpha$; if $\sum_{i=1}^rk_i=0$ then $X_1-X_2=0$ annihilates $\alpha$.   
   To see this, note that if $\sum_{i=1}^rl_i=0$, then $\pi(\mu(v_0))=\pi(1)$ leads to:
   \begin{align*}    \pi(\mu(\alpha))&=\pi(1)\pi(\mu(v_1^{l_1}))\pi(1)\pi(\mu(v_1^{l_2}))\dotsm\pi(1)\pi(\mu(v_1^{l_r}))\\
       &=\pi(\mu(v_1^{l_1})\mu(v_1^{l_2})\dotsm\mu(v_1^{l_r}))\\
       &=\pi(\mu(v_1^{l_1+\dotsm+l_r}))=\pi(1).
   \end{align*}
Similarly, if $\sum_{i=1}^rk_i=0$, then $\pi(\mu(v_1))=\pi(1)$ leads to $\pi(\mu(\alpha)) = \pi(1)$.

\item \label{Item:X0Annihilation} Assume that $X_0=0$ annihilates $P$, then $P=\sum_{i=1}^mQ_iX_0R_i$ for some $Q_i,R_i\in R$.  Being a nonzero sum of monomials where $X_0$ appears in every monomial, $X_0$ must also appear in $P$.

   \item  \label{Item:Xi-X0Annihilation}Assume that $\sum_{i\in\mathbb Z}W_{S,a_i}(\alpha)=0$.  Then the relations $X_i-X_0=0$ ($i\in\mathbb Z$) annihilate $\alpha$.

   To see this, we first note that these relations annihilate each $a_ia_0^{-1}$, thus $\pi(\mu(a_ia_0^{-1}))=\pi(1)$, from which we have 
   \begin{align*}
       \pi(\mu(a_i))=\pi(\mu((a_ia_0^{-1})a_0))=\pi(1)\pi(\mu(a_0))=\pi(\mu(a_0)).
   \end{align*} Now if $\alpha=a_{i_1}^{k_1}\dotsm a_{i_n}^{k_n}$ with $k_j,i_j\in\mathbb Z$, then $\sum_{j\in\mathbb Z}k_j=0$,  and so we have 
   \begin{align*}
       \pi(\mu(\alpha))&=\pi(\mu(a_{i_1}))^{k_1}\dotsm \pi(\mu(a_{i_n}))^{k_n}\\
                       &=\pi(\mu(a_0))^{k_1}\dotsm\pi(\mu(a_{i_n}))^{k_n}\\
                       &=\pi(\mu(a_0^{k_1+\dotsm+k_n}))=\pi(1).
   \end{align*}
   \item \label{Item:AppearingX1-X2}Assume that $X_1-X_2=0$ annihilates $P$. Then both $X_1$ and $X_2$ appear in $P$. 
To see this, for every function $\sigma:\mathbb Z\rightarrow\mathbb Z$, we define a ring homomorphism $f_\sigma:R \rightarrow R$ by extending the following map on monomials of $R$ linearly to all elements of $R$:
$$f_\sigma(X_{i_1}\dotsm X_{i_k})=X_{\sigma(i_1)}\dotsm X_{\sigma(i_k)}.$$ 

Now let $I$ be the ideal of $R$ generated by $X_1-X_2$, and define $\sigma$ by $\sigma(1)=2$ and $\sigma(k)=k$ for all $k\neq1$. Then $I \subseteq \ker(f_\sigma)$, and if $X_1$ does not appear in $P \in R$, then $f_\sigma(P)=P$ so that $P \notin I$.  Similarly, $P \notin I$ if $X_2$ does not appear in $P$.

\item \label{Item:Xi-X0AnnihilationPolynomial}Assume that $X_i-X_0=0$ ($i\in\mathbb Z$) annihilate $P$. Then there exist $i,j\in \mathbb Z$ with $i\neq j$, such that $X_i,X_j$ both appear in $P$.  This follows by an argument similar to \ref{Item:AppearingX1-X2}, now with $\sigma$ defined by $\sigma(i)=0$ for all $i\in\mathbb Z$.
\end{enumerate}
\end{example}

   \begin{lemma}\label{L:X0X1X2AllAppear} 
   Let $\alpha \in N$ be nontrivial. Assume that $\alpha\in C$ and $W_{S',v_0}(\alpha)=W_{S',v_1}(\alpha)=0$. Then the variables $X_0,X_1,X_2$ all appear in $L(\alpha)$.
    \end{lemma}
    \begin{proof}
    Note that the relation $X_0=0$ annihilates $\alpha$ by Example \ref{E:SomeExamplesAboutAnnihilation}\ref{Item:X0AnnihilationX1-X2}. So it annihilates $L(\alpha)$ by Lemma \ref{L:IdealGeneratedByHomogeneousPolynomials}. Therefore $X_0$ appears in $L(\alpha)$ by Example \ref{E:SomeExamplesAboutAnnihilation}\ref{Item:X0Annihilation}.
 Similarly, the relation $X_1-X_2=0$ annihilates  $\alpha$, and so it annihilates $L(\alpha)$. Hence both $X_1$ and $X_2$ appear in $L(\alpha)$ by Example \ref{E:SomeExamplesAboutAnnihilation}\ref{Item:AppearingX1-X2}.
    \end{proof}

\begin{proposition}\label{P:One-Relator-ConjugateIntoC}
    If $\alpha\in C\backslash\{1\}$, then $
        \mathrm{NSS}_A(\{\alpha\})\cap C=\mathrm{NSS}_C(\{\alpha\})$.
    Similarly, we have $\mathrm{NSS}_B(\{\beta\})\cap D=\mathrm{NSS}_D(\{\beta\})$ for all $\beta\in D\backslash\{1\}$.
\end{proposition}

\begin{proof}
We only prove the first equality, the second follows by applying the isomorphism $\Phi:A\rightarrow B$, which sends $C$ to $D$.

    Since $\mathrm{NSS}_C(\{\alpha\}) \subseteq \mathrm{NSS}_A(\{\alpha\})\cap C$ is clear, we need only prove the other containment. 
    Let $\tau\in \mathrm{NSS}_A(\{\alpha\})\cap C$, then $\tau$ can be written in the form $\tau=\alpha^{b^{i_1}g_1}\dotsm \alpha^{b^{i_r}g_r}$, where $i_j\in\mathbb Z$, $g_j\in N$, and $r\geq1$.  We only need to prove that all these $i_j=0$, and then Lemma \ref{L:StrongMalnormalityLikeByFreiheitssatz}(2) gives $g_i\in C$ for all $i$ and thus $\tau\in\mathrm{NSS}_C(\{\alpha\})$, as desired.  Let $m$ and $M$ be the minimum and maximum of $\{i_1,\dotsc,i_r\}$ respectively, then we only need to prove that $M=m=0$.

     Since $\alpha, \tau\in C$, we know that $L(\alpha)$ is a nonzero homogeneous polynomial $P(X_0,X_1,X_2)$ and then 
     $$L(\tau)=\sum_{p=1}^rP(X_{i_p},X_{i_p+1},X_{i_p+2})$$ is also a nonzero homogeneous polynomial in $X_0,X_1,X_2$.

We first consider the case that the variables $X_0,X_1,X_2$ all appear in $L(\alpha)$. In this case, since both  $X_{M+2}$ and $X_{m}$ appear in $L(\tau)$, we have $m\geq0$ and $M\leq 0$, hence $M=m=0$, as desired.

Now we assume that not all the variables $X_0,X_1,X_2$ appear in $L(\alpha)$. Then by Lemma \ref{L:X0X1X2AllAppear}, one of $w_0:=W_{S',v_0}(\alpha)$ and $w_1:=W_{S',v_1}(\alpha)$ is nonzero. 
Thus $L(\alpha)=w_0X_0+w_1(X_2-X_1)$ by Lemma \ref{L:SimplePropertyOfLeadingTerm}\eqref{Item:Degree1LeadingTermInC}. 
Thus exactly one of $w_0$ and $w_1$ is zero and $L(\alpha)$ is either $w_0X_0$ or $w_1(X_2-X_1)$. Correspondingly, $L(\tau)$ is one of the following 
\begin{align*}
    &\sum_{p=1}^rw_0X_{i_p}=w_0\sum_{j=m}^Mn_jX_j, \quad \sum_{p=1}^rw_1(X_{i_p+2}-X_{i_p+1})=w_1\sum_{j=m}^Mn_j(X_{j+2}-X_{j+1}),
\end{align*} where $n_j=|\{p\mid i_p=j\}|$.  But $\tau\in C$, thus $L(\tau)=c_0X_0+c_1X_1+c_2X_2$, with $c_1+c_2=0$, by Lemma \ref{L:SimplePropertyOfLeadingTerm}\eqref{Item:Degree1LeadingTermInC}. If $L(\tau)$ is the first one, then $m\geq0$, $M\leq 2$, $n_1w_0+n_2w_0=0$, thus $n_1=n_2=0$ and then $M=m=0$, as desired. If $L(\tau)$ is the second one, then $m\geq-1$, $M\leq0$ and $(n_{-1}-n_0)w_1+n_0w_1=0$, which implies that $n_{-1}=0$ and thus $M=m=0$. 
\end{proof}

\begin{lemma}
\label{L:One-Relator-NotConjugateIntoC}
    Assume that $\beta \in A$ is not conjugate into $C$, and that $\mathrm{NSS}_A(\beta)\cap C\neq\emptyset$. 
     Then $W_{S,a_0}(\beta)W_{S,a_1}(\beta)<0$ and for every $\tau \in \mathrm{NSS}_A(\beta) \cap C$, there are $r,s\geq1$ such that $$(W_{S', a_0}(\tau), W_{S', a_2a_1^{-1}}(\tau))=(rW_{S,a_0}(\beta),sW_{S,a_1}(\beta)).$$
\end{lemma}
\begin{proof}
    First note that if $\beta \notin N$ then $\mathrm{NSS}_A(\beta) \cap N = \emptyset$, so that $\mathrm{NSS}_A(\beta) \cap C = \emptyset$ as well.  Hence $\beta \in N$.  Now let $\tau \in \mathrm{NSS}_A(\beta) \cap C$, 
    then we can write $\tau = \beta^{b^{k_1} h_1} \cdots \beta^{b^{k_n}h_n}$, where these $k_i \in \mathbb{Z}$, $h_i \in N$ and $n\geq1$.  Set $m = \min\{k_1, \ldots, k_n\}$, $\alpha = \beta^{b^m}$ and $j_i= k_i-m$ for all $i$, then we have
    \[ \tau = \alpha^{b^{j_1}h_1} \cdots \alpha^{b^{j_n}h_n}, \text{ with }\min\{j_1,\dotsc,j_n\}=0.
    \]
     Set $M = \max\{j_1, \ldots, j_n\}$, and note that if $M=0$ then Lemma \ref{L:StrongMalnormalityLikeByFreiheitssatz}(1) would imply $\tau \notin C$, so $M\geq 1$.  Now suppose that $ L(\alpha) = P(X_{i_1}, \ldots, X_{i_t})$ where $X_{i_1},\dotsc, X_{i_t}$ with $i_1<i_2<\dotsc<i_t$ and $t\geq1$ are the only variables appearing in $L(\alpha)$. Then 
    \[ L(\tau)=\sum_{q=1}^nP(X_{i_1+j_q},\dotsc,X_{i_t+j_q}),
    \]
and in particular both $X_{i_1}$ and $X_{i_t+M}$ appear in $L(\tau)$.  But since $\tau \in C$, only $X_0, X_1$ and $X_2$ can appear in $L(\tau)$, so we conclude that $0 \leq i_1 < i_t +M\leq 2$.  Since $M \geq 1$ this means that $t \in \{1,2\}$.
\begin{enumerate}
    \item Suppose $t=1$. Then 
    $L(\alpha) = kX_{i_1}^l$ 
    for some nonzero $k \in \mathbb{Z}$ and $l \geq 1$. 
    If $l \geq 2$, then by By Lemma \ref{L:SimplePropertyOfLeadingTerm}(2), we have $W_{S, a_i}(\alpha) = 0$ for all $i$, implying that the relations $X_i -X_0=0 \; (i \in \mathbb{Z})$ annihilate $\alpha$ by Example \ref{E:SomeExamplesAboutAnnihilation}(4).  Hence these same relations annihilate $L(\alpha)$ by Lemma \ref{L:IdealGeneratedByHomogeneousPolynomials}. On the other hand, $L(\alpha) = kX_{i_1}^l$ is not annihilated by $X_i -X_0=0 \; (i \in \mathbb{Z})$ by Example \ref{E:SomeExamplesAboutAnnihilation}(6).  So $l\geq 2$ is not possible.

    Therefore $l=1$, $L(\alpha) = kX_{i_1}$ and then $L(\tau)=k\sum_{q=1}^nX_{i_1+j_q}$.  Now, since $\tau \in C$, we know that $L(\tau)=rX_0+s(X_2-X_1)$ for some $r,s\in\mathbb Z$. As $k\neq0$ and there is one $j_p=0$, we thus have $s=0$, $i_1=0$ and then all $j_q=0$, meaning $M=0$, a contradiction.
    \item Suppose $t=2$. Then $i_1=0$, $i_2=1$, $M=1$ and $L(\alpha) = P(X_0, X_1)$ with both variables appearing in $L(\alpha)$.  Then $\{j_1,\dotsc,j_p \}=\{0,1\}$ and $L(\tau) = rL(\alpha) + sL(\alpha^b)$ for some positive integers $r, s$ with $r+s =n$. 

If $\deg(L(\alpha)) \geq 2$, then $W_{S, a_i}(\alpha) = 0$ for all $i$ and so $W_{S, a_i}(\tau) = 0$ for all $i$ as well.  Then since $\tau \in C$,  $X_0=0$ annihilates $L(\tau)$ by Example \ref{E:SomeExamplesAboutAnnihilation}(2), and so $X_0$ appears in every monomial of $L(\tau)$ by Example \ref{E:SomeExamplesAboutAnnihilation}(3).
On the other hand, $L(\tau) = rL(\alpha) + sL(\alpha^b) = rP(X_0, X_1) + sP(X_1, X_2)$, and so there is at least one monomial in $L(\tau)$ where $X_2$ appears, but $X_0$ doesn't. This is a contradiction.

Therefore $\deg(L(\alpha)) =1$. Then $L(\alpha) = w_0X_0 + w_1X_1$, where $w_i=W_{S,a_i}(\beta)$ as $\alpha$ and $\beta$ are conjugate, and $L(\tau) = r(w_0X_0+w_1X_1) + s(w_0X_1+w_1X_2)$.  But since $\tau \in C$, we know by Lemma \ref{L:SimplePropertyOfLeadingTerm}\eqref{Item:Degree1LeadingTermInC} that $(W_{S', a_0}(\tau), W_{S', a_2a_1^{-1}}(\tau))=(rw_0,sw_1)$ and $rw_1+sw_0+sw_1 = 0$, which implies that $w_0w_1<0$, as desired.
\end{enumerate}
\end{proof}

We recall notation for the next lemma. We use $B$ to denote the free group on generators $c$ and $d$, and $M$ the normal subgroup of $B$ generated by $c$. There is a free generating set $T:=\{c_i\mid i\in\mathbb Z\}$, where $c_i=c^{d^i}$ for each $i\in\mathbb Z$ for $M$, and a free generating set $T'$, which is the same as $T$ except with $c_1$ replaced by $c_2c_1^{-1}$.  The weight $W_{T,c_i} :M \rightarrow \mathbb{Z}$ is the homomorphism which sends $c_i$ to $1$ and all other $c_j$  to $0$, we similarly define $W_{T', c_i}$ and $W_{T', c_2c_1^{-1}}$. Last, $D$ is the subgroup of $M$ generated by $\{c_0,c_2c_1^{-1}\}$. 
The real numbers in $\{r_i \mid i \in I\}$ are said to be \textbf{of the same sign} if $r_i >0$ for all $i \in I$ or $r_i <0$ for all $i \in I$.

\begin{lemma}
\label{L:SamesignNSSIntersectD}
Let $D_0 \subseteq D$.  If the integers in $\{ W_{T', c_0}(\beta) \mid \beta \in D_0\}$ are of the same sign, and if the integers in $\{ W_{T', c_2c_1^{-1}}(\beta) \mid \beta \in D_0\}$ are of the same sign, then $\mathrm{NSS}_B(D_0)\cap D=\mathrm{NSS}_D(D_0)$.
\end{lemma}
\begin{proof}
    We have $\mathrm{NSS}_D(D_0) \subseteq \mathrm{NSS}_B(D_0)\cap D$ by definition, and so we only need to prove the other containment. Let $\tau = \gamma_1^{g_1}\cdots \gamma_n^{g_n} \in \mathrm{NSS}_B(D_0)\cap D$ with $\gamma_i \in D_0$ and $g_i \in B$, and we need to show that $\tau\in \mathrm{NSS}_D(D_0)$.  

    First,  by setting $g_i = d^{l_i}h_i$ for some $l_i \in \mathbb{Z}$ and $h_i \in M$, we can write
    $\tau = \gamma_1^{d^{l_1}h_1}\cdots \gamma_n^{d^{l_n}h_n}$. Letting $M = \max\{l_1, \ldots, l_n\}$ and $Z= \{ i \mid l_i = M\}$, we have
    \[ W_{T, c_{2+M}}(\tau) = \sum_{i=1}^n W_{T, c_{2+M}}(\gamma_i^{d^{l_i}h_i}) = \sum_{i \in Z}W_{T', c_{2}c_1^{-1}}(\gamma_i),
    \]
which is nonzero since it is a sum of integers having the same sign.  Since $\tau \in D$, we know $W_{T, c_i}(\tau) =0$ for $i \notin \{0,1,2\}$ and so this forces $M=0$.  Arguing similarly yields $\min \{l_1, \ldots, l_n\}=0$ as well, so that $l_i=0$ for all $i$ and $\tau = \gamma_1^{h_1} \ldots \gamma_n^{h_n}$.

    Second, $W_{T', c_0}(\tau) \neq 0$ for all $\tau \in \mathrm{NSS}_D(D_0)$,  because $W_{T',c_0}(\tau)$ is positive (resp. negative) when $W_{T', c_0}(\beta)$ is positive (resp. negative) for all $\beta \in D_0$.  In particular, $1 \notin \mathrm{NSS}_D(D_0)$.
    
    Now we have $1\not\in D':=\mathrm{NSS}_D(D_0)$, all $\gamma_i\in D_0\subseteq D'$, all $h_i\in M$ and $\tau=\gamma_1^{h_1} \ldots \gamma_n^{h_n}\in D$. Note that $D$ is multi-malnormal in $M$ by Corollary  \ref{C:FactorSubgroupsAreMulti-malnormal}, thus we know $h_i \in D$ for all $i$ by Lemma \ref{L:multi-malnormalStrongMalnormalMalnormal}(1).  But then $\tau=\gamma_1^{h_1}\dotsm \gamma_n^{h_n} \in \mathrm{NSS}_D(D_0)$, as we needed to show.
\end{proof}

\begin{proposition}\label{P:NonConjugateIntoCMatching}
    Assume that $\alpha\in A$ is not conjugate into $C$. Then $P_\alpha:=\mathrm{NSS}_A(\{\alpha\})$ and $Q_\alpha:=\mathrm{NSS}_B(\phi(P_\alpha\cap C))$ satisfy $1\not\in P_\alpha$, $1\not\in Q_\alpha$ and $\phi(P_\alpha\cap C)=Q_\alpha\cap D.$

    Similarly, if $\beta\in B$ is not conjugate into $D$, then $Q'_\beta:=\mathrm{NSS}_B(\{\beta\})$ and $P'_\beta:=\mathrm{NSS}_A(\phi^{-1}(Q'_\beta\cap D))$ satisfy that $1\not\in Q'_\beta$, $1\not\in P'_\beta$ and $\phi^{-1}(Q'_\beta\cap D)=P'_\beta\cap C$.
\end{proposition}
\begin{proof} We only prove the first statement as the second follows by symmetry.
First, if $P_\alpha\cap C=\emptyset$, then $Q_\alpha=\emptyset$ and the equality is trivially true. Hence we assume that $P_\alpha\cap C\neq\emptyset$. Now, by Lemma \ref{L:One-Relator-NotConjugateIntoC}, every $g\in P_\alpha\cap C$ satisfies $(W_{S', a_0}(g), W_{S', a_2a_1^{-1}}(g))=(rw_0,sw_1)$ for some $r,s\geq1$, where $w_i:=W_{S,a_i}(\alpha)$ satisfies $w_0w_1<0$.
    Applying the isomorphism $\phi:C\rightarrow D$ given by $\phi(a_0)=c_0^{-1}$ and $\phi(a_2a_1^{-1}a_0)=c_2c_1^{-1}c_0$, we have $\phi(a_2a_1^{-1})=c_2c_1^{-1}c_0^2$ and so we find $(W_{T', c_0}(\phi(g)), W_{T', c_2c_1^{-1}}(\phi(g))) = (-rw_0+2sw_1, sw_1)$.

  Now the numbers in $\{ W_{T', c_2c_1^{-1}}(\phi(g))\mid g\in P_\alpha\cap C\}$ are of the same sign, and so are those in $\{ W_{T', c_0}(\phi(g)) \mid g\in P_\alpha\cap C\}$.  By Lemma \ref{L:SamesignNSSIntersectD}, we have, as desired, 
    \begin{align}\label{E:QalphacapDPhiPbetacapD}
        Q_\alpha\cap D=\mathrm{NSS}_B(\phi(P_\alpha\cap C))\cap D=\mathrm{NSS}_D(\phi(P_\alpha\cap C))=\phi(P_\alpha\cap C),
    \end{align}
    where the last equality is because $P_\alpha\cap C$ is a normal subsemigroup of $C$ and $\phi:C\rightarrow D$ is an isomorphism.  Finally, we note that $1 \notin P_{\alpha}$ since $A$ is GTF,  and therefore $1 \notin Q_{\alpha}$ by \eqref{E:QalphacapDPhiPbetacapD}.
\end{proof}

\begin{theorem}\label{T:GTF-One-Relator}
    The group $G=A\ast_\phi B$ from Theorem \ref{thm:onerelatorGTFnonBO} is GTF.
\end{theorem}
\begin{proof}
    Since $C$ (resp. $D$) is RTF in $A$ (resp. $B$) by Lemma \ref{L:One-Relator-RTF}, we only need to find two families of normal subsemigroups satisfying conditions (1) and (2) of Theorem \ref{T:GTFAmalgamIffFamilies}. For $\alpha\in A\backslash\{1\}$, let $P_{\alpha}=\mathrm{NSS}_A(\{\alpha\})$ and $Q_\alpha=\mathrm{NSS}_B(\phi(P_\alpha\cap C))$. For $\beta\in B\backslash\{1\}$, let $Q'_{\beta}=\mathrm{NSS}_B(\{\beta\})$ and $P'_\beta=\mathrm{NSS}_A(\phi^{-1}(Q'_\beta\cap D))$. We check that the following families satisfy the required conditions:
    \begin{align*}
        \mathcal F_A:=\left\{P_\alpha\mid \alpha\in A\backslash\{1\}\right\}\cup \left\{P'_\beta\mid\beta\in B\backslash\{1\}\text{ is not conjugate into }D\right\},\\
        \mathcal F_B:=\left\{Q_\alpha\mid \alpha\in A\backslash\{1\}\text{ is not conjugate into }C\right\}\cup \left\{Q'_\beta\mid\beta\in B\backslash\{1\}\right\}.
    \end{align*}
 
First, no set in $\mathcal F_A$ contains $1$, by Proposition \ref{P:NonConjugateIntoCMatching}. 
Thus $A \setminus \{ 1\} = \bigcup_{P \in \mathcal F_A} P$, and similarly, $B \setminus \{1\} = \bigcup_{Q \in \mathcal F_B} Q$. Next, given $P\in\mathcal F_A$, we must find  $Q\in\mathcal F_B$ such that $\phi(P\cap C)=Q\cap D$. There are three possible cases. 
\begin{enumerate}
                \item Suppose $P=P_\alpha$ where $\alpha$ is conjugate to $\alpha' \in C$. Then $P=P_{\alpha'}$, and with $\beta':=\phi(\alpha')$ and $Q:=Q'_{\beta'}$ we have:
                \begin{align*}
                    P\cap C&=P_{\alpha'}\cap C=\mathrm{NSS}_A(\{\alpha'\})\cap C=\mathrm{NSS}_C(\{\alpha'\}),\\
                    Q\cap D&=Q'_{\beta'}\cap D=\mathrm{NSS}_B(\{\beta'\}\cap D=\mathrm{NSS}_D(\{\beta'\}).
                \end{align*} The last equality on each line above is from Proposition \ref{P:One-Relator-ConjugateIntoC}. Now as $\phi(\mathrm{NSS}_C(\{\alpha'\}))=\mathrm{NSS}_D(\{\beta'\})$, we have $\phi(P\cap C)=Q\cap D$ as required. 
                \item Suppose $P=P_\alpha$ with $\alpha\in A\backslash\{1\}$ not conjugate into $C$. Letting $Q=Q_\alpha$, we have $\phi(P\cap C)=Q\cap D$ by Proposition \ref{P:NonConjugateIntoCMatching}.
                \item Suppose $P=P'_\beta$ with $\beta\in B\backslash\{1\}$ not conjugate into $D$. Choosing $Q=Q'_\beta$, we have $\phi(P\cap C)=Q\cap D$  by Proposition \ref{P:NonConjugateIntoCMatching}.
\end{enumerate}           
Last, a symmetric argument shows that for all $Q\in\mathcal F_B$, we can find $P\in\mathcal F_A$ such that $\phi(P\cap C)=Q\cap D$ holds.  This completes the proof.
\end{proof}

\section{A group which is GTF and not left-orderable}\label{Section:GTFAndNonLOGroups}
As another application of Theorem \ref{T:GTFAmalgamIffFamilies}, in this section we construct an amalgam which is GTF and not left-orderable. Our construction is given by Definition \ref{D:DefineSubgroupC} and Theorem \ref{T:NonleftOrderableAndGTF-AmalgamatingMap}. 
 
In Subsection \ref{Subsection:ConstructionOfTheAmalgamAndNonLO}, we define an amalgam $G=A*_\phi B$ where $\phi:C\rightarrow D$ is an isomorphism between subgroups $C \subseteq A$ and $D \subseteq B$.  
We prove that $G$ is not left-orderable, and prove that its GTF assuming the results in the subsequent sections. 

In Subsection \ref{Subsection:RTFOfC} we prove that $C$ is RTF in $A$. 
We introduce the left-first product in Subsection \ref{Subsection:LFP&Originality} and study some of its properties, which are applied in Subsection \ref{Subsection:LFP-multi-malnormalityOfC} to prove that $C$ is multi-malnormal in $A$.  

\subsection{The construction of a GTF and non-left-orderable group}\label{Subsection:ConstructionOfTheAmalgamAndNonLO}
 In this subsection, we define the amalgam $G=A\ast_\phi B$ which we will prove is GTF and non-LO.
 
 In the definition below, $\mathrm{sign}(x)=\frac{x}{|x|}$ is the usual sign function of real numbers.
\begin{definition} \label{D:DefineSubgroupC} 
Let $C$ be the subgroup of the free group $A=\langle a,b\mid\rangle$ generated by $S=\{\alpha_i\mid i=1,2,\dotsc,m\}$, with 
$\alpha_i=a^{k_{i1}^{(1)}}b^{k_{i1}^{(2)}}a^{k_{i2}^{(1)}}b^{k_{i2}^{(2)}}\dotsm a^{k_{is}^{(1)}}b^{k_{is}^{(2)}}$, $m \geq 8$ and $s\geq10$, where $k_{ij}^{(t)}$ are nonzero integers satisfying the following conditions.
\begin{enumerate}[(A)]
\item \label{Item:SVeryDistinct}
    The cardinalities of the following two sets are both $sm+\binom m2$:
\begin{align*}
    &\{|k_{ij}^{(1)}|\mid 1\leq i\leq m, 1\leq j\leq s\}\cup \{|k_{i1}^{(1)}-k_{i'1}^{(1)}|\mid 1\leq i<i'\leq m\},\\
    &\{|k_{ij}^{(2)}|\mid 1\leq i\leq m, 1\leq j\leq s\}\cup \{|k_{is}^{(2)}-k_{i's}^{(2)}|\mid 1\leq i<i'\leq m\}.
\end{align*}
    
    \item \label{Item:SSameSign} For $i\in\{1,2,\dotsc,8\}$ and $t\in\{1,2\}$, $\epsilon_{it}:=\mathrm{sign}(k_{i1}^{(t)})=\mathrm{sign}(k_{i2}^{(t)})=\dotsm=\mathrm{sign}(k_{is}^{(t)}).$
    \item \label{Item:SignOfTheFirst8} For \label{Item:SPositiveNegative}$i\in\{1,2,3,4\}$, $\epsilon_{i1}=\epsilon_{i2}=1$ and for $i\in\{5,6,7,8\}$, $\epsilon_{i1}=-\epsilon_{i2}=1$.                                
\end{enumerate}
\end{definition}
Condition \ref{Item:SVeryDistinct} says that for each fixed $t \in \{1,2\}$, the $sm+\binom{m}{2}$ numbers $k_{ij}^{(t)}$ for all $i,j$ and $k_{is}^{(t)}-k_{i's}^{(t)}$ for all $i<i'$ have distinct absolute values. Conditions \ref{Item:SSameSign} and \ref{Item:SPositiveNegative} together say that all exponents of $a$ and $b$ in $\alpha_1,\alpha_2,\alpha_3,\alpha_4$ are positive, and for $\alpha_5,\alpha_6,\alpha_7,\alpha_8$, the exponents of $a$ are positive and those of $b$ are negative. 
Condition \ref{Item:SVeryDistinct} and the restriction that $s\geq10$ will be needed to show that $C$ is RTF and multi-malnormal in $A$ (see Theorem \ref{T:CIsRTFInA} and Theorem \ref{T:multi-malnormalityOfC}), while Conditions \ref{Item:SSameSign} and \ref{Item:SPositiveNegative} together with $m\geq8$ are used to show non-left-orderability of the amalgam in Theorem \ref{T:NonleftOrderableAndGTF-AmalgamatingMap}. 

That $S$ is a free generating set of $C$ is proved in Lemma \ref{L:SmallCancellationOfS}(4), but we implicitly use this fact in the statement of the next theorem, which is the main result of this section and proves Theorem   
\ref{thm:introGTFnonLO}. The details of the proof constitute the remainder of the manuscript.

\begin{theorem}\label{T:NonleftOrderableAndGTF-AmalgamatingMap}
    Assume the setup in Definition \ref{D:DefineSubgroupC}. Let $B=\langle c,d\mid\rangle$ and $\Phi:A\rightarrow B$ be the isomorphism given by $\Phi(a)=c$, $\Phi(b)=d$. Set $D=\Phi(C)$ and $\beta_i=\Phi(\alpha_i)$. 
        Let $\phi:C\rightarrow D$ be an isomorphism sending $\alpha_1,\dotsc,\alpha_8$ to 
   $\beta_1$, $\beta_2^{-1}$, $\beta_5$, $\beta_6^{-1}$, $\beta_3$, $\beta_4^{-1}$, $\beta_7$, $\beta_8^{-1}$, respectively.
Then the amalgam $G=A\ast_\phi B$ is GTF and not left-orderable. 
\end{theorem}
\begin{proof}
By Theorem \ref{T:CIsRTFInA}, $C$ is RTF in $A$, and thus $D$ is RTF in $B$. By Theorem \ref{T:multi-malnormalityOfC}, $C$ is multi-malnormal in $A$ and thus $D$ is multi-malnormal in $B$. Together with the fact that free groups are GTF, these results imply that $G$ is GTF by Corollary \ref{C:RTF+multi-malnormalityImpliesGTF}. 

To prove that $G$ is not left-orderable,  first note that none of $a, b, c, d$ is equal to the identity.  Thus by Theorem 1.48 in \cite{CR16} we only need to prove that $1$ is contained in the subsemigroup of $G$ generated by $\{a^{\epsilon_1},b^{\epsilon_2},c^{\epsilon_3},d^{\epsilon_4}\}$ for all the choices $\epsilon_i=\pm1$.  We can assume $\epsilon_1=1$ for all choices under consideration, by taking inverses.  However, by the conditions on the exponents given in Definition \ref{D:DefineSubgroupC} parts \ref{Item:SSameSign} and \ref{Item:SPositiveNegative}, the element $1=\alpha_i\phi(\alpha_i)^{-1}$ for $i=1,2,\dotsc,8$ belongs to the subsemigroup of $G$ generated by the following sets respectively:
\begin{align*}
    &\{a,b,c^{-1},d^{-1}\}, \{a,b,c,d\}, \{a,b,c^{-1},d\}, \{a,b,c,d^{-1}\},\\
    &\{a,b^{-1},c^{-1},d^{-1}\}, \{a,b^{-1},c,d\},\{a,b^{-1},c^{-1}, d\},\{a,b^{-1},c, d^{-1}\}.
\end{align*} 
\end{proof}
Now it remains to prove that $C$ is RTF and multi-malnormal in $A$. This is the task of the next three subsections. 

\subsection{The RTF-ness of \texorpdfstring{$C$}{Lg} in \texorpdfstring{$A$}{Lg}} \label{Subsection:RTFOfC}
Our goal in this subsection is to prove that $C$ in Definition \ref{D:DefineSubgroupC} is RTF in $A$. 
We view $A$ as the free product $\langle a\mid\rangle \ast\langle b\mid\rangle$ of infinite cyclic groups $\langle a\mid\rangle$ and $\langle b\mid\rangle$. 
Hence, the length $l(g)$ of $g\in A$ is defined. We remark this $l(g)$ is \textit{not} the usual length in a free group with respect to the generating set $\{a,b\}$, as we do not count generators with multiplicity. 
For example, we have $l(\alpha_i)=2s$ for all $\alpha_i$ in $S$, the generating set of $C$.

Assume that $l(g)=n$, then write $g$ as an alternating product $g=x_1\dotsm x_n$, where $x_i:=c_i^{k_i}$ with $k_i\in\mathbb Z\backslash\{0\}$, $c_i\in\{a,b\}$ and $c_i\neq c_{i+1}$. In this case, the alternating product is uniquely determined by $g$, and we also call it the \textbf{reduced form} of $g$. We call $x_i=c_i^{k_i}$ the \textbf{$i$th component} of $g$ and denoted as $B_i(g)$. For convenience, we also call $x_i$ the \textbf{$(n-i+1)$th component from RHS} and denoted as $x_i=RB_{n-i+1}(g)$. For $i\leq n$, we define the \textbf{left $i$th factor} and the \textbf{right $i$th factor} of $g$ to be $$L_i(g)=x_1x_2\dotsm x_i, \text{ and } R_i(g)=x_{n-(i-1)}\dotsm x_{n-1}x_n,$$ and define $L_0(g)=1$ and $R_0(g)=1$. For $1\leq i\leq j\leq n$, we call $x_i\dotsm x_j$ a \textbf{factor} of $g$ and denote it as $B_{[i,j]}(g)$. So for example, the left $i$th factor $L_i(g)$ is $B_{[1,i]}(g)$.

We also define the set $L_i(C)=\{L_i(c)\mid c\in C, l(c)\geq i\}$ and $R_i(C)=\{R_i(c)\mid c\in C, l(c)\geq i\}$. 
Recall that if $l(g_1\dotsm g_n)=\sum_{i=1}^nl(g_i)$ for $g_i\in A$, we say the product $g_1\dotsm g_n$ is reduced. Assume that $g_1,\dotsc,g_n\in A$ with all $l(g_i)\geq2$ for $2\leq i\leq n-1$, we say that the product $g_1\dotsm g_n$ is \textbf{almost reduced} if $l(g_ig_{i+1})\geq l(g_i)+l(g_{i+1})-1$ for all $i$ with $1\leq i\leq n-1$. 

We need some results about the structure of the elements in $C$. Recall that $S=\{\alpha_i\mid i=1,2,\dotsc, m\}$ is the generating set of $C$. 
\begin{lemma}\label{L:SmallCancellationOfS}
    Let $u_1, \dotsc, u_k\in S\cup S^{-1}$ with $k\geq1$, such that $u_iu_{i+1}\neq1$ for all $i$ with $1\leq i\leq k-1$. Then we have
    \begin{enumerate}[(1)]
    \item The product $u_1\dotsm u_k$ is almost reduced.
        \item  \label{Item:SmallCancellationOfS-TheInitialAndEndOf}$L_{2s-1}(u_1\dotsm u_k)=L_{2s-1}(u_1)$ and $R_{2s-1}(u_1\dotsm u_k)=R_{2s-1}(u_k)$.
        \item \label{Item:SmallCancellationOfS-LengthGeneralLowerBound}$l(u_1\dotsm u_k)\geq 2ks-(k-1)$.
        \item\label{Item:SmallCancelllationOfS-LengthLowerBound} $S$ is a free generating set of $C$ and $l(g)\geq 2s$ for all $g\in C\backslash\{1\}$. 
        \item \label{Item:SmallCancelllationOfS-NonSymmetric}Suppose the alternating product $A_1\dotsm A_{i+2}$ is in $L_{i+2}(C)$ 
        with $i\geq1$. Then $A_{i+2}\neq A_{i}^{-1}$. 
  \end{enumerate}   
\end{lemma}
\begin{proof}
First, we show that if $u,v\in S\cup S^{-1}$ with $uv\neq1$, then $K(u,v)=0$ and $l(uv)\geq4s-1$. 
Recall that $\alpha_i=a^{k_{i1}^{(1)}}b^{k_{i1}^{(2)}}\dotsm a^{k_{is}^{(1)}}b^{k_{is}^{(2)}}$, so if $u$ and $v$ are both from $S$, then the product $uv$ is reduced, because $R_1(u)$ is a power of $b$ and $L_1(v)$ is a power of $a$. Similarly, if $u$ and $v$ are both from $S^{-1}$, then $uv$ is reduced. If $u\in S$ and $v\in S^{-1}$, we can assume that $u=\alpha_i$ and $v=\alpha_j^{-1}$ with $i\neq j$. Then we have 
\begin{align*}
    uv&=a^{k_{i1}^{(1)}}b^{k_{i1}^{(2)}}\dotsm a^{k_{is}^{(1)}}b^{k_{is}^{(2)}}\cdot b^{-k_{js}^{(2)}}a^{-k_{js}^{(1)}}\dotsm b^{-k_{j1}^{(2)}} a^{-k_{j1}^{(1)}}\\
    &=a^{k_{i1}^{(1)}}b^{k_{i1}^{(2)}}\dotsm a^{k_{is}^{(1)}}b^{k_{is}^{(2)}-k_{js}^{(2)}}a^{-k_{js}^{(1)}}\dotsm b^{-k_{j1}^{(2)}} a^{-k_{j1}^{(1)}},
\end{align*} where the last expression is an alternating product of $4s-1$ components, as $k_{is}^{(2)}-k_{js}^{(2)}\neq0$ by our choice of exponents. The case that $u\in S^{-1}$ and $v\in S$ can be treated similarly, so that in all cases $K(u,v)=0$ and $l(uv)\geq4s-1$. 

With this result in hand, (1) follows immediately, (2) and (3) follow by induction on $k$ and (4) follows from (3). 

To prove (5), we assume that $A_{i+2}=A_{i}^{-1}$. Then $h:=B_{[i,i+2]}(c)=A_{i}A_{i+1}A_{i}^{-1}$ for a $c\in C$. Now $c$ can be written as an almost reduced product $v_1\dotsm v_k$, with each $v_j\in S\cup S^{-1}$ and each $v_iv_{i+1}\neq1$. Since $l(v_i)=2s\geq 4$, we know that $h$ is a factor of some $v_iv_{i+1}$. However, by Condition \ref{Item:SVeryDistinct}, there is not a pair of components of $v_{i}v_{i+1}$ such that they are the inverse of each other. This is a contradiction.
\end{proof}
Let $Sc(C)$ be the set of components of elements in $C$; i.e., $$Sc(C)=\{E\mid E \text{ is a component of some } c\in C\} \subset A.$$ 

\begin{definition}\label{D:Precursor} 
Define a map $p: Sc(C)\rightarrow A$ as follows.  Given $E \in Sc(C)$, there are two possibilities: 
\begin{enumerate}
\item There exists a unique $A_1\dotsm A_{2s} \in S \cup S^{-1}$ with $E = A_i$ for some $i$ with $1\leq i\leq 2s.$ Set
$$p(E)=p(A_i)=A_1\dotsm A_{i-1} \text{ for } 1\leq i\leq 2s,  (\text{with }p(A_1)=1).$$

\item There exists a unique pair $B_1\dotsm B_{2s}, A_1\dotsm A_{2s}\in S\cup S^{-1}$ with $E = B_{2s}A_1$.  Set
$$p(E) = p(B_{2s}A_1)=B_1\dotsm B_{2s-1}.$$
\end{enumerate}
We call $p(E)$ for $E\in Sc(C)$ the \textbf{prefix} of $E$ in $C$.
\end{definition}Note that the uniqueness claims in Definition \ref{D:Precursor}  follow from our choice of exponents in the elements of the basis $S$ of $C$. 
It can be verified that 
\begin{align}\label{P(E)&P(E-1)}
    p(E)E(p(E^{-1}))^{-1}\in C, \quad \forall E\in Sc(C).
\end{align}

\begin{lemma}\label{L:Prefix}
Assume that an alternating product $D_1\dotsm D_i$ (resp. $E_i\dotsm E_1$) with $1\leq i$ is in $L_i(C)$ (resp. $R_i(C)$). Then $D_1\dotsm D_{i-1}(p(D_i))^{-1}$ (resp. $p(E_i)E_iE_{i-1}\dotsm E_1$) is in $C$. 
\end{lemma}
\begin{proof} We only prove the case involving $D_1\dotsm D_i$, as the other case can be proved similarly.

Let $D_1\dotsm D_i=L_i(c)$ for some $c\in C$. There are two possible cases to consider, 
 where we write the component $D_i$ in bold and blue (so $D_i$ is $A_k$ or $B_{2s}A_1$, respectively):
$$c=g_1ug_2=g_1A_1\dotsm\boldsymbol{\textcolor{blue}{A_{k}}}\dotsm  A_{2s}g_2,\text{ or } c=g_1vug_2=g_1B_1\dotsm \dotsm  B_{2s-1}(\boldsymbol{\textcolor{blue}{B_{2s}A_1}})A_2\dotsm A_{2s} g_2,$$
 where $g_1,g_2\in C$, both $g_1ug_2$ and $g_1vug_2$ are almost reduced products, $u=A_1\dotsm A_{2s}$ and $v=B_1\dotsm B_{2s}$ are alternating products with $u, v \in S\cup S^{-1}$, $l(B_{2s}A_1)=1$ and the product $g_1A_1$ (resp. $A_{2s}g_2$) is reduced if $k=1$ (resp. $k=2s$).
 Now we have for the first case $$D_1\dotsm D_{i-1}(p(D_i))^{-1}=g_1A_1\dotsm A_{k-1}(p(A_{k}))^{-1}=g_1A_1\dotsm A_{k-1}\cdot (A_1\dotsm A_{k-1})^{-1}=g_1\in C,$$ as desired.
In the second case, we similarly find $D_1\dotsm D_{i-1}(p(D_i))^{-1}=g_1\in C$.
\end{proof}

\begin{definition}\label{D:unaltered-ProductOfTwo}
    Let $g=A_p\dotsm A_1$ and $h=B_1\dotsm B_q$ be elements of $A$ in reduced form, and $k=K(g,h)$. 
    \begin{enumerate}
        \item We say that the components $A_i$ for $1\leq i\leq k$ of $g$ and the components $B_j$ for $1\leq j\leq k$ of $h$ are \textbf{canceled} in the product $gh$. For each $i$ with $1\leq i\leq k$, we say the components $A_i$ and $B_i$ cancel (in the product $gh$).
        \item We say a component $A_i$ (of $g$) is \textbf{unaltered} (in the product $gh$) if (i) $p\geq i\geq k+2$ or (ii) $p\geq i=k+1=q+1$ or (iii) $p,q\geq i=k+1$ and $l(A_{k+1}B_{k+1})=2$ \footnote{Since we are considering the free group of rank two, this can only happen when $k=0$.}.  
        Similarly, a component $B_j$ (of $h$) is \textbf{unaltered} (in the product $gh$) if (i) $q\geq j\geq k+2$ or (ii) $q\geq j=k+1=p+1$ or (iii) $p,q\geq j=k+1$ and $l(A_{k+1}B_{k+1})=2$. 
        \item We say that the two components $A_{k+1}$ and $B_{k+1}$ \textbf{merge} into a component in the product $gh$ if $p,q\geq k+1$ and $l(A_{k+1}B_{k+1})=1$.\footnote{Note that if $p,q\geq k+1\geq2$, then $l(A_{k+1}B_{k+1})=1$, because we are working in the free group of rank 2; this doesn't follow if the rank is greater than 2.} In this case, we also say that the component $A_{k+1}$ (of $g$) is \textbf{merged with} the component $B_{k+1}$ (of $h$) to form the component $A_{k+1}B_{k+1}$ (of $gh$) and the the component $A_{k+1}B_{k+1}$ (of $gh$) is the \textbf{merge} of $A_{k+1}$ and $B_{k+1}$. 
    \item We say that a factor $B_{[i,j]}(g)$ is unaltered (resp. canceled) if all $B_t(g)$ for $i\leq t\leq j$ are unaltered (resp. canceled). Similarly, a factor $B_{[i,j]}(h)$ is unaltered (resp. canceled) if all $B_t(h)$ for $i\leq t\leq j$ are unaltered (resp. canceled). A factor $g'$ of $g$ and a factor $h'$ of $h$ is said to cancel if $l(g')=l(h')$ and for each $i$ with $1\leq i\leq l(g')$, $RB_i(g')$ and $B_i(h')$ cancel in $gh$. 
     \item A factor of $g$ or $h$ is said to be \textbf{altered} in the product $gh$ if it's not unaltered. 
     \end{enumerate}
\end{definition} Note that sometimes we use the notation $P(g,h)$ for the product $gh$. For example, $P(g_1g_2,g_3g_4)$ refers to the product $gh$, where $g=g_1g_2$ and $h=g_3g_4$. We may also just write it as the product $(g_1g_2)(g_3g_4)$ when there is no risk of confusion.  
\begin{remark}\label{R:unaltered-RightFactor-ReducedProduct}
    We remark that a factor $f$ of $h$ is unaltered in $P(g,h)$ if and only if we can write $h=h'fh''$ as a reduced product such that the product $P(gh', fh'')$ is reduced. Moreover, $h$ is a right factor of $gh$ if and only if $P(g,h)$ is reduced. (The ``if" is clear. Also, if $h$ is a right factor or $gh$, then we can write $gh$ as a reduced product $g'h$, but then $g=g'$, so $P(g,h)$ is reduced.)
\end{remark}

Note an altered component is either one that is canceled or one that is merged with another component to form a component of the product, while an unaltered component ``remains" a component of the product.
  
Let $\gamma\in A$ and $1\leq i\leq l(\gamma)$. If $L_i(\gamma)\in L_i(C)$ (resp. $L_i(\gamma)\not\in L_i(C)$), we say that $\gamma$ is left \textbf{$i$-compatible} (resp. \textbf{left $i$-incompatible}) with $C$. We say that $\gamma$ is \textbf{left compatible} (resp. \textbf{left incompatible}) with $C$ if it is left $l(\gamma)$-compatible (resp. left $l(\gamma)$-incompatible) with $C$. We similarly define right $i$-compatibility, right $i$-incompatibility and right compatibility.  For $\alpha\in A$, let $\Lambda(\alpha)$ (resp. $P(\alpha)$) be the largest integer $i$ (with $i\leq l(\alpha)$) such that $L_i(\alpha)\in L_i(C)$ (resp. $R_i(\alpha)\in R_i(C)$).  For example, we have $\Lambda(c)=l(c)=P(c)$ for $c\in C$. We also have 
\begin{align}\label{CompabilityAndCancellationNumber}
    K(g,\beta)\leq \Lambda(\beta) \text{ and } K(\beta, g)\leq P(\beta)\qquad  \forall \beta\in A, g\in C. 
\end{align} Note that $\Lambda(\alpha^{-1})=P(\alpha)$, because $(L_i(C))^{-1}=R_i(C)$. Note that if $\gamma$ is left or right $i$-incompatible with $C$, then $\gamma\not\in C$, an observation used in  Subsection \ref{Subsection:multi-malnormality}.

An element $\alpha\in A$ is called \textbf{left $C$-simplified} if either (1) $\Lambda(\alpha)< s$ or (2) $\alpha\in L_s(C)$. (Note that $\alpha\in L_s(C)$ implies that $\Lambda(\alpha)=s$.) We define right $C$-simplified similarly. We say that $\alpha$ is \textbf{$C$-simplified} if it's both left and right $C$-simplified.  
\begin{lemma}\label{L:C-simplified}
    Let $\alpha\in A\backslash C$.  
    \begin{enumerate}[(1)]
        \item There exists $c\in C$ and a left (resp. right) $C$-simplified $\alpha'$ such that $\alpha=c\alpha'$ (resp. $\alpha=\alpha'c$), $l(\alpha')\leq l(\alpha)$, and $l(\alpha')<l(\alpha)$ if $\alpha$ is not left (resp. right) $C$-simplified.
        \item There exist $c_1,c_2\in C$ and $\alpha'$ that is $C$-simplified such that $\alpha=c_1\alpha'c_2$ with $l(\alpha')\leq l(\alpha)$. 
    \end{enumerate}
\end{lemma} 
\begin{proof}
    We make induction on $l(\alpha)$. Both statements (1) and (2) are true if $l(\alpha)<s$, as we can take $c=c_1=c_2=1$. 
    
    First, we show (1), only proving the statement that is not in parentheses since the other can be proved similarly. It is clearly true if $\alpha$ is left $C$-simplified. Now, assume that $\alpha$ is not left $C$-simplified. Then $\Lambda(\alpha)\geq s$ and $\alpha\not\in L_s(C)$, and in particular $l(\alpha)>s$. There is a $c_1\in C$ such that $L_s(\alpha)=L_s(c_1)$. By \ref{Item:SmallCancellationOfS-TheInitialAndEndOf} of Lemma \ref{L:SmallCancellationOfS}, there is a $u\in S\cup S^{-1}$ such that $L_s(c_1)=L_s(u)$, and thus $L_s(\alpha)=L_s(u)$.  Letting $\alpha_1=u^{-1}\alpha$, then we have $\alpha=u\alpha_1$ and we show in the next paragraph that $l(\alpha_1)<l(\alpha)$. Thus, by induction we can write $\alpha_1=c'\alpha'$ with $c'\in C$ and $\alpha'$ left $C$-simplified and $l(\alpha')\leq l(\alpha_1)$, leading to the desired product $\alpha=(uc')\alpha'$ with $uc'\in C$ and $\alpha'$ left $C$-simplified and $l(\alpha')<l(\alpha)$. 
    
    Let $k=K(u^{-1},\alpha)$. Then $k\geq s$ as $L_s(\alpha)=L_s(u)$. If $k>s$, then we clearly have $l(\alpha_1)<l(\alpha)$ as $l(u)=2s$. 
    If $k=s$, then there is a merge of components after the cancellations, as $l(u^{-1}), l(\alpha)>s$. Thus, we have $l(\alpha_1)=l(u^{-1}\alpha)<l(\alpha)$ again. 

    Now we show (2).  By (1) there is a $c_1'\in C$, and an $\alpha_1$ which is left $C$-simplified, such that $\alpha=c_1'\alpha_1$ and $l(\alpha_1)\leq l(\alpha)$. If $\alpha_1$ is right $C$-simplified, then we are done. If not, then there is a $c_2'\in C$ and an $\alpha_2$ that is right $C$-simplified such that $\alpha_1=\alpha_2c_2'$ and $l(\alpha_2)<l(\alpha_1)$. Now we have $\alpha=c_1'\alpha_2c_2'$ and $l(\alpha_2)<l(\alpha)$. Using induction, there are $c_1'',c_2''\in C$, and $\alpha'$ that is $C$-simplified, such that $\alpha_2=c_1''\alpha'c_2''$ with $l(\alpha')\leq l(\alpha_2)$. Finally, we have $\alpha=(c_1'c_1'')\alpha'(c_2''c_2')$ where $c_1'c_1''$ and $c_2''c_2'$ are in $C$, $\alpha'$ $C$-simplified and $l(\alpha')\leq l(\alpha)$.
\end{proof}

The next result follows from Lemma \ref{L:Prefix}.
\begin{corollary}\label{C:Incompatibility&CancellationNumber}
    Let $\beta\in A\backslash C$, $h$ be a left factor of an element in $C$, $k=K(\beta,h)$ and $n=l(\beta)$. Then $\beta h$ is not left $(n-k\!+1)$-compatible with $C$. If moreover $k\leq l(h)-2$, then $\beta h$ is $(n-k+1)$-incompatible with $C$. 
\end{corollary}
\begin{proof}
    Let $\beta=B_n\dotsm B_k\dotsm B_1$ and $h=A_1\dotsm A_k\dotsm A_l$ be alternating products. Then we have that $B_k\dotsm B_1A_1\dotsm A_k=1$ and $\beta h$ is one of the following alternating products:
    \begin{align*}
        B_nB_{n-1}\dotsm B_{k+1}A_{k+1}\dotsm A_{l-1}A_l, \text{ or } B_nB_{n-1}\dotsm B_{k+2}(B_{k+1}A_{k+1})A_{k+2}\dotsm A_l, 
    \end{align*} where the first one can occur when $k=0$ or $k=n$. 
    
Suppose that $\beta h$ is left $(n\!-\!k\!+\!1)$-compatible with $C$ and we will find a contradiction. Using Lemma \ref{L:Prefix}, we have $c_i:=A_1A_2\dotsm A_i(p(A_{i+1}))^{-1}\in C$ for $0\leq i\leq l$. 
(1) If $\beta h$ is the first product, then we use Lemma \ref{L:Prefix} to arrive at the containment:
\begin{align*}
    C\ni B_nB_{n-1}\dotsm B_{k+1}(p(A_{k+1}))^{-1}=B_nB_{n-1}\dotsm B_{k+1}(A_1\dotsm A_k)^{-1}c_k=\beta c_k,
\end{align*}thus $\beta\in C$, a contradiction. (2) If $\beta c$ is the second product, using Lemma \ref{L:Prefix} we have 
\begin{align*}
    C\ni& B_n\dotsm B_{k+2}B_{k+1}A_{k+1}(p(A_{k+2}))^{-1}=B_n\dotsm B_{k+2}B_{k+1}A_{k+1}(A_1\dotsm A_{k+1})^{-1}c_{k+1}=\beta c_{k+1},
\end{align*} again concluding $\beta\in C$, a contradiction. 

If moreover $k\leq l(h)-2$, then $l(\beta h)\geq n-k+1$, so now $\beta h$ is $(n\!-\!k\!+\!1)$-incompatible with $C$.  
 \end{proof}

\begin{lemma}\label{L:betauCancellationnumber}
Let $\beta\in A\backslash C$, $g\in C$ and $h$ be a left factor of an element in $C$. Then we have 
$$K(g,\beta h)\leq K(g,\beta)+1.$$
\end{lemma}
\begin{proof}Let $i=K(g, \beta)$, assume that $i':=K(g, \beta h)\geq i+2$ and we will find a contradiction. Note that $i$ (resp. $i'$) is maximal such that $L_i(\beta)=L_i(g^{-1})$ (resp. $L_{i'}(\beta h)=L_{i'}(g^{-1})$); in particular, $\beta h$ is left $(\!i\!+\!2)$-compatible with $g$. However, letting $n=l(\beta)$ and $k=K(\beta,h)$, $\beta h$ is not left $(n\!-\!k\!+\!1)$-compatible with $C$ by Corollary \ref{C:Incompatibility&CancellationNumber}. Hence $n-k+1\geq i+3$, leading to $k\leq n-i-2$ and thus $L_{i+1}(\beta)$ is unaltered in $P(\beta,h)$. Therefore, $L_{i+1}(\beta)=L_{i+1}(\beta h)=L_{i+1}(g^{-1})$, contradicting the maximality of $i$.
\end{proof}

\begin{lemma}\label{L:CIsIsolatedAndLambdaPower}
Let $\alpha\in A\backslash C$, $g\in C$ and $n\geq2$. Then  $\alpha^n\in A\backslash C$ and 
$K(g,\alpha^n)\leq K(g,\alpha)+1$.
\end{lemma}
\begin{proof}Let $k=K(g,\alpha)$, then $k$ is maximal such that $L_k(\alpha)=L_k(g^{-1})$. 
Write $\alpha$ as an alternating product $$\alpha=A_1^{-1}\dotsm A_r^{-1}B_1\dotsm B_tA_r\dotsm A_1,$$ where $r\geq0$, $t\geq1$ and $B_1B_t\neq1$. (Also, $l(B_1B_t)=1$ if $r>0$.) Then we can write $\alpha^n$ as one of the following alternating products:
    \begin{align*}
       &A_1^{-1}\dotsm A_r^{-1}(B_1^n) A_r\dotsm A_1 \quad (\text{when } t=1),\\\nonumber
        &B_1\dotsm B_tB_1\dotsm B_t\dotsm B_1\dotsm B_t \quad (\text{when } t\geq2,\; l(B_1B_t)=2),\\\nonumber
        &A_1^{-1}\dotsm A_r^{-1}B_1B_2\dotsm B_{t-1}(B_tB_1)B_2\dotsm (B_tB_1)B_2\dotsm B_t A_r\dotsm A_1 \quad (\text{when } t\geq2,  \;l(B_1B_t)=1).
    \end{align*}
    Assume that $\alpha^n$ is the first alternating product. If $r=0$, then $l(\alpha^n)=1$ and thus $K(g,\alpha^n)\leq 1\leq k+1$, and $\alpha^n\not\in C$ by \ref{Item:SmallCancelllationOfS-LengthLowerBound} of Lemma \ref{L:SmallCancellationOfS}. If $r\geq1$, then by \ref{Item:SmallCancelllationOfS-NonSymmetric} of Lemma \ref{L:SmallCancellationOfS}, we know that $\alpha^n\not\in C$ and $ A_1^{-1}\dotsm A_r^{-1}(B_1^m) A_r$ is not $L_{r+2}(g^{-1})$. Hence $K(g,\alpha^n)\leq r+1$. Hence $K(g,\alpha^n)\leq k+1$ if $k\geq r$. If $k<r$, then $K(g,\alpha^n)=k$ as $L_{r}(\alpha^n)=L_r(\alpha)$. So we always have $K(g,\alpha^n)\leq k+1$, as desired. 
    
  From now on we only consider the other two cases. We first show by contradiction that $L_{r+t+1}(\alpha^n)\not\in L_{r+t+1}(C)$. Suppose that $L_{r+t+1}(\alpha^n)\in L_{r+t+1}(C)$. 
    \begin{enumerate}       
        \item Suppose $\alpha^n$ is the second product. Then $r=0$ and $B_1\dotsm B_tB_1\in L_{t+1}(C)$.  Hence $B_1\in L_1(C)$, and by Lemma \ref{L:Prefix},  $\alpha=B_1\dotsm B_t=B_1\dotsm B_t(p(B_1))^{-1}\in C$, a contradiction. 
        \item Suppose $\alpha^n$ is the third product. Recalling that $l(B_tB_1)=1$, by Lemma \ref{L:Prefix} we have 
        $$h_1:=A_1^{-1}\dotsm A_r^{-1}B_1\dotsm B_{t-1}B_tB_1(p(B_2))^{-1}\in C, \quad h_2:=A_1^{-1}\dotsm A_r^{-1}B_1(p(B_2))^{-1}\in C,$$
        from which we find that $h_1h_2^{-1}=\alpha$ is in $C$, again a contradiction.
         \end{enumerate}
         
         It now follows that $\alpha^n\not\in C$, as $l(\alpha^n)\geq r+t+1$.  Suppose $K(g,\alpha^n)\geq k+2$.  Then as $L_{r+t-1}(\alpha)=L_{r+t-1}(\alpha^n)$, this forces $k \geq r+t-1$, so $K(g,\alpha^n)\geq r+t+1$.  But then $L_{r+t+1} (\alpha^n)=L_{r+t+1}(g^{-1})\in L_{r+t+1}(C)$, a contradiction.  Thus $K(g,\alpha^n)\leq k+1$, as desired.
         \end{proof}

\begin{corollary}\label{C:Khalpham&Kalpha}
    Let $\alpha\in A\backslash C$, $g\in C$, $n\geq1$. If $h$ is a left factor of an element in $C$, then 
    \begin{align*}
        K(g, \alpha^nh)\leq K(g,\alpha)+2\leq \Lambda(\alpha)+2.
    \end{align*}
\end{corollary}
\begin{proof}
    Applying Lemma \ref{L:CIsIsolatedAndLambdaPower} we know $\alpha^n \in A \setminus C$  and so Lemma \ref{L:betauCancellationnumber} with $\alpha^n$ in place of $\beta$ gives $K(g, \alpha^n h) \leq K(g, \alpha^n) +1$.  Then Lemma \ref{L:CIsIsolatedAndLambdaPower} gives $K(g, \alpha^n) +1 \leq K(g, \alpha)+2$. That $K(g, \alpha)+2 \leq \Lambda(\alpha)+2$ is by \eqref{CompabilityAndCancellationNumber}.
\end{proof}

\begin{lemma}\label{L:TwoSidedCancellationWithGenerators}
    Let $\alpha\in A\backslash C$ be $C$-simplified, 
    $h_2$ (resp. $h_1$) be a right (resp. left) factor of an element in $C$, 
    $g\in S\cup S^{-1}$ and $m\geq1$ be an integer. Write $g=A_1\dotsm A_{2s}$ as an alternating product and assume that $i:=K(\alpha,g)$ and $j:=K(g,\alpha)$ are both nonzero.  Then we have:
\begin{enumerate}[(1)]
     \item  \label{Item:TwoSidedancallationWithGenerators-ijlalpha}
        $l(\alpha)$ is even, the product $\alpha^m$ is reduced, $R_{l(\alpha)}(\alpha^m)=\alpha$, $K(h_2,\alpha)\leq j$ and $i,j\leq s-1$. 
     \item \label{Item:TwoSidedCancellationWithGenerators-ExplicitFormOfAlpha} If $K(h_2\alpha^m,g)>i$ or $K(g,\alpha^mh_1)>j$, then $m=1$ and $\alpha$ is the alternating product 
         \begin{align}\label{E:TwoSidedancellationWithGenerators-ExplicitFormOfAlpha}
             A_{2s}^{-1}\dotsm A_{2s-j+1}^{-1}(A_{2s-j}^{-1}A_{i+1}^{-1})A_i^{-1}\dotsm A_1^{-1};
         \end{align}
     \item \label{Item:Khalphamg+1TwosidedCancellationWithGenerators}$K(h_2\alpha^m,g)\leq i+1$ and $K(g,\alpha^mh_1)\leq j+1$.
\end{enumerate}
\end{lemma}

\begin{proof} 
\begin{enumerate}[(1)]
\item Since $i,j$ are nonzero and $l(A_1A_{2s})=2$, we conclude $l(B_1(\alpha)R_1(\alpha))=2$ and so $l(\alpha)$ is even. Hence, the product $\alpha^m=\alpha\dotsm\alpha$ is reduced and thus $R_{l(\alpha)}(\alpha^m)=\alpha$.  

Note that $i,j\leq s$ as $\alpha$ is $C$-simplified. If $i=s$, $\alpha=(L_s(g))^{-1}$ since $\alpha$ is $C$-simplified, and thus $L_1(\alpha)=(R_1(\alpha^{-1}))^{-1}=(B_s(g))^{-1}\neq (R_1(g))^{-1}$ by Condition \ref{Item:SVeryDistinct}, contradicting to $j>0$. Hence $i\leq s-1$, and similarly $j\leq s-1$. We also have $\Lambda(\alpha)=j$ and hence $K(h_2,\alpha)\leq j$. 

Since $A_1,\dotsc,A_{2s}$ are distinct,  $L_j(\alpha)=(R_j(g))^{-1}$ and $R_i(\alpha)=(L_i(g))^{-1}$ don't have common components, hence $\alpha$ can be written as alternating product
    $$\alpha=A_{2s}^{-1}\dotsm A_{2s-j+1}^{-1}B_1\dotsm B_t A_i^{-1}\dotsm A_1^{-1}, \text{ with }t\geq0.$$ 
 \item  \label{Item:TwoSidedCancellationWithGenerators-ProvingTheExplicitExpression} Assume that $K(h_2\alpha^m,g)\geq i+1$. Then $RB_{i+1}(\alpha^m)=RB_{i+1}(\alpha)$ is altered in $P(h_2,\alpha^m)$. But this can't happen if $m\geq2$ as $K(h_2,\alpha^m)=K(h_2,\alpha)\leq j$. Hence $m=1$ and $R_{i+1}(h_2\alpha)=A_{i+1}^{-1}\dotsm A_1^{-1}$. Note that $h_2$ is a right factor of an element in $C$ and the component $A_{2s}^{-1}$ of $\alpha$ is canceled in $P(h_2, \alpha)$. Therefore, letting $j'=\min\{l(h),2s-1\}$, we have $R_{j'}(h_2)=A_{2s-j'+1}\dotsm A_{2s}$. 
 Now, if $t=0$ or $t\geq2$, we can't have $R_{i+1}(h_2\alpha)=A_{i+1}^{-1}\dotsm A_1^{-1}$. So $t=1$ and we also need $2s-j'+1\leq 2s-j$ and $A_{2s-j}B_1=A_{i+1}^{-1}$, from which $B_1=A_{2s-j}^{-1}A_{i+1}^{-1}$ follows.
  This proves \ref{Item:TwoSidedCancellationWithGenerators-ExplicitFormOfAlpha} when $K(h_2\alpha^m,g)\geq i+1$. We treat the case $K(g,\alpha^mh_1)\geq j+1$ similarly and reach the same conclusion.
 
\item    Supposing $K(h_2\alpha^m,g)\geq i+2$, we will find a contradiction. (This proves the first inequality; the second is done similarly.) As we still have $K(h_2\alpha^m,g)\geq i+1$, we have $m=1$ by the the argument of \ref{Item:TwoSidedCancellationWithGenerators-ProvingTheExplicitExpression}. But $K(h_2\alpha,g)\leq K(\alpha,g)+1=i+1$ by (the other version of) Lemma \ref{L:betauCancellationnumber}, a contradiction as desired.
    \end{enumerate}
\end{proof}

The next result helps handle a special case in the proof of the RTF-ness of $C$ in $A$.
\begin{proposition}\label{P:WhenTheTwoSidedCancelationIsNotControled}
    Let $\alpha\in A\backslash C$ be $C$-simplified, $g\in C\backslash\{1\}$, $m,n\geq1$ and $h_2$ (resp. $h_1$) be a right (resp. left) factor of an element in $C$. Assume that 
    \begin{align}\label{Ine:WhenTheTwoSidedCancelationIsNotControled-TwoMuchCancellation}
        K(h_2\alpha^m,g)+K(g,\alpha^nh_1)>l(g)-2.
    \end{align} Then $g\in S\cup S^{-1}$ and $R_s(g)\alpha L_s(g)$ is the component $B_s(g)$ or the component $B_{s+1}(g)$.
\end{proposition}
\begin{proof} First $\Lambda(\alpha), P(\alpha)\leq s$ as $\alpha$ is $C$-simplified. Then $i,j\leq s$ by \eqref{CompabilityAndCancellationNumber}, where  $i:=K(\alpha,g)$, $j:=K(g,\alpha)$. Then we proceed to prove the following.
\begin{enumerate}[(1)]
    \item \label{Item:giinscups-1}$g\in S\cup S^{-1}$ and so the RHS of \eqref{Ine:WhenTheTwoSidedCancelationIsNotControled-TwoMuchCancellation} is $2s-2$.
    
To show this, suppose $g\not\in S\cup S^{-1}$.  Then $l(g)\geq 4s-1$ by \ref {Item:SmallCancelllationOfS-LengthLowerBound} and \ref{Item:SmallCancellationOfS-LengthGeneralLowerBound} of Lemma \ref{L:SmallCancellationOfS}. By Corollary \ref{C:Khalpham&Kalpha} (and its ``right-sided" version), 
     we get
    \[ K(h_2\alpha^m,g)+K(g,\alpha^nh_1) \leq P(\alpha) + 2  + \Lambda(\alpha)+2 =2s+4.
    \]
 Now \eqref{Ine:WhenTheTwoSidedCancelationIsNotControled-TwoMuchCancellation} gives $2s+4>4s-3$ or $2s<7$, a contradiction since $s \geq 10$.  

   \item \label{Item:WhenTheTwoSidedCancelationIsNotControled-Nonzero} Both $i$ and $j$ are nonzero. (Recall that $i:=K(\alpha,g)$ $j:=K(g,\alpha)$.)

   Because if one of them is zero, then the LHS of \eqref{Ine:WhenTheTwoSidedCancelationIsNotControled-TwoMuchCancellation} is at most $s+4$ by Corollary \ref{C:Khalpham&Kalpha}, the RHS is $2s-2$ by \ref{Item:giinscups-1}, leading to $s<6$, again a contradiction to $s \geq 10$. 
   \item\label{Item:leqs-1} $l(\alpha)$ is even and $i,j \leq s-1$. 
   This now follows from Lemma \ref{L:TwoSidedCancellationWithGenerators}\ref{Item:TwoSidedancallationWithGenerators-ijlalpha}.

    \item $i,j\geq s-2$. To prove this, suppose $i \leq s-3$.  Then by Lemma \ref{L:TwoSidedCancellationWithGenerators}\ref{Item:Khalphamg+1TwosidedCancellationWithGenerators} and  \ref{Item:leqs-1}, we have
\[K(h_2\alpha^m,g)+K(g,\alpha^nh_1) \leq i + 1 + j+1 \leq (s-3)+1+(s-1)+1=2s-2,
\]
 contradicting \eqref{Ine:WhenTheTwoSidedCancelationIsNotControled-TwoMuchCancellation} by \ref{Item:giinscups-1}.  Similarly, we show that $j \leq s-3$ is not possible.

    \item $K(h_2\alpha^{m}, g)> i$ or $K(g, \alpha^n h_1)> j$. Because otherwise we have, using \ref{Item:leqs-1} and \ref{Item:giinscups-1},
     \begin{align*}
    K(h_2\alpha^m, g)+K(g, \alpha^nh_1)\leq s-1+s-1=2s-2=l(g)-2,
     \end{align*} contradicting \eqref{Ine:WhenTheTwoSidedCancelationIsNotControled-TwoMuchCancellation}.
\end{enumerate}
     We now can apply Lemma \ref{L:TwoSidedCancellationWithGenerators}\ref{Item:TwoSidedCancellationWithGenerators-ExplicitFormOfAlpha} to conclude that
     $\alpha$ is the alternating product $$A_{2s}^{-1}\dotsm A_{2s-j+1}^{-1}(A_{2s-j}^{-1}A_{i+1}^{-1})A_i^{-1}\dotsm A_1^{-1},$$ where $A_1\dotsm A_{2s}$ is the reduced form of $g$. In particular, we have $l(\alpha)= i+j+1$. Since $l(\alpha)$ is even and $i,j\in\{s-2,s-1\}$,  we have $\{i,j\}=\{s-1,s-2\}$. First, consider $i=s-1$. Then $\alpha$ can be written as the alternating product $\alpha=A_{2s}^{-1}\dotsm A_{s+2}^{-1}(A_{s+1}^{-1}A_{s-1}^{-1})A_{s-2}^{-1}\dotsm A_1^{-1}$ from which we have $R_s(g)\alpha L_s(g)=A_{s}$, as desired. Similarly, we find $R_s(g)\alpha L_s(g)=A_{s+1}$ when $i=s-2$. 
\end{proof}

We are now ready to prove the main result of this subsection.
\begin{theorem} \label{T:CIsRTFInA}
The subgroup $C$ of $A$ given by Definition \ref{D:DefineSubgroupC} is RTF in $A$. \footnote{In fact, here we only need condition \ref{Item:SVeryDistinct} on the exponents $k_{ij}^{(t)}$; conditions \ref{Item:SSameSign}, \ref{Item:SignOfTheFirst8} and that $m\geq8$ are not required for this proof.}
\end{theorem}
\begin{proof}
Let $\beta\in A\backslash C$. We need to prove that $T:=h_0\beta h_1\beta h_2\dotsm \beta h_k\neq1$ for all $h_i\in C$ and $n\geq1$. By Lemma \ref{L:C-simplified}, we can write $\beta=c_1\alpha c_2$, where $c_i\in C$ and $\alpha$ is $C$-simplified. Plugging this into the expression of $T$, we can rewrite $T$ in the following form 
\begin{align}\label{E:RTF-Target}
    T=g_0\alpha^{m_1}g_1\alpha^{m_2}g_2\dotsm \alpha^{m_{n-1}}g_{n-1}\alpha^{m_n}g_n
\end{align} where $g_i\in C$, $m_i\geq1$, $n\geq1$ and $g_1,g_2,\dotsc,g_{n-1}$ are nontrivial.  
Applying Corollary \ref{C:TwosidedCancellationAndSumOfCancellationLengths} to \eqref{E:RTF-Target} and considering $\alpha_i$ in the corollary be $\alpha^{m_i}$, we know that $T\neq1$ if the conditions of the corollary are satisfied. Therefore, we only need to consider the case when there is an $i$ with $1\leq i\leq n-1$, a right factor $g_{i-1,2}'$ of $g_{i-1}$ and a left factor $g_{i+1,1}'$ of $g_{i+1}$,
such that 
\begin{align}\label{I:RTFLeqs+2CA}
    K(g_{i-1,2}'\alpha^{m_i}, g_i)+K(g_i, \alpha^{m_{i+1}}g_{i+1,1}')>l(g_i)-2.
\end{align}
Then by Proposition \ref{P:WhenTheTwoSidedCancelationIsNotControled}, we have $g_i\in S\cup S^{-1}$ and $\gamma:=R_s(g_i)\alpha L_s(g_i)$ is $B_s(g_i)$ or $B_{s+1}(g_i)$. 

Let $\tilde C$ be the subgroup of $A$ generated by
     $\tilde S:=(S\backslash\{g_i,g_i^{-1}\})\cup\{L_s(g_i),R_s(g_i)\}$. Then $C$ is a subgroup of $\tilde C$. Plugging $\alpha=(R_s(g_i))^{-1}\gamma(L_s(g_i))^{-1}$ into the expression \eqref{E:RTF-Target}, we can write $T$ in the following form 
     \begin{align}\label{E:RTF-Target'}
     T=h_0\gamma^{n_1}h_1\gamma^{n_2}h_2\dotsm \gamma^{n_r}h_r
 \end{align} where $h_i\in \tilde C$, $n_i\geq1$, $r\geq1$ and $h_1,\dots,h_{r-1}$ are nontrivial. We want to prove that $T\neq1$. This can be derived from another application of Corollary \ref{C:TwosidedCancellationAndSumOfCancellationLengths}, as follows, using that $l(\gamma)=1$.
 \begin{enumerate}[(1)]
     \item \label{Item:TheLowerBoundForLengthInTheExtendedGroup} First we have $l(g)\geq s$ for all $g\in \tilde C\backslash\{1\}$.  This is because all elements in the generating set $\tilde S$ have length at least $s$ and that $K(u,v)=0$ for all $u,v\in \tilde S\cup \tilde S^{-1}$ with $uv\neq1$. 
     \item  \label{Item:TheLocalPropertyOfTheExtendedGroup} If the alternating product $B_i\dotsm B_1$ (resp. $D_1\dotsm D_j$) is a right (resp. left) factor of an element in $\tilde C$, with $i,j\leq s-1$ and $B_iD_j=1$, then $i=j$ and $B_kD_k=1$ for all $k\leq i$. To see this, note that  $i,j\leq s-1$ implies there are $u,v\in \tilde S \cup \tilde S^{-1}$ such that $B_{i}\dotsm B_1=R_i(u)$ and $D_1\dotsm D_j=L_j(v)$. Hence $B_1^{-1}\dotsm B_i^{-1}=L_i(u^{-1})$. Since $B_i^{-1}=D_j$, by Definition \ref{D:DefineSubgroupC}\ref{Item:SVeryDistinct}, we know that $u^{-1}=v$ and also $i=j$. Then $B_k^{-1}=D_k$ for all $k\leq i$ follows, as desired. 

     \item  \label{Item:TheBoundForCancellationNumberWhenLenghtIs1}For any right (resp. left) factor $h$ of an element in $\tilde C$, for every $g\in\tilde C\backslash\{1\}$, and every positive integer $m$, we have $K(h\gamma^m,g)\leq 1$ (resp. $K(g,\gamma^mh)\leq1$). 

     We only prove the case when $h$ is a right factor. Assume that $K(h \gamma^m,g)\geq 2$ and we will find a contradiction. 
     Let $h=B_r\dotsm B_1$ and $g=D_1\dotsm D_t$ be alternating products. 
     Then $l(h\gamma^m), t\geq2$ and $h\gamma^m$ is one of the following alternating products:
     \begin{align*}
         B_r\dotsm B_1(\gamma^m), \; B_r\dotsm B_2(B_1\gamma^m), \text{ or } B_r\dotsm B_2,
     \end{align*} which leads to $B_1D_2=1$,  $B_1\gamma^mD_1=1$ and $B_2D_2=1$, and $B_2D_1=1$, respectively. But none of these is possible by \ref{Item:TheLocalPropertyOfTheExtendedGroup}. 

     \item Now, for each $i$ with $1\leq i\leq r-1$, every right factor $h_{i-1,2}$ of $h_{i-1}$ and every left factor $h_{i+1,1}$ of $h_{i+1}$,  using \ref{Item:TheBoundForCancellationNumberWhenLenghtIs1}, $s\geq4$ and \ref{Item:TheLowerBoundForLengthInTheExtendedGroup}, we have
     \begin{align*}
         K(h_{i-1,2}\gamma^{n_i}, h_i)+K(h_i,\gamma^{n_{i+1}}h_{i+1,1})\leq 2\leq s-2\leq l(h_i)-2.
     \end{align*}
     With this, it follows from Corollary \ref{C:TwosidedCancellationAndSumOfCancellationLengths} that $T$ in \eqref{E:RTF-Target'} is nontrivial. 
 \end{enumerate}
\end{proof}

\subsection{Left-first products and the origin of components}\label{Subsection:LFP&Originality}
Having completed the proof of the RTF-ness of $C$ in $A$, we now want to deal with the multi-malnormality of $C$ in $A$. In this subsection, we provide some tools and results, which will be applied in the next subsection in the proof that $C$ is multi-malnormal in $A$. 

Consider a product $c_1^{g_1} \cdots c_n^{g_n}$ with $c_i \in C$ and $g_i \in A \setminus C$.  Our primary technique will be to express $c_1^{g_1} \cdots c_n^{g_n}$ as an alternating product, and then use the exponent conditions of Definition \ref{D:DefineSubgroupC} to limit cancellation.  However, in order to use these exponent conditions, we must keep track of which components ``come from" the $c_i$'s, which will be accomplished in Subsection \ref{Subsection:LFP-multi-malnormalityOfC} using the tools of this section. 

Let us first look at an example to illustrate the subtlety of this task.  Consider a product $g_1g_2g_3$ with $g_1=a,g_2=a^{-1},g_3=a$, then their product is $g_1g_2g_3=a$. Where the component $a$ ``comes from" is ambiguous: it ``comes from" $g_1$ if we first compute $g_2g_3$ and then $g_1(g_2g_3)$; it ``comes from" $g_3$ if we compute $g_1g_2$ first and then calculate $(g_1g_2)g_3$. Therefore, the order of the multiplication matters when considering the origin of the components appearing in the final alternating product (even though the product is independent of the order in which we carried out the multiplications).

Recall that in Definition \ref{D:unaltered-ProductOfTwo}, we defined the unalteredness and cancellation of components and factors in a product $gh$ of two elements $g$ and $h$. We now generalize these to a product $g_1\dotsm g_n$. Note that if a component $A_i$ of $g$ is unaltered in $gh$, then we can view $A_i$ as a component of $gh$, and doing so allows us to talk about whether $A_i$ is unaltered in $(gh)f$ for $f\in A$.  This iterative approach to unaltered components is key to the definition below. 
\begin{definition}\label{D:lFP}
    Let $g_1,\dotsc,g_n\in A$ and $P_t=g_1g_2\dotsm g_t$ for $1\leq t\leq n$. We call the tuple $(P_1,P_2,\dotsc,P_n)$ the \textbf{left-first product} of $(g_1,\dotsc,g_n)$ and denote it as $LFP(g_1,\dotsc,g_n)$. 
    \begin{enumerate}
    \item Let $g_i'$ be a factor of $g_i$ with $1\leq i\leq n$. We say that $g_i'$ is \textbf{unaltered} in $LFP(g_1,\dotsc,g_n)$ if it's unaltered in $P_{j-1}g_j$, for $j=i,i+1,\dotsm,n$. (Here, we use that when $g_i'$ is unaltered in $P_{j-1}g_j$ we may view it as a factor of $P_j$.)  
    We say that $g_i'$ is \textbf{altered} in $LFP(g_1,\dotsc,g_n)$ if it's not unaltered in $LFP(g_1,\dotsc,g_n)$. 
    \item Let $g_i'$ (resp. $g_j'$) be a factor of $g_i$ (resp. $g_j$) with $1\leq i< j\leq n$. We say that $g_i'$ and $g_j'$  \textbf{cancel} in $LFP(g_1,\dotsc,g_n)$ if $g_i'$ is unaltered in $LFP(g_1,\dotsc,g_{j-1})$ and $g_i'$ and $g_j'$ cancel in $P_{j-1}g_j$. 
    \end{enumerate}
\end{definition}

Similarly, we define the right-first product $RFP(g_n,\dotsc,g_1)$ of the tuple $(g_n,\dotsc,g_1)$ to be $(P_1,P_2,\dotsc,P_n)$, where $P_i=g_i\dotsm g_1$; a factor $g_{i}'$ of $g_i$ is unaltered in $RFP(g_n,\dotsc,g_1)$ if it's unaltered in $g_jP_{j-1}$ for $j=i,i+1,\dotsc,n$; and a factor $g_j'$ of $g_j$ and a factor $g_{i}'$ of $g_i$ with $n\geq j>i\geq1$ cancel in $RFP(g_n,\dotsc,g_1)$ if $g_i'$ is unaltered in $RFP(g_{j-1},\dotsc,g_1)$ and $ g_{j}', g_{i}'$ cancel in $g_jP_{j-1}$.  Moreover, for convenience,  we say that $g_i'$ is \textbf{canceled by} $g_j''$ in $LFP(g_1,\dotsc,g_{n})$ (resp. $RFP(g_n,\dotsm, g_1)$) if $g_i'$ and a factor $g_j'$ of $g_j''$ cancel in $LFP(g_1,\dotsc,g_{n})$ (resp. $RFP(g_n,\dotsm, g_1)$).

The rest of this section presents several results concerning left-first products and right-first products. Note that for each result involving the left-first product, there is an analogous result involving the right-first product and vice versa. For brevity, we usually only state and prove one of the two versions for each pair of results. 

For two factors $f_1:=B_{[i_1,j_1]}(g)$, $f_2:=B_{[i_2,j_2]}(g)$ of $g$ with $1\leq i_1\leq j_1<i_2\leq j_2\leq l(g)$, we say that $f_2$ is \textbf{to the RHS} of $f_1$ in $g$. 
The next result is clear and follows by induction. 
\begin{lemma}\label{L:RelativePositionOfUnalteredSections}
    Consider the product $g=g_1g_2\dotsm g_n$. Let $f_i,f_j$ be factors of $g_i,g_j$ respectively with $1\leq i<j\leq n$. 
    Assume that $f_i,f_j$ are unaltered in $LFP(g_1,\dotsc,g_n)$. Then $f_j$ is to the RHS of $f_i$ in $g$.  
\end{lemma}

\begin{lemma}\label{L:PairCancellationToATrivialProduct}
    Let $g_1,\dotsc,g_n\in A$ with $n\geq 2$. Assume that the component $RB_r(g_1)$ of $g_1$ and the component $B_t(g_n)$ of $g_n$ cancel in $LFP(g_1,\dotsc,g_n)$, then we have $RB_r(g_1)B_t(g_n)=1$ and $R_r(g_1)g_2\dotsm g_{n-1}L_t(g_n)=1$.
\end{lemma}
\begin{proof}
   We use induction on $n$. For $n=2$ it is true by definition, so consider $n\geq 3$. Note that $RB_r(g_1)$ is unaltered in $g_1g_2$ as it's unaltered in $LFP(g_1,\dotsc,g_{n-1})$, so $RB_r(g_1)=RB_{r'}(g_1g_2)$ for a positive integer $r'$ and $R_{r'}(g_1g_2)=R_{r}(g_1) g_2$. Now $RB_{r'}(g_1g_2)$ and $B_t(g_n)$ cancel in $LFP(g_1g_2, g_3,\dotsc, g_n)$. Thus by the induction assumption, we have $RB_{r'}(g_1g_1)B_t(g_n)=1$ and $R_{r'}(g_1g_2)g_3\dotsm g_{n-1}L_t(g_n)=1$, from which desired equalities follow by replacement.
\end{proof}

\begin{lemma}\label{L:AlteredRightSection}
    Let $g,h\in A$ and $1\leq i\leq l(h)$. Then $R_i(gh)=R_i(h)$ if and only if $R_i(h)$ is unaltered in the product $gh$.
\end{lemma}
\begin{proof} The direction of ``if" is clear. Now assume that $R_i(gh)=R_i(h)$. Writing $h=h_1R_i(h)$ as reduced product, $R_i(h)$ is a right factor of $gh=(gh_1)R_i(h)$. Now by Remark \ref{R:unaltered-RightFactor-ReducedProduct}, $P(gh_1,R_i(h))$ is reduced , and thus $R_i(h)$ is unaltered in $P(g,h)$, as desired.
\end{proof}
\begin{corollary}\label{C:UnalteredInLFPToUnalteredInProduct}
    Let $g_1,\dotsc,g_n\in A$, $1\leq i\leq l(g_1)$ and $1\leq j\leq n$.
    Assume that $L_i(g_1)$ is unaltered in $LFP(g_1,\dotsc,g_n)$. Then $L_i(g_1)$ is unaltered in $P(g_1\dotsm g_j, g_{j+1}\dotsm g_n)$. 
\end{corollary}
\begin{proof} As $L_i(g_1)$ is unaltered in $LFP(g_1,\dotsc,g_n)$, we have $L_i(g_1)=L_i(g_1\dotsm g_j)=L_i(g_1\dotsm g_n)$. Then $L_i(g_1)$ is unaltered in the product $P(g_1\dotsm g_j, g_{j+1}\dotsm g_n)$ by (the left version of) Lemma \ref{L:AlteredRightSection}.
\end{proof}

The next result is about the multiplicativity of the unalteredness of left-first product. 
\begin{lemma}\label{L:LeftfirstAlteration-ChangeInitial}
     Let $g_1,\dotsc,g_n\in A$, $1< k<n$ and $B_t$ be a component of of $g_k$. Then $B_t$ is unaltered in $LFP(g_1, \dotsc, g_n)$ if and only if $B_t$ is unaltered in both $LFP(g_1, \dotsc, g_k)$ and $LFP(g_k, \dotsc, g_n)$. 
\end{lemma}

\begin{proof}
     Note that $B_t$ is unaltered in $LFP(g_1,\dotsc,g_k)$ if $B_t$ is unaltered in $LFP(g_1,\dotsc,g_n)$. Now we assume that $B_t$ is unaltered in $LFP(g_1,\dotsc,g_k)$ and we will prove that $B_t$ is unaltered in $LFP(g_1,\dotsc,g_n)$ if and only if $B_t$ is unaltered in $LFP(g_k,\dotsc,g_n)$. This finishes the proof.
     
     First, we write $g_k$ as a reduced product $g_k=g_k'B_tg_k''$. Since $B_t$ is unaltered in the product $P(g_1\dotsm g_{k-1},g_k)$ the last component of $g_1\dotsm g_{k-1}g_k'B_t$ is $B_t$. Note that the last component of $g_k'B_t$ is $B_t$ as well. Hence for every $h\in A$, $P(g_1\dotsm g_{k-1}g_k'B_t, h)$ is reduced if and only if $P(g_k'B_t,h)$ is reduced. 

     Now, $B_t$ is unaltered in $P(g_1,\dotsc,g_n)$ if and only if for each $j$ with $k+1\leq j\leq n$, the product $P(g_1\dotsm g_{k-1}g_k'B_t,g_{k}''g_{k+1}\dotsm g_j)$ is reduced, which is equivalent to $P(g_k'B_t,g_{k}''g_{k+1}\dotsm g_j)$ being reduced. Hence $B_t$ is unaltered in  $P(g_1,\dotsc,g_n)$ if and only if it is unaltered in $P(g_k,\dotsc, g_n)$, as desisred.
\end{proof}

\begin{corollary}
    \label{C:CrossOverCancellationInLeftFirstProduct}
    Let $g_1,\dotsc,g_n\in A$ and $1< k<n$. Assume that components $B_1$ of $g_1$ and $B_k$ of $g_k$ are unaltered in $LFP(g_1,\dotsc, g_k)$ and $B_1$ is altered in $LFP(g_1,\dotsc, g_n)$. Then the component $B_k$ of $g_k$ is altered in $LFP(g_k,\dotsc, g_n)$. 
\end{corollary}
\begin{proof}
    First $B_k$ is to the RHS of $B_1$ in $g_1\dotsm g_k$ by Lemma \ref{L:RelativePositionOfUnalteredSections}. Since $B_1$ is altered in $LFP(g_1\dotsm g_k,g_{k+1},\dotsc,g_n)$, 
so is $B_k$. Thus $B_k$ is altered in $LFP(g_1,\dotsc,g_n)$ as well. But $B_k$ is unaltered in $LFP(g_1,\dotsc,g_k)$, so $B_k$ is altered in $LFP(g_k,\dotsc,g_n)$ by Lemma \ref{L:LeftfirstAlteration-ChangeInitial}.
\end{proof}

\begin{lemma}\label{L:CancellationInLFP-RestrictingLFP}
    Let $g_1,\dotsc,g_n\in A$, $S$ be a factor of $g_i$, $T$ a factor of $g_j$ with $1\leq i<j\leq n$. Assume that $S$ and $T$ cancel in $LFP(g_1,\dotsc,g_n)$. Then $S$ and $T$ cancel in $LFP(g_i,\dotsc,g_j)$. 
\end{lemma}

\begin{proof}
    Since $S$ is unaltered in $LFP(g_1,\dotsc,g_{j-1})$, it is unaltered in $LFP(g_i,\dotsc,g_{j-1})$ by Lemma \ref{L:LeftfirstAlteration-ChangeInitial}. We now only need to prove that $S$ and $T$ cancel in $P(g_i\dotsm g_{j-1}, g_j)$. 

We write  $g_i=g_i'Sg_{i}''$ as a reduced product. Since $S$ is unaltered in $LFP(g_1,\dotsc, g_{j-1})$ (resp. $LFP(g_i,\dotsc,g_{j-1})$), the product $P(g_1\dotsm g_{i-1}g_i'S, g_{i}''g_{i+1}\dotsm g_{j-1})$ (resp.  $P(g_i'S, g_{i}''g_{i+1}\dotsm g_{j-1})$) is reduced.

Since $S$ and $T$ cancel in $P(g_1\dotsm g_{j-1},g_j)$, writing $g_j=g_j'Tg_j''$ as a reduced product, we have that $(g_{i}''g_{i+1}\dotsm g_{j-1})g_j'=1$ and $ST=1$. Now we see that $S$ and $T$ cancel in $P(g_i\dotsm g_{j-1},g_j)$. 
\end{proof}

The next result is about the associativity of cancellation. It will only be used at the very end of this section. 
\begin{lemma}\label{L:AssociationOfCancellationOfAPair}
    Let $\alpha,\beta,\gamma$ be in $A$. Assume that a factor $S$ of $\alpha$ is unaltered in $P(\alpha,\beta)$ and a factor $T$ of $\gamma$ is unaltered in $P(\beta,\gamma)$. Then $S$ and $T$ cancel in $P(\alpha, \beta\gamma)$ if and only if they cancel in $P(\alpha\beta, \gamma)$. 
\end{lemma}
\begin{proof}
    Let $\alpha=\alpha'S\alpha''$ and $\gamma=\gamma'T\gamma$ be reduced products.
    Then the products $(\alpha'S)(\alpha''\beta)$ and $(\beta\gamma')(T\gamma'')$ are reduced. We now have that $S$ and $T$ cancel in $P(\alpha, \beta\gamma)$ (resp. $P(\alpha\beta, \gamma)$) if and only if $\alpha''(\beta\gamma')=1$ (resp. $(\alpha''\beta)\gamma'=1$) and $ST=1$, which implies the desired equivalence.
\end{proof}

\subsection{Multi-malnormality of \texorpdfstring{$C$}{Lg} in \texorpdfstring{$A$}{Lg}}\label{Subsection:LFP-multi-malnormalityOfC}
 Let $C'$ be a normal subsemigroup of $C$ with $1\not\in C'$.
We are interested in products of conjugates of the form $T:=C_1C_2\dotsm C_n$, where $C_i=c_i^{g_i}$ with $c_i\in C'$ and $g_i\in A\backslash C$. In this subsection, we first examine a single conjugate, then study the product of conjugates, and finally prove that $C$ is multi-malnormal in $A$ (see Theorem \ref{T:multi-malnormalityOfC}). 

\subsubsection{Almost reduced form and standard form of conjugates}
Recall the subgroup $C$ of $A$ given in Definition \ref{D:DefineSubgroupC}.

\begin{definition}\label{D:AlmostReducedForm&StandardForm}
Consider $c^g$ with $c\in C\backslash \{1\}$ and $g\in A\backslash C$.
We define the \textbf{almost reduced form} $\chi^{\gamma}$ of $c^g$ by 
$$\chi=c^{L_k(g)}, \quad \gamma=(L_k(g))^{-1}g,$$ 
where $k:=\max\{K(g^{-1},c), K(c,g)\}$.
The \textbf{standard form} $\lambda\mu\rho$ of $c^g$ is defined by 
$$\lambda=\gamma^{-1}\xi_1,\quad  \mu=\xi_1^{-1}\chi\xi_2^{-1}, \quad \rho=\xi_2\gamma,$$
where $\xi_1$ (resp. $\xi_2$) is 1 if the product $\gamma^{-1}\chi$ (resp. $\chi\gamma$) is reduced and it is $L_1(\chi)$ (resp. $R_1(\chi)$) otherwise. \looseness=-1
\end{definition}

We remark that we will only need to consider the case that $g$ is left $C$-simplified, in which case $K(\gamma^{-1},\chi)$ and $K(\chi, \gamma)$ are both zero, thus the products $\gamma^{-1}\chi$ and $\chi\gamma$ are reduced or have a merge. Thus the definition of the standard form says that if there is a merge in $\gamma^{-1}\chi$ (resp. $\chi\gamma$), then we let the first (resp. last) component of $\chi$ ``merge into" $\gamma^{-1}$ (resp. $\gamma$) to form $\lambda$ (resp. $\rho$).  We next record several essential properties of these definitions. 
\begin{lemma}\label{L:AlmostReducedAndReducedForm}
 Consider $c^g$ with $c\in C\backslash\{1\}$ and $g\in A\backslash C$ left $C$-simplified. Let $\chi^{\gamma}$ and $\lambda\mu\rho$ be the almost reduced form and standard form of $c^g$ respectively. Let $i=K(g^{-1},c)$, $j=K(c,g)$ and $k=\max\{i,j\}$. Then we have the following:
\begin{enumerate}
    \item\label{Item:AlmostReducedAndReducedForm-reduced} The product $\gamma^{-1}\chi\gamma$ is almost reduced; the product $\lambda\mu\rho$ is reduced. 
    \item \label{Item:AlmostReducedAndReducedForm-Similarity}$l(\lambda)=l(\gamma)=l(\rho)$. Moreover, if $l(\gamma)>0$, we have 
    $$L_{l(\gamma)-1}(\lambda)=L_{l(\gamma)-1}(\gamma^{-1}),\quad R_{l(\gamma)-1}(\rho)=R_{l(\gamma)-1}(\gamma).$$
    \item\label{Item:AlmostReducedAndReducedForm-length} $l(\chi)\geq 2s$, $l(\mu)\geq 2s-1$, $i+j\leq l(c)-s$, $i+j\leq l(\mu)-s+1$ and  $l(c^g)\geq 2s-1+2l(\gamma)$.
    \item 
      \label{Item:cgIncompatible}$c^g$ is left $(l(\gamma)+2)$-incompatible and right $(l(\gamma)+2)$-incompatible with $C$. 
\end{enumerate}
\end{lemma}
\begin{proof}
\begin{enumerate}
    \item These are direct consequences of the definitions.
    \item \label{Item:ProvingThreeInequalities-AlmostReduced} By definition, $\lambda=\gamma^{-1}$ or $\lambda = \gamma^{-1}L_1(\chi)$ when the product $\gamma^{-1}\chi$ has a merge. Therefore, we have $l(\lambda)=l(\gamma)$ and that $\gamma^{-1}$ and $\lambda$ differ at most at the last component, implying that $L_{l(\gamma)-1}(\lambda)=L_{l(\gamma)-1}(\gamma^{-1})$ when $l(\gamma)>0$. Similarly, we prove the other equalities. 
First $i,j\leq s$ as $g$ is left $C$-simplified.  Also the last inequality follows from the second because $\lambda\mu\rho$ is reduced and $l(\lambda)=l(\gamma)=l(\rho)$.  Now we prove the first four inequalities considering two cases. (i) $i\neq0$ and $j\neq0$. Then write $c=u_1\dotsm u_p$ as a reduced word using the free generating set $S$, it must satisfy $u_1=u_p^{-1}$ and $p \geq 3$ because of our choice of exponents in Definition \ref{D:DefineSubgroupC}\ref{Item:SVeryDistinct}, therefore $i=j\leq s$ and $l(c)\geq 6s-2$,
    from which the first four inequalities follow. 
    (ii) $i=0$ or $j=0$. Then $\chi$ is the cyclic conjugate of $c$, so we have $l(\chi)=l(c)\geq2s$ and $i+j\leq l(c)-s$. Also, the fourth inequality follows from the second. The second inequality clearly holds if $l(\chi)>2s$.  On the other hand, if $l(\chi)=2s$, then $\gamma^{-1}\chi$ or $\chi \gamma$ is reduced, so $\xi_1$ or $\xi_2$ appearing in Definition \ref{D:AlmostReducedForm&StandardForm} is the identity, hence we still have the second inequality.
 
    \item Let $l'=l(\gamma)+2=n-k+2$, where $n:=l(g)$. We only prove that $c^g$ is left $l'$-incompatible with $C$, as the other statement can be done similarly. We have the following. 
    \begin{enumerate}
        \item\label{Item:Ll'cgAndLl'g-1c} $L_{l'}(c^g)=L_{l'}(g^{-1}c)$.  Write $c=c_1c_2$ as a reduced product, such that $l(c_2)=j+2$, then $l(c_1)=l(c)-j-2\geq i+s-2\geq i+8$, thus $K(g^{-1},c_1)=i$. 
        We now see that $R_1(c_1)$ (resp. $L_1(c_2)$) is unaltered in $g^{-1}c_1$ (resp. $c_2g$), and thus the products $(g^{-1}c_1)(c_2g)$ and $(g^{-1}c_1)c_2$ are reduced.
        Note that we have, recalling $l(c)\geq i+j+s$ and $i\leq k$,
    \begin{align*}
        l(g^{-1}c_1)&\geq l(g^{-1})+l(c_1)-2i-1= n+l(c)-j-2-2i-1\\
        &\geq n-i-3+s\geq n-k-3+s= l'-5+s\geq l'+5.
    \end{align*} Therefore $L_{l'}(c^g)=L_{l'}(g^{-1}c_1c_2g)=L_{l'}(g^{-1}c_1)=L_{l'}(g^{-1}c_1c_2)=L_{l'}(g^{-1}c)$. 
    \item If $k=i$, then $c^g$ is left $l'$-incompatible with $C$. This is because by Corollary \ref{C:Incompatibility&CancellationNumber}, we have that $g^{-1}c$ is $(n-k+1)$-incompatible with $C$, from which the desired result follows using $L_{l'}(c^g)=L_{l'}(g^{-1}c)$. 
    \item  If $k>i$, then $c^g$ is $l'$-incompatible with $C$. To prove this, we first note that $k=j$, then $i=0$ by (\ref{Item:ProvingThreeInequalities-AlmostReduced})(i), and $R_k(c)L_k(g)=1$.
Now consider two subcases.    
\begin{enumerate} \item $k=1$. In this case, we want to prove that $L_{n+1}(g^{-1}c)\not\in L_{n+1}(C)$. But this is true by Corollary \ref{C:Incompatibility&CancellationNumber} (as $i=0$). 
\item $k\geq2$. In this case, we have $L_{n-k+1}(g^{-1})=L_{n-k+1}(g^{-1}c)$ (as $i=0$). 
Thus we only need to show that $L_{n-k+1}(g^{-1})\not\in L_{n-k+1}(C)$. 
Otherwise, $R_{n-k+1}(g)\in R_{n-k+1}(C)$ and then $g':=p(B_k(g))R_{n-k+1}(g)\in C$ by Lemma \ref{L:Prefix}. But $p(B_k(g))=L_{k-1}(g)$ because $L_k(g)=(R_k(c))^{-1}=L_k(c^{-1})=L_k(u)$ for some $u\in S\cup S^{-1}$. Hence $C\ni g'=L_{k-1}(g)R_{n-k+1}(g)=g$, a contradiction. 
\end{enumerate}
    \end{enumerate}
\end{enumerate}
\end{proof}

The next result will assist us in proving Corollary \ref{C:CancellationOfBlocksFromCi}, which together with  Lemma \ref{L:PairCancellationToATrivialProduct} tells us that if a component from $\mu_i$ and a component from $\mu_j$ cancel, then the partial product  $g_iC_{i+1}\dotsm C_{j-1}g_j^{-1}$ is in $C$, and therefore we can reduce the number of conjugates in the product $T=C_1\dotsm C_n$ using Lemma \ref{L:aiCiCj-1aj-1}.

\begin{lemma}\label{L:PropertyOfC-Conjugate-LocalProperty}
Let $c\in C\backslash\{1\}$, $g\in A\backslash C$ be left $C$-simplified. 
Let $c^{g}=D_1\dotsm D_t\dotsm D_n$ be in reduced form and $c^{g}=\lambda\mu\rho$ be in standard form. If $D_t$ is a component of $\mu$, then we have
\begin{align}\label{Con:ConjugateAndp}
gD_1\dotsm D_{t-1}(p(D_t))^{-1}\in C\quad \text{and}\quad p(D_t) D_t\dotsm D_ng^{-1}\in C.
\end{align}
\end{lemma}
\begin{proof}
   Note that each containment in \eqref{Con:ConjugateAndp} implies the other, as the product of the terms considered is equal to $c \in C$. Thus, we only need to show that for each component $D_t$ of $\mu$, one of the containments holds. Let $i=K(g^{-1},c), j=K(c,g)$, $k=\max\{i,j\}$ and $h_1=L_k(g)$, then $g=h_1\gamma$ is a reduced product,  and $i+j\leq l(c)-s$ by Lemma \ref{L:AlmostReducedAndReducedForm}(\ref{Item:AlmostReducedAndReducedForm-length}), which implies $R_1(h_1^{-1}ch_1)=R_1(ch_1)$.

   (1) We first show that $g\lambda\mu$ (resp. $\lambda\mu g^{-1})$ is left (resp. right) compatible with $C$. We only prove the statement about $g\lambda\mu$, as the other can be done similarly. Note that 
   $$g\lambda\mu=h_1\gamma \cdot  \gamma^{-1}\xi_1 \cdot \xi_1^{-1}h_1^{-1}ch_1\xi_2^{-1}=ch_1\xi_2^{-1}$$
   where $\xi_2$ is $1$ of $R_1(\chi)=R_1(h_1^{-1}ch_1)=R_1(ch_1)$. Hence $g\lambda\mu$ is a left factor of $ch_1$. Thus, we only need to prove that $ch_1$ is left compatible with $C$. If $k=j$, then $R_k(c)h_1=1$ and then $ch_1$ is a left factor of $c$,  so $ch_1$ is left compatible with $C$. If $k>j$, then $k=i$ and then $h_1^{-1}L_k(c)=1$, thus $ch_1=cL_k(c)$ is a left factor of $c^2$, hence $ch_1$ is left compatible with $C$. 

   (2) We now prove that $R_{l(\mu)-j-1}(\mu)$ (resp. $L_{l(\mu)-i-1}(\mu)$) is unaltered in $P(g,\lambda\mu)$ (resp. $P(\lambda\mu,g^{-1}))$. Again, we only prove the statement about $P(g,\lambda\mu)$. This is clearly true if  $L_1(\mu)$ is not canceled by $g$ in $P(g,\lambda\mu)$. Now we assume that $L_1(\mu)$ is canceled by $g$. Recalling that $g=h_1\gamma$ and $\lambda\mu$ are reduced products, and $l(\gamma)=l(\lambda)$, we know then $\gamma\lambda=1$, thus $\xi_1=1$ as $\lambda=\gamma^{-1}\xi_1$. Hence $P(g,\lambda\mu)=P(h_1\gamma, \gamma^{-1}\mu)$ with $\mu=h_1^{-1}ch_1\xi_2^{-1}$ and we only need to prove that $K(h_1,\mu)\leq j$. This is true if $k=j$, as $l(h_1)=k$. If $k>j$, then $k=i$ and $c$ can be written as a reduced product $h_1c'$, now $\mu=c'h_1\xi_2^{-1}$, and $L_1(\mu)=L_1(c')$ as $l(c)\geq i+j+s$ by Lemma \ref{L:AlmostReducedAndReducedForm}(\ref{Item:AlmostReducedAndReducedForm-length}),  thus $h_1\mu$ is reduced, implying $K(h_1,\mu)=0\leq j$, as desired.

   By Lemma \ref{L:Prefix}, (1) and (2) together imply that if $D_t$ is a component of $R_{l(\mu)-j-1}(\mu)$ (resp. $L_{l(\mu)-i-1}(\mu)$), we have the first (resp. second) containment of \eqref{Con:ConjugateAndp}. But $i+j\leq l(\mu)-s-1$ by Lemma \ref{L:AlmostReducedAndReducedForm}(\ref{Item:AlmostReducedAndReducedForm-length}), so all components of $\mu$ are components of $R_{l(\mu)-j-1}(\mu)$ or components of $L_{l(\mu)-i-1}(\mu)$. So  for each component $D_t$ of $\mu$, one containment of \eqref{Con:ConjugateAndp} holds, as desired.
\end{proof}

\subsubsection{Products of conjugates}
We now turn to the study of a product of conjugates. The next result follows from Lemma \ref{L:PropertyOfC-Conjugate-LocalProperty}.
\begin{corollary}\label{C:CancellationOfBlocksFromCi}
Consider a product $T=C_1\dotsm C_n$, where for each $i$ with $1 \leq i \leq n$, $C_i=c_i^{g_i}$, $c_i\in C\backslash\{1\}$, $g_i\in A\backslash C$ is left $C$-simplified and $C_i=\lambda_i\mu_i\rho_i$ is the standard form of $C_i$. Consider the multiplication $C_rC_{r+1}\dotsm C_t$ with $1\leq r<t\leq n$. Let $C_r=B_1\dotsm B_{p}\dotsm B_k$ and $C_t=D_1\dotsm D_{q}\dotsm D_l$ be in reduced form, where $B_{p}$ is a component of $\mu_r$ and $D_{q}$ one of $\mu_t$. Assume that $B_{p}D_{q}=1$ and $B_{p}\dotsm B_kC_{r+1}\dotsm C_{t-1}D_1\dotsm D_{q}=1$. Then $c_{rt}:=g_rC_{r+1}\dotsm C_{t-1}g_t^{-1}\in C$.
\end{corollary}
\begin{proof}
By Lemma \ref{L:PropertyOfC-Conjugate-LocalProperty}, we have 
\begin{align*}
\eta:=p(B_{p})B_{p}\dotsm B_k g_r^{-1}\in C, \quad \eta':=g_tD_1\dotsm D_{q-1}(p(D_{q}))^{-1}\in C.
\end{align*} Since $\delta:=p(D_{q})D_{q}(p(D_{q}^{-1}))^{-1}\in C$ by \eqref{P(E)&P(E-1)} and $B_{p}D_{q}=1$, we find 
\begin{align*}
    C\ni \theta:=\eta'\delta=g_tD_1\dotsm D_{q}(p(D_{q}^{-1}))^{-1}=g_tD_1\dotsm D_{q}(p(B_{p}))^{-1}.
\end{align*}
Now we obtain $c_{rt}=g_rC_{r+1}\dotsm C_{t-1}g_t^{-1}=\eta^{-1}\theta^{-1}\in C$ by the computation 
\begin{align*}
1=B_{p}\dotsm B_kC_{r+1}\dotsm C_{t-1}D_1\dotsm D_{q}&=(p(B_{p}))^{-1}\eta g_rC_{r+1}\dotsm C_{t-1}g_t^{-1}\theta p(B_{p}). 
\end{align*} 
\end{proof}

So, this result tells us that if a certain type of cancellation happens in the product $T$, then the partial product $c_{rt}$ is in $C$. The next result says that if $c_{rt}\in C$, then we can rewrite $T$ as a product of fewer conjugates than the original expression for $T$. 

\begin{lemma}\label{L:aiCiCj-1aj-1} Let $C'$ be a normal subsemigroup of $C$, $T=C_1\dotsm C_n$ with $C_i=c_i^{g_i}$, $c_i\in C'$ and $g_i\in A$ for all $i$.
Assume that $c_{rt}:=g_{r}C_{r+1}\dotsm C_{t-1} g_{t}^{-1}\in C$ with $1\leq r<t\leq n$. Then
\[C_rC_{r+1}\dotsm C_{t-1}C_t=C_r'C_{r+1}\dotsm C_{t-1}, \text{ where }C_{r}'=(c_{r}')^{g_{r}}~ \text{ for some } c_{r}'\in C'.\]
\end{lemma}
\begin{proof}We have
\begin{align*}
     C_rC_{r+1}\dotsm C_{t-1}C_t
    =&g_r^{-1}c_rg_rC_{r+1}\dotsm C_{t-1}g_t^{-1}c_tg_t=g_{r}^{-1}c_{r}c_{rt}c_tg_{t}\\
    =&g_{r}^{-1}c_{r}c_{rt}c_tc_{rt}^{-1}g_{r}C_{r+1}\dotsm C_{t-1}=C_r'C_{r+1}\dotsm C_{t-1},
\end{align*}
where $C_r'={c_r'}^{g_{r}}$ with $c_r'=c_{r}c_{t}^{c_{rt}^{-1}}\in C'$, as $C'$ is a normal subsemigroup of $C$ . 
\end{proof}
\subsubsection{Proof of multi-malnormality of \texorpdfstring{$C$}{Lg}}\label{Subsection:multi-malnormality}
We now prove the main result of this subsection. 

\begin{theorem}\label{T:multi-malnormalityOfC}
    The subgroup $C$ given by Definition \ref{D:DefineSubgroupC} is multi-malnormal in $A$.
\end{theorem}

 \begin{proof}
   Let $C'$ be a normal subsemigroup of $C$ with $1\not\in C'$. Let $\mathcal S$ be the set of tuples $((t_1, a_1),\dotsc, (t_k, a_k))$ such that $t_i\in C'$, $a_i\in A\backslash C$, $k\geq1$ and $t_1^{a_1}\dotsm t_k^{a_k}\in C$.
 We want to prove that $\mathcal S$ is empty. By assuming that $\mathcal S$ is nonempty, we will find a contradiction. 

For a tuple $V:=((t_1, a_1),\dotsc, (t_k, a_k))\in \mathcal S$, we associate two numbers with $V$: the number of conjugates $NC(V)=k$
and the sum  of conjugate lengths $SLC(V)=\sum_{i=1}^k l(t_i^{a_i})$. These define a map $\phi: \mathcal{S} \rightarrow \mathbb{N}^2$, where $\phi(V) = (NC(V), SLC(V)).$   Lexicographically order $\mathbb N^2$ according to the rule $(m_1,n_1)<(m_2,n_2)$ if $m_1<m_2$,  or if $m_1=m_2$ and $n_1<n_2$. 
  Choose $V_0 \in \mathcal{S}$ satisfying $\phi(V_0)\leq \phi(V)$ for all $V \in \mathcal{S}$.

Suppose $V_0=((c_1,g_1),\dotsc, (c_n,g_n))$, then $n\geq1$, $c_i\in C'$, $g_i\in A\backslash C$ and $\Pi:=c_1^{g_1}\dotsm c_n^{g_n}\in C.$ Moreover, by using Lemma \ref{L:C-simplified}(1) we may assume that these $g_i$ are left $C$-simplified.  
For each $t \in \{1, \ldots, n\}$, let $C_t=c_t^{g_t}=\chi_t^{\gamma_t}=\lambda_t\mu_t\rho_t,$ where the last two expressions are the almost reduced form and the standard form of $C_t$. To simplify notation, for $1\leq r\leq t\leq n$ we let $\Pi_{r,t}=C_rC_{r+1}\dotsm C_t$ (and for convenience, define $\Pi_{r,t}=1$ when $r>t$).  We also write $L_{i,j}$ for the left-first product $LFP(C_i,\dotsc, C_j)$, $R_{i,j}$ for the right-first product $RFP(C_i,\dotsc, C_j)$. Let $x$ be a fixed integer satisfying $3 \leq x \leq s-6$.  We now proceed by two steps to reach a contradiction.

Our first step is to establish three preparatory claims using minimality of $V_0$.  
\begin{enumerate}[wide, labelwidth=!, labelindent=0pt, label=\textbf{Claim \Alph*}]
\item\label{Claim:CancellationOfGammaiAndPiCan'tBeTwomuch}: Let $1\leq r<t\leq n$, then we have $$l(\gamma_r\Pi_{r+1,t})\geq l(\gamma_r)-1,\qquad l(\Pi_{r,t-1}\gamma_t^{-1})\geq l(\gamma_t^{-1})-1.$$
    We only prove the first inequality as the second one can be done similarly. For contradiction, assume $l(\gamma_r\Pi_{r+1,t})<l(\gamma_r)-1$. 
Letting $C_r'=C_r^{\Pi_{r+1,t}}$, we have
\begin{align*}
    \Pi_{r,t}=C_r\Pi_{r+1,t}=\Pi_{r+1,t}\Pi_{r+1,t}^{-1}C_r\Pi_{r+1,t}=\Pi_{r+1,t}C_r',
\end{align*} 
and hence we find
\begin{align*}
    \Pi&=\Pi_{1,n}=\Pi_{1,r-1}\Pi_{r,t}\Pi_{t+1,n}=\Pi_{1,r-1}\Pi_{r+1,t}C_r'\Pi_{t+1,n}\\
    &=\Pi_{1,r-1}\Pi_{r+1,t}C_r'\Pi_{t+1,n}C_r'^{-1} C_r'\\
    &=\Pi_{1,r-1}\Pi_{r+1,t}\Pi_{t+1,n}^{C_r'^{-1}} C_r'.
\end{align*} 
If $C_r'\in C$, then $\Pi\in C$, $\Pi':=\Pi_{1,r-1}\Pi_{r+1,t}\Pi_{t+1,n}^{C_r'^{-1}}$ is in $C$ as $\Pi\in C$, and $\Pi'$ is a product of $n-1$ conjugates of the form $c^g$ with $c\in C'$ and $g\in A\backslash C$, contradicting the minimality of $V_0$. Therefore, we have $C_r'\in A\backslash C$. 
Since $C_r'=C_r^{\Pi_{r+1,t}}=c_r^{g_r\Pi_{r+1,t}}=c_r^{g_r'}$, where $g_r':=g_r\Pi_{r+1,t}$, we have $g_r'\in A\backslash C$ as $c_r\in C$. Also, $C_r'=(\chi_r^{\gamma_r})^{\Pi_{r+1,t}}=\chi_r^{\gamma_r\Pi_{r+1,t}}$, and thus
\begin{align*}
l(C_r')\leq l(\chi_r)+2l(\gamma_r\Pi_{r+1,t})<l(\chi_r)+2(l(\gamma_r)-1)=l(\chi_r)+2l(\gamma_r)-2\leq l(C_r)
\end{align*} where the last inequality is because the product $C_r=\gamma_r^{-1}\chi_r \gamma_r$ is almost reduced. Considering the product $\Pi=\Pi_{1,r-1}\Pi_{r+1,t}C_r'\Pi_{t+1,n}$, we obtain a tuple $V_1\in\mathcal S$, such that $NC(V_1)=NC(V_0)$ but $SLC(V_1)<SLC(V_0)$, contradicting the minimality of $V_0$. 
\qed

\item \label{Item:ClaimE}: Let $1\leq i<j\leq n$ and assume that 
$l(\Pi_{i,j-1})\geq l(C_i)-6$. 
Then $l(\Pi_{i,j-1})\geq l(\gamma_i)+x+1$ and $B_{l(\gamma_i)+x+1}(\Pi_{i,j-1})$ is not canceled by the factor $\lambda_j$ of $C_j$ in $P(\Pi_{i,j-1}, C_j)$.

To prove the inequality, we only need $l(C_i)-6\geq l(\gamma_i)+x+1$, thus by the last equality of Lemma \ref{L:AlmostReducedAndReducedForm}(\ref{Item:AlmostReducedAndReducedForm-length}), we only need $2s-1-6\geq x+1$, which follows from $3\leq x\leq s-6$.

    Now we prove the second statement by contradiction, assuming it is canceled by $\lambda_j$.   
     Then the number of components of $\Pi_{i,j-1}$ canceled by $\lambda_j$ in $P(\Pi_{i,j-1},\lambda_j)$ is at least $\kappa:=l(\Pi_{i,j-1})-l(\gamma_i)-x$, so $K(\Pi_{i,j-1},\lambda_j)\geq \kappa$. Since $\gamma_j^{-1}$ and $\lambda_j$ are of the same length, and at most differ in their right-most components (by Lemma \ref{L:AlmostReducedAndReducedForm})\eqref{Item:AlmostReducedAndReducedForm-Similarity}, we thus have 
    $K(\Pi_{i,j-1},\gamma_j^{-1})\geq \kappa-1$.
         It now follows that 
          \begin{align*}
              l(\Pi_{i,j-1}\gamma_j^{-1})&\leq l(\Pi_{i,j-1})+l(\gamma_j^{-1})-2K(\Pi_{i,j-1},\gamma_j^{-1})\\&\leq l(\Pi_{i,j-1})+l(\gamma_j^{-1})-2(\kappa-1)\\
              &=l(\gamma_j^{-1})-l(\Pi_{i,j-1})+2l(\gamma_i)+2x+2\\&\leq l(\gamma_j^{-1})-l(C_i)+6+2l(\gamma_i)+2x+2\\
              &\leq l(\gamma_j^{-1})-2s+2x+9\leq l(\gamma_j^{-1})-3,
          \end{align*} where the penultimate inequality is obtained using the last inequality of Lemma \ref{L:AlmostReducedAndReducedForm}\eqref{Item:AlmostReducedAndReducedForm-length} and the last inequality is because $x\leq s-6$. This contradicts  \ref{Claim:CancellationOfGammaiAndPiCan'tBeTwomuch}.  \qed   
\item \label{Claim:NonCancellationAmongComponentsFromC}: Let $1\leq i\leq r<t\leq j\leq n$. Let $B_p$ and $D_q$ be components of $\mu_r$ and $\mu_t$, respectively. Then $B_p$ and $D_q$ do not cancel in $L_{i,j}$. (Recall that we write $L_{i,j}$ for $LFP(C_i,\dotsc,C_j)$.)

     We prove by contradiction, supposing that the components $B_p$ and $D_q$ cancel in $L_{i,j}$. By Lemma \ref{L:CancellationInLFP-RestrictingLFP}, $B_p$ and $D_q$ cancel in $L_{r,t}$ too. Writing $C_r=B_1\dotsm B_p\dotsm B_k$ and $C_t=D_1\dotsm D_q\dotsm D_l$ as reduced forms, then $B_pD_q=1$ and $B_p\dotsm B_kC_{r+1}\dotsm C_{t-1}D_1\dotsm D_q=1$
     by Lemma \ref{L:PairCancellationToATrivialProduct}. Then by Corollary \ref{C:CancellationOfBlocksFromCi}, $g_rC_{r+1}\dotsm C_{t-1}g_t^{-1}\in C$. Applying Lemma \ref{L:aiCiCj-1aj-1}, we have 
   $\Pi_{r,t}=C_r'C_{r+1}\dotsm C_{t-1}$, where $C_r'={c_{r}'}^{g_r}$ with $c_r'\in C'$. Now $\Pi=\Pi_{1,r-1}C_r'C_{r+1}\dotsm C_{t-1}\Pi_{t+1,n}\in C$ is a product of $n-1$ terms of the form $c^g$ with $c\in C'$ and $g\in A\backslash C$. This contradicts the minimality of $V_0$. \qed
\end{enumerate}

Our second step is to prove the following claim to finish the proof of the theorem. To ease notation, we write $\mathcal L_t$ in place of $L_{l(\gamma_t)+x}(C_t)$ and  $\mathcal R_t$ in place of $R_{l(\gamma_t)+x}(C_t)$.
 \begin{claim}
\label{main_claim} 
   If $i, j$ are integers satisfying $1 \leq i  \leq j \leq n$, 
   then $\mathcal L_i$ is unaltered in $L_{i,j}$ and $\mathcal R_j$ is unaltered in $R_{i,j}$.
\end{claim}
Let us see how the claim contradicts the nonemptiness of $\mathcal S$. Assuming Claim \ref{main_claim}, we know $\mathcal L_1$ is unaltered in $L_{1,n}$ and so $L_{l(\gamma_1)+x}(C_1) = L_{l(\gamma_1)+x}(\Pi)$.  But $C_1$ is $(l(\gamma_1)+2)$-incompatible with $C$ by Lemma \ref{L:AlmostReducedAndReducedForm}(\ref{Item:cgIncompatible}), so that $\Pi$ is also $(l(\gamma_1)+2)$-incompatible with $C$, in particular, $\Pi \notin C$. 
Recalling that the assumption that $\mathcal S$ is nonemtpy leads to the existence of $V_0$ and that $\Pi\in C$, this gives a desired contradiction. 

 The remainder of the paper is therefore devoted to proving Claim \ref{main_claim}. 
We induct on the number $j-i+1$ of conjugates in the product $\Pi_{i,j}$. Note that if $j-i+1=1$, then the conclusion of Claim \ref{main_claim} is trivially true. So, let $K\geq 2$ and assume that the conclusion of Claim \ref{main_claim} holds for all $i,j$ with $j-i+1<K$. Assuming that $j-i+1=K$, we want to prove that the conclusion of Claim \ref{main_claim} is true in this case. We proceed with two substeps to achieve this goal. Our first substep is to prove the following \ref{Item:ClaimRBlgammiUnaltered} and \ref {Item:ClaimD} based on the induction assumption.

\begin{enumerate}[wide, labelwidth=!, labelindent=0pt, label=\textbf{Claim \Alph*}]
\setcounter{enumi}{3}
\item \label{Item:ClaimRBlgammiUnaltered}: Let $i'$ be an integer satisfying $i\leq i'<j$. Assume that a component $B_{i'}$ of $C_{i'}$ is unaltered in $L_{i',j-1}$ but altered in $L_{i',j}$. Then the factor $RB_{l(\gamma_{i'})+2}(C_{i'})$ of $C_{i'}$ is unaltered in $L_{i',j-1}$.

We prove this with two sub-substeps. Note that the claim is trivially true if $i'=j-1$, and so we focus on the case that $i'\leq j-2$. 
\begin{enumerate}
    \item \label{Item:ppiik-1ck}
    Let $k$ be a fixed integer satisfying $i'+1\leq k\leq j-1$. 
   We prove that  at least one component of $C_k$ is altered in $P(\Pi_{i',k-1},C_k)$, and if $t$ is maximal such that $B_t(C_k)$ is altered, then $l(\gamma_k)+x\leq t\leq l(C_k)-(l(\gamma_k)+x)$. 

     To prove this, we rule out two other possible cases for the product $P(\Pi_{i',k-1}, C_k)$. 
     \begin{enumerate}
         \item No component of $C_k$ is altered or the rightmost component of $C_k$ altered is $B_t(C_k)$ with $t<l(\gamma_k)+x$.  In this case, $B_{l(\gamma_k)+x}(C_k)$ is unaltered in $L_{i',k}$, and by assumption, the component $B_{i'}$ of $C_{i'}$ is unaltered in $L_{i', j-1}$ (thus also unaltered in $L_{i',k}$) but altered in $L_{i',j}$. Therefore by Corollary \ref{C:CrossOverCancellationInLeftFirstProduct}, $B_{l(\gamma_k)+x}(C_k)$ is altered in $L_{k, j}$.  But this means $\mathcal L_k$ is altered in $L_{k,j}$, contradicting the induction assumption as $j-k+1<K$.  
         \item The rightmost component of $C_k$ altered is $B_t(C_k)$ with $t > l(C_k)-(l(\gamma_k)+x)$. Then $B_t(C_k)$ is a component of $\mathcal R_k$, thus $\mathcal R_k$ is altered in $P(\Pi_{i',k-1}, C_k)$. But $\mathcal R_k$ is unaltered in $R_{i',k}$ by the induction assumption (as $k-i'+1<K$), thus $\mathcal R_k$ is unaltered in  $P(\Pi_{i',k-1}, C_k)$ by (the right version of) Corollary \ref{C:UnalteredInLFPToUnalteredInProduct}, a contradiction.
     \end{enumerate}
    \item  Second, we prove the following Subclaim. Let's first see how this Subclaim finishes the proof of \ref{Item:ClaimRBlgammiUnaltered}. 
By this Subclaim, for each $k$ with $i'+1\leq k\leq j-1$, $\eta_{i'}:=\lambda_{i'}\mu_{i'}(R_1(\mu_{i'}))^{-1}$ is a left factor of $\Pi_{i',k}$, and thus $\eta_{i'}$ is unaltered in $P(\Pi_{i',k-1},C_k)$ by (the left version of) Lemma \ref{L:AlteredRightSection}. Therefore $\eta_{i'}$ is unaltered in $L_{i',j-1}$ and thus so is its right-most component $RB_{l(\gamma_{i'})+2}(C_{i'})$, as desired by \ref{Item:ClaimRBlgammiUnaltered}. 

     \begin{quote}    
    \textbf{Subclaim}:
    For every $k$ with $i'\leq k\leq j-1$, $\Pi_{i',k}$ can be written as an almost reduced product
    \begin{align}\label{E:KeyExpression-Fact1-ProofOfmulti-malnormality}
\Pi_{i',k}=\lambda_{i'}\zeta_{j_0}\zeta_{j_1}\dotsm \zeta_{j_{r_k}} h_k,
    \end{align}
    where $\zeta_{j_0}=\mu_{i'}$, $r_k\geq0$, $h_k\in A$, $i'=j_0<j_1<\dotsm<j_{r_k}\leq k$, the products $\lambda_{i'}\zeta_{j_0}$ and  $\zeta_{j_{r_k}}h_k$ are reduced, 
    each $\zeta_{t}$ is a right factor of $\mu_{t}$ with $l(\zeta_t)\geq x+1$, each $\zeta_t':=(L_1(\zeta_t))^{-1}\zeta_t(R_1(\zeta_t))^{-1}$ is unaltered in $L_{i',k}$, and $\zeta_{j_0}'':=\zeta_{j_0}(R_1(\zeta_{j_0}))^{-1}$ and $\zeta_{j_{r_k}}'':=(L_1(\zeta_{j_{r_k}}))^{-1}\zeta_{j_{r_k}}$ are unaltered in $L_{i',k}$. 
     \end{quote}
    To prove this subclaim, we induct on $k$. It is clearly true when $k=i'$, as we can take $r_k=0$ and $h_{i'}=\rho_{i'}$. Now assume that $\Pi_{i',k}$ with $i'\leq k\leq j-2$ can be written as the desired almost reduced product 
    \eqref{E:KeyExpression-Fact1-ProofOfmulti-malnormality}; we want to prove that $\Pi_{i',k+1}$  can too. We have
    $$
\Pi_{i',k+1}=\Pi_{i',k}C_{k+1}=\lambda_{i'}\zeta_{j_0}\zeta_{j_1}\dotsm \zeta_{j_{r_k}} h_k\lambda_{k+1}A_1\dotsm A_{m_{k+1}}\rho_{k+1},
   $$
     where we write $\mu_{k+1}$ in reduced form $A_1\dotsm A_{m_{k+1}}$. By \ref{Item:ClaimRBlgammiUnaltered}(a), the rightmost component of $C_{k+1}$ being altered in  $P(\Pi_{i',k}, C_{k+1})$ is $A_r$ with $x\leq r\leq m_{k+1}-x$. Thus $A_{r-1}$ is canceled in $P(\Pi_{i',k}, C_{k+1})$. Let $\delta_k=\zeta_{j_0}\zeta_{j_1}\dotsm \zeta_{j_{r_k}}$, then $\Pi_{i',k}=\lambda_{i'}\delta_kh_k$ is a reduced product. We want to show that $A_{r-1}$ is canceled by $h_k$ in $P(\Pi_{i',k},C_{k+1})$. First, it cannot be canceled by $\lambda_{i'}$, as that will lead to $\mathcal L_{i'}$ being altered in $L_{i',k+1}$, contradicting the induction assumption (as $k+1-i'+1<K$). Second, note that each $\zeta_t'$ for $t=j_0,j_1,\dotsc,j_{r_k}$ is a factor of $\mu_t$ unaltered in $L_{i',k}$, so $\zeta_t'$ cannot cancel $A_{r-1}$ by \ref{Claim:NonCancellationAmongComponentsFromC} as $A_{r-1}$ is a component $\mu_{k+1}$. Similarly, $\zeta_{j_0}''$ or $\zeta_{j_{r_k}}''$ cannot cancel $A_{r-1}$ either. 
     Therefore, if $A_{r-1}$ is canceled by $\delta_k$, then it can only be canceled by $R_1(\zeta_{j_{t-1}})L_1(\zeta_{j_t})$ for some $t\in\{1,2,\dotsc,r_k\}$, but then $A_{r-2}$ (note $r \geq x \geq 3$) is canceled by a component of $\zeta_{j_t}'$, contradicting \ref{Claim:NonCancellationAmongComponentsFromC}. 
     Therefore $A_{r-1}$ can only be canceled by a component of $h_k$, call this component $D$. 
    \begin{enumerate}
        \item If $D$ is not the first component of $h_k$, then letting $h_{k+1}=h_kC_{k+1}$, we know that $$\Pi_{i',k+1}=\lambda_{i'}\zeta_{j_0}\zeta_{j_1}\dotsm \zeta_{j_{r_k}} h_{k+1}$$ is the desired almost reduced product. 
        \item If $D$ is the first component of $h_k$, then  $h_k\lambda_{k+1}A_1\dotsm A_{r-1}=1$ and thus
    \begin{align*}
    \Pi_{i',k+1}
    &=\lambda_{i'}\zeta_{j_0}\zeta_{j_1}\dotsm \zeta_{j_{r_k}} A_r\dotsm A_{m_{k+1}}\rho_{k+1}=\lambda_{i'}\zeta_{j_0}\zeta_{j_1}\dotsm \zeta_{j_{r_k}} \zeta_{j_{r_{k+1}}}\rho_{k+1},
    \end{align*}
    where $\zeta_{j_{r_{k+1}}}:=A_r\dotsm A_{m_{k+1}}$ is a right factor of $\mu_{k+1}$ with $l(\zeta_{j_{r_{k+1}}})=m_{k+1}-r+1\geq x+1$, $K(\zeta_{j_{r_k}}, \zeta_{j_{r_{k+1}}})=0$ and the product $\zeta_{j_{r_{k+1}}}\rho_{k+1}$ is reduced. This gives the desired almost reduced product for $\Pi_{i',k+1}$, proving the Subclaim and thus \ref{Item:ClaimRBlgammiUnaltered}. \qed
    \end{enumerate} 
\end{enumerate}
 
\item \label{Item:ClaimD}: Assume that $\mathcal L_i$ is altered in $L_{i,j}$. Then $l(\Pi_{i,j-1})\geq l(C_i)-6$.

We prove this by contradiction, supposing that $l(\Pi_{i,j-1})< l(C_i)-6$. First we have $i\neq j-1$ and thus $i<j-1$ as $i<j$.
Let $r=K(C_i, \Pi_{i+1,j-1})$, then we have $$l(C_i)-6> l(\Pi_{i,j-1})\geq l(C_i)+l(\Pi_{i+1,j-1})-2r-1,$$ from which we find $ l(\Pi_{i+1,j-1})<2r-5.$ Since $\mathcal L_i$ is unaltered in $L_{i,j-1}$ by the induction assumption, and it is altered in $L_{i,j}$, we know that by \ref{Item:ClaimRBlgammiUnaltered}, $RB_2(\mu_i)=RB_{l(\gamma_i)+2}(C_i)$ is unaltered in $L_{i,j-1}$, hence unaltered in $P(C_i,\Pi_{i+1,j-1})$ by Corollary \ref{C:UnalteredInLFPToUnalteredInProduct}. Since exactly $r$ components of $C_i$ are canceled in $P(C_i,\Pi_{i+1,j-1})=P(\lambda_i\mu_i\rho_i, \Pi_{i+1,j-1})$,
we have $K(\rho_i,\Pi_{i+1,j-1})\geq r-1$. It then follows that $K(\gamma_i,\Pi_{i+1,j-1})\geq r-2$ by  Lemma \ref{L:AlmostReducedAndReducedForm}\eqref{Item:AlmostReducedAndReducedForm-Similarity}. 
Using this, the inequality $l(\Pi_{i+1,j-1})<2r-5$ and \ref{Claim:CancellationOfGammaiAndPiCan'tBeTwomuch}, we reach a contradiction:
\begin{align*}
  l(\gamma_i)-1\leq  l(\gamma_i\Pi_{i+1,j-1})\leq l(\gamma_i)+l(\Pi_{i+1,j-1})-2(r-2)< l(\gamma_i)-1.
\end{align*}  \qed
\end{enumerate} 
Our second substep is to complete the proof of Claim \ref{main_claim}, using  \ref{Item:ClaimE}, \ref{Claim:NonCancellationAmongComponentsFromC}, \ref{Item:ClaimRBlgammiUnaltered} and \ref{Item:ClaimD}. We only prove $\mathcal L_i$ is unaltered in $L_{i,j}$, as the ``right version" is proved similarly. \looseness=-1

 Suppose that  $\mathcal L_i$ is altered in $L_{i,j}$.  By the induction assumption $\mathcal L_i$ is unaltered in $L_{i,j-1}$, so the component $RB_{l(\gamma_{i})+2}(C_{i})$ is unaltered in $L_{i,j-1}$ by \ref{Item:ClaimRBlgammiUnaltered}. It follows that $B_i:=B_{l(\gamma_i)+x+1}(C_i)$ is unaltered in $L_{i,j-1}$ as well, because $l(C_i)\geq 2l(\gamma_i)+2s-1$ by Lemma \ref{L:AlmostReducedAndReducedForm}\eqref{Item:AlmostReducedAndReducedForm-length} and $2s-1\geq x+3$. 

 Now, since $\mathcal L_i$ is unaltered in $L_{i,j-1}$ but altered in $L_{i,j}$, it is therefore altered in $P(\Pi_{i,j-1}, C_j)$, implying that $B_i$ is canceled in $P(\Pi_{i,j-1}, C_j)$. But since $B_i$ is a component from $\mu_i$, it cannot be canceled by $\mu_j$ in $P(\Pi_{i,j-1}, C_j)$ by \ref{Claim:NonCancellationAmongComponentsFromC}.  Combining \ref{Item:ClaimD} and \ref{Item:ClaimE}, we know that $B_i$ is not canceled by $\lambda_j$ either. Therefore $B_i$ is canceled by $\rho_j$ in $P(\Pi_{i,j-1}, C_j)$.

Because $B_i$ is canceled by $\rho_j$ in $P(\Pi_{i,j-1}, C_j)$, we know $\mathcal R_j$ is altered in $P(\Pi_{i,j-1}, C_j)$. So $\mathcal R_j$ is altered in $R_{i,j}$ by (the right version of) Corollary \ref{C:UnalteredInLFPToUnalteredInProduct}.  
Starting from that $\mathcal R_j$ is altered in $R_{i,j}$, we make an argument similar to the previous two paragraphs but using the right-sided version of the claims whenever necessary, to conclude that $D_j:=RB_{l(\gamma_j)+x+1}(C_j)$ is unaltered in $R_{i+1,j}$ and $D_j$ 
is canceled by  $\lambda_i$ in the product $P(C_i, \Pi_{i+1,j})$.

Now, as $D_j$ is unaltered in $R_{i+1, j}$, it is unaltered in $P(\Pi_{i+1,j-1}, C_j)$ by (the right version of) Corollary \ref{C:UnalteredInLFPToUnalteredInProduct}. Moreover, $\lambda_i$ is unaltered in $P(C_i, \Pi_{i+1,j-1})$ by the induction assumption and Corollary \ref{C:UnalteredInLFPToUnalteredInProduct}. We can therefore apply Lemma \ref{L:AssociationOfCancellationOfAPair}, concluding that since $D_j$ is canceled by $\lambda_i$ in the product $P(C_i, \Pi_{i+1,j}) = P(C_i, \Pi_{i+1,j-1}C_j)$, 
$D_j$ is canceled by $\lambda_i$ in $P(C_i\Pi_{i+1,j-1},C_j)=P(\Pi_{i,j-1}, C_j)$.

Hence in the product $P(\Pi_{i,j-1}, C_j)$, $B_i$ (resp. $D_j$) is canceled by $\rho_j$ (resp. $\lambda_i$). We now show this leads to a contradiction. Since $B_i=B_{l(\gamma_i)+x+1}(C_i)$ is unaltered in $L_{i,j-1}$, $\lambda_i$ is a left factor of $C_i$ and $l(\lambda_i)=l(\gamma_i)$,  
we can write $\Pi_{i, j-1} = \lambda_i\sigma_1 B_i\tau_1$ as a reduced product, where $\sigma_1,\tau_1\in A$. Similarly, as $D_j=RB_{l(\gamma_j)+x+1}(C_j)$ and $\rho_j$ is a right factor of $C_j$ with $l(\rho_j)=l(\gamma_j)$, we can write $C_j$ as a reduced product $\sigma_2D_j\tau_2\rho_j$, where $\sigma_2,\tau_2\in A$.       Then  the cancellation of $B_i$ by $\rho_j$ in the product
    $\Pi_{i,j-1}C_j=(\lambda_i\sigma_1 B_i\tau_1)(\sigma_2D_j\tau_2\rho_j)$ implies that $D_j$ is canceled by $\tau_1$, contradicting that $D_j$ is canceled by $\lambda_i$.
\end{proof}

\bibliography{gen_torsionbib}

\bibliographystyle{plain}

\end{document}